\documentclass{amsart}
\usepackage{verbatim}
\usepackage[textsize=scriptsize]{todonotes}
\usepackage{etoolbox}
\usepackage{stackrel}
\usepackage{longtable}
\usepackage{xcolor}

\newtoggle{final}

\usepackage{tikz}
\usetikzlibrary{matrix,arrows}
\usepackage{tikz-cd}

\toggletrue{final}

\usepackage{etoolbox}
\usepackage[margin=1in]{geometry}
\usepackage{amsmath}
\usepackage{amssymb}
\usepackage{amsthm}
\usepackage{amscd}
\usepackage{enumerate}
\usepackage[pdfusetitle,unicode,hidelinks]{hyperref}
\usepackage{bbm}
\usepackage{etoolbox}

\usepackage[utf8]{inputenc}
\usepackage[T1]{fontenc}

\newcommand{\tomemail}{\href{mailto:tom.bachmann@zoho.com}{tom.bachmann@zoho.com}}

\newtheorem{proposition}{Proposition}
\newtheorem{corollary}[proposition]{Corollary}
\newtheorem{lemma}[proposition]{Lemma}
\newtheorem{theorem}[proposition]{Theorem}

\newtheorem*{conjecture*}{Conjecture}
\newtheorem*{theorem*}{Theorem}
\newtheorem*{corollary*}{Corollary}
\newtheorem*{proposition*}{Proposition}
\newtheorem*{lemma*}{Lemma}
\theoremstyle{definition}
\newtheorem{definition}[proposition]{Definition}

\newtheorem*{definition*}{Definition}
\newtheorem*{construction*}{Construction}
\theoremstyle{remark}
\newtheorem{remark}[proposition]{Remark}
\newtheorem*{remark*}{Remark}
\newtheorem{question}[proposition]{Question}
\newtheorem{example}[proposition]{Example}
\newtheorem*{example*}{Example}

\newcommand{\id}{\operatorname{id}}
\newcommand{\Z}{\mathbb{Z}}

\newcommand{\F}{\mathbb{F}}

\let\scr=\mathcal
\let\bb=\mathbb
\newcommand{\Gm}{{\mathbb{G}_m}}
\newcommand{\Gmp}[1]{{\mathbb{G}_m^{\wedge #1}}}
\def\A{\bb A}

\newcommand{\1}{\mathbbm{1}}

\newcommand{\SH}{\mathcal{SH}}

\DeclareMathOperator*{\colim}{colim}

\let\lim=\relax
\DeclareMathOperator*{\lim}{lim}
\def\Map{\mathrm{Map}}

\def\map{\mathrm{map}}
\def\CAlg{\mathrm{CAlg}}
\def\CMon{\mathrm{CMon}}
\def\NAlg{\mathrm{NAlg}}

\def\PSh{\mathcal{P}}

\def\Span{\mathrm{Span}}
\def\Cat{\mathcal{C}\mathrm{at}{}}
\def\Spc{\mathcal{S}\mathrm{pc}{}}
\def\Fin{\cat F\mathrm{in}}
\def\Fun{\mathrm{Fun}}

\newcommand{\Spec}{\mathrm{Spec}}

\newcommand{\wequi}{\simeq}
\newcommand{\Mod}{\text{-}\mathcal{M}\mathrm{od}}
\def\adj{\rightleftarrows}

\DeclareRobustCommand{\ul}{\underline}

\newcommand{\Hom}{\operatorname{Hom}}

\def\op{\mathrm{op}}

\let\cat=\mathrm
\def\Sm{{\cat{S}\mathrm{m}}}
\def\Sch{\cat{S}\mathrm{ch}{}}
\def\FEt{\mathrm{FEt}{}}

\def\Nis{\mathrm{Nis}}

\def\mot{\mathrm{mot}}

\def\ph{\mathord-}

\theoremstyle{definition}

\newtheorem{notation}[proposition]{Notation}

\numberwithin{proposition}{section}

\iftoggle{final} {
\renewcommand{\todo}[1]{}
\newcommand{\NB}[1]{}
\newcommand{\tdn}[1]{}
}{ % else

\newcommand{\NB}[1]{\todo[color=gray!40]{#1}}
\newcommand{\tdn}[1]{\todo[color=green]{#1}}
}

\newcommand{\sslash}{/\mkern-6mu/}
\def\E{\mathbb{E}}

\newcommand{\Cor}{\mathrm{Cor}}
\def\fet{\mathrm{f\acute et}}
\def\tpffqf{\mathrm{tpffqf}}
\def\tpet{\mathrm{tp\acute et}}
\def\all{\mathrm{all}}
\def\D{\mathrm{D}}
\newcommand{\Stk}{\cat{S}\mathrm{tk}}
\def\NMon{\mathrm{NMon}}
\def\NSym{\mathrm{NSym}}
\def\cof{\mathrm{cof}}
\def\Ex{Ex}
\def\free{\mathrm{free}}
\def\prop{\mathrm{prop}}
\newcommand{\repr}{\mathrm{repr}}

\newcommand{\Shv}{\cat{S}\mathrm{hv}}

\newcommand{\SmQP}{\mathrm{SmQP}}
\newcommand{\SmQA}{\mathrm{SmQA}}

\newcommand{\qaff}{\mathrm{qaff}}
\newcommand{\htp}{\mathrm{htp}}

\newcommand{\Sph}{\mathrm{Sph}}
\newcommand{\Orb}{\mathrm{Orb}}
\newcommand{\Sp}{\mathrm{Sp}}
\newcommand{\SHS}{\SH^{S^1}\!}

\newcommand{\sift}{\mathrm{sift}}

\usepackage{appendix}

\title{Motivic spectral Mackey functors}
\date{\today}

\author{Tom Bachmann}
\address{Department of Mathematics, University of Oslo, Norway}
\email{\tomemail}

\begin{document}

\maketitle

\begin{abstract}
We show that if $G$ is a finite constant group acting on a scheme $X$ such that $|G| \in \scr O^\times_X$, then the $G$-equivariant motivic stable homotopy category of $X$ is equivalent to the stabilization of the category of motivic $G$-spaces with finite étale transfers over $X$ at the \emph{trivial} representation sphere.
Along the way we obtain several results of independent interest, among them: we construct and study norms in the motivic homotopy theory of stacks, and we extend the homotopy $t$-structure to DM-stacks and establish some favorable properties.
\end{abstract}

\tableofcontents

\section{Introduction}
\subsection{A motivic Guillou--May theorem}
Let $G$ be a finite group and $X$ a scheme with a $G$-action.
Assume that $G$ is tame (relative to $X$), i.e. $|G| \in \scr O_X^\times$.
Following \cite{hoyois-equivariant} and \cite{khan2021generalized}, we can study $G$-equivariant stable motivic homotopy theory over $X$, or equivalently stable motivic homotopy theory over the stack $X \sslash G$ (also sometimes denoted $[X/G]$); this is a presentably symmetric monoidal stable $\infty$-category which we denote $\SH(X \sslash G)$.
This category is obtained by starting with presheaves on $G$-equivariant smooth $X$-schemes, inverting $\A^1$- and Nisnevich equivalences, and stabilizing with respect to the representation spheres.

Denote by $\Cor^\fet(\Sm_{X \sslash G})$ the $(2,1)$-category with the same objects as $\Sm_{X \sslash G}$, and morphisms given by the spans \[ Y_1 \sslash G \xleftarrow{p} Z \sslash G \to Y_2 \sslash G, \] where the map $p$ is required to be finite étale (i.e. the corresponding map $Z \to Y_1$ of schemes is finite étale).
We define the category of \emph{motivic spaces with finite étale transfers} to be \[ \Spc^\fet(X \sslash G) := L_\mot \PSh_\Sigma(\Cor^\fet(\Sm_{X \sslash G})). \]
It is presentably symmetric monoidal, with the monoidal structure induced by just the product of stacks (over $X \sslash G$).

We can now state a motivic analog of a theorem of Guillou--May \cite{guillou2011models} in classical genuine $G$-equivariant stable homotopy theory.
\begin{theorem}[see Theorem \ref{thm:main}] \label{thm:intro-main}
Let $G$ be a finite discrete group acting on a scheme $X$, with $|G| \in \scr O_X^\times$.
Then there is an equivalence of symmetric monoidal $\infty$-categories \[ \SH(X \sslash G) \wequi \Spc^\fet(X \sslash G)[T^{-1}]. \]
Here $T$ denotes (the free motivic space with finite étale transfers on) the Thom space of the \emph{trivial} representation sphere.
\end{theorem}

Recall that $\SH(X \sslash G)$ was constructed by inverting \emph{all} representation spheres, or equivalently the regular representation, whereas the right hand side above is obtained by only inverting the trivial representation sphere.
In other words, one of the most surprising aspects of the above result is that the image of $T^\rho$ in $\Spc^\fet(X \sslash G)[T^{-1}]$ is already invertible.

\begin{remark}
Let $\scr S$ be a tame DM-stack with affine diagonal.
Then $\scr S$ admits a quasi-affine Nisnevich covering by stacks of the form $X \sslash G$, for varying and not necessarily constant finite étale groups $G$.
If all $G$ occurring are constant, then Theorem \ref{thm:intro-main} remains true for $\scr S$ in place of $X \sslash G$.
More generally Theorem \ref{thm:intro-main} for all tame DM-stacks would follow from what I call the \emph{Adams hypothesis}; see \S\ref{sec:intro-questions} for more about this.
\end{remark}
In fact, most results in this article are proved for classes of stacks more general than global quotient stacks.
For the purposes of this introduction, we will forgo this additional generality and only explain results in terms of global quotients by finite constant groups.

\subsection{Related results}
In the course of proving Theorem \ref{thm:intro-main}, we establish some related results of independent interest.
\subsubsection{Generation results}
One surprising aspect of Theorem \ref{thm:intro-main} is that objects of the form $\Sigma^\infty_+ \scr Y \wedge \Gmp{n}$, for $\scr Y \in \Sm_{X \sslash G}$ and $n \in \Z$ (and $\Gm$ carrying the trivial action) generate the right hand side, and thus the left hand side.
In fact more is true.
In order to state the next result, we recall some more terminology.
A stack is called a \emph{gerbe} if its inertia (over $\Z$) is flat; for a quotient stack of the form $X \sslash G$ this means that the scheme of inertia subgroups \[ I_X \subset G \times X \to X \] is flat.
A more appealing description may be as follows (see Lemma \ref{lemm:gerbes-admissible-auto}): $X \sslash G$ is a gerbe if and only if $X$ is a disjoint union of $G$-schemes of the form $Y \times_H G$, where $H$ is a subgroup of $G$, $Y$ is an $H$-scheme and there exists a normal subgroup $N \subset H$ such that $N$ acts trivially on $Y$ and $H/N$ acts freely on $Y$.

\begin{proposition}[see Proposition \ref{prop:gens}] \label{prop:intro-gens}
The category $\SH(X \sslash G)$ is generated by objects of the form $\Sigma^\infty_+ \scr Y \wedge \Gmp{n}$, for $\scr Y \in \Sm_{X \sslash G}$ and $n \in \Z$.
If $X \sslash G$ is a gerbe, we may assume that $\scr Y$ also is.
\end{proposition}

\subsubsection{Conservative families of fixed point functors}
In classical genuine equivariant homotopy theory, for every subgroup $H \subset G$ there is a \emph{fixed point functor} $(\ph)^H: \SH(BG) \to \SH$, and together these form a conservative collection.
The same is not true in equivariant motivic homotopy theory, essentially because $G$-schemes of the form $G/H \times Y$, where $Y$ carries the trivial action, do not generated $\Sm_{B \sslash G}$.
Proposition \ref{prop:intro-gens} shows a way out.
Let $\scr T \in \Sm_{B \sslash G}$ be a gerbe, where $B$ carries the trivial action.
Then there is an initial morphism $\scr T \to T$, where $T$ is an algebraic space; roughly speaking we have $\scr T = (X \times_H G) \sslash G$ where $N \subset H$ acts trivially and $H/N$ acts freely, and then $T = X/(H/N)$.
Consider the composite \[ (\ph)^{\scr T}: \SH(B \sslash G) \to \SH(\scr T) \to \SH(T), \] where the first functor is pullback along $\scr T \to B \sslash G$ and the second is right adjoint to pullback along $\scr T \to T$.
We call this the \emph{$\scr T$-fixed points functor}.

\begin{example}
Let $\scr T = (G/H \times B)\sslash G$, so that $T = B$.
Then the composite \[ \Sm_{B \sslash G} \to \Sm_{\scr T} \to \Sm_B \] (scheme-level version of $(\ph)^{\scr T}$) sends $X \sslash G$ to $X^H$.
\end{example}

Proposition \ref{prop:intro-gens} implies that the left adjoints of all the $\scr T$-fixed points functors have dense image.
Thus we obtain the following result.
\begin{corollary} \label{cor:into-conservative}
The family of $\scr T$-fixed point functors \[\{ (\ph)^{\scr T}: \SH(B \sslash G) \to \SH(T) \}_{\scr T}\] (indexed by smooth stacks over $B \sslash G$ which are gerbes) is jointly conservative.
\end{corollary}

\begin{example}
Let $G = C_2$.
Since there are no non-trivial proper subgroups, all gerbes arise as disjoint unions of stacks of the form $X \sslash C_2$, where the action on $X$ is either trivial or free.
For $X=B$ we obtain the usual fixed point functor \[ (\ph)^{C_2}: \SH(B \sslash C_2) \to \SH(B). \]
On the other hand if $X$ carries a free $C_2$-action with quotient $Y$ (i.e. $X \to Y$ is a $C_2$-torsor), we obtain a more exotic fixed point functor \[ (\ph)^{X}: \SH(B \sslash C_2) \to \SH(Y). \]
It is an easy consequence of Corollary \ref{cor:into-conservative} that the functor $(\ph)^{C_2}$ together with the functors $(\ph)^X$ for $X \to Y$ a $C_2$-torsor (and $Y \in \Sm_B$) form a conservative collection.
\begin{comment}
To understand the meaning of $(\ph)^X$, consider the scheme-level analog \[ (\ph)^X: \Sm_{B \sslash C_2} \to \Sm_{X \sslash C_2} \to \Sm_{Y}. \]
Since the action on $X$ is free, we get $X \sslash C_2 \wequi Y$ as stacks, and so the second functor is an equivalence.
It follows that for $Z \sslash C_2 \in \Sm_{B \sslash C_2}$ we get \[ (Z \sslash C_2)^X = Z \times_{C_2} X \in \Sm_Y. \]
For example if $X=C_2$ then $(Z \sslash C_2)^X = Z$ and so $(\ph)^X$ is the forgetful functor.
\end{comment}
\end{example}

\subsubsection{The homotopy $t$-structure}
Another use of a small family of generators is to define a $t$-structure.
Thus denote by \[ \SH(X \sslash G)_{\ge 0} \] the subcategory generated under colimits and extensions by objects of the form $\Sigma^\infty_+ \scr Y \wedge \Gmp{n}$, where $\scr Y \in \Sm_{X \sslash G}$ and $n \in \Z$.
This defines the non-negative part of a $t$-structure \cite[Proposition 1.4.4.11]{lurie-ha}.
It has the following favorable properties:
\begin{itemize}
\item If $f: \scr X \to \scr Y$ is a morphism of stacks as above (but not necessarily representable), then $f^*: \SH(\scr Y) \to \SH(\scr X)$ is right $t$-exact.
  If $f$ is smoothly representable, then $f^*$ is $t$-exact.
\item The $\scr T$-fixed points functors $(\ph)^{\scr T}: \SH(B \sslash G) \to \SH(T)$ are $t$-exact (Proposition \ref{prop:fixed-t-exact}).
\item If $X$ is finite dimensional, then the $t$-structure on $\SH(X \sslash G)$ is non-degenerate (Proposition \ref{prop:homotopy-t-nondegenerate}).
\end{itemize}

\subsection{Sketch proof of the motivic Guillou--May theorem} \label{subsub:proof-sketch}
In order to prove Theorem \ref{thm:intro-main}, we establish the following facts:
\begin{enumerate}
\item The generation by gerbes result, i.e. Proposition \ref{prop:intro-gens}.
\item The canonical functor $\SH(X \sslash G) \to \Spc^\fet(X \sslash G)[T^{-\rho}]$ is an equivalence.
\item The canonical functor $\Spc^\fet(B \sslash G)[T^{-1}] \to \Spc^\fet(B \sslash G)[T^{-\rho}]$ is fully faithful.
\end{enumerate}
The desired result is deduced as follows.
The functor \[ \Spc^\fet(B \sslash G)[T^{-1}] \to \Spc^\fet(B \sslash G)[T^{-\rho}] \stackrel{(2)}{\wequi} \SH(B \sslash G) \] is fully faithful by (3) and essentially surjective by (1), hence an equivalence.
Thus $T^\rho$ is invertible in $\Spc^\fet(B \sslash G)[T^{-1}]$ and so in $\Spc^\fet(X \sslash G)[T^{-1}]$.
Hence \[ \Spc^\fet(X \sslash G)[T^{-1}] \wequi \Spc^\fet(X \sslash G)[T^{-\rho}] \stackrel{(2)}{\wequi} \SH(X \sslash G). \]

The proof of (1) uses filtration by isotropy arguments, imported into motivic homotopy theory by Gepner--Heller.
In fact traces of such generation by gerbes results can be found in their work, for example in \cite[proof of Proposition 3.27]{gepner-heller}.

The status of (2) is a bit curious.
I expect that this result should be true in great generality, and for essentially ``formal'' reasons, generalizing the fact that passage to commutative monoids is a smashing localization of presentable categories \cite{gepner2016universality}.
Unfortunately all my attempts at giving a proof along these lines got bogged down in cumbersome technicalities, and ultimately seemed not worth the effort.
Instead I present an amusing direct proof, employing good properties of the homotopy $t$-structure (which seemed interesting to develop for their own sake).
The disadvantage of this approach is that it uses the tom Dieck splitting theorem, and hence (for now) only works in the context of finite discrete groups.

The proof of (3) uses the tom Dieck splitting theorem in an essential way, and this was my original insight.
Write \[ \sigma^\infty: \Spc^\fet(B \sslash G)[T^{-1}] \adj \Spc^\fet(B \sslash G)[T^{-\rho}]: \omega^\infty \] for the canonical adjunction.
We wish to prove that for $E \in \Spc^\fet(B \sslash G)[T^{-1}]$ we have $E \wequi \omega^\infty \sigma^\infty E$.
To do so, as discussed above, it suffices to prove that $E^{\scr T} \wequi (\omega^\infty \sigma^\infty E)^{\scr T} \in \SH(T)$ for every gerbe $\scr T \in \Sm_{B \sslash G}$.
Since $T=T^\rho$ over algebraic spaces, it is easy to see that $(\omega^\infty \sigma^\infty E)^{\scr T} \wequi (\sigma^\infty E)^{\scr T}$, and so the main difficulty is to show that $E^{\scr T} \wequi (\sigma^\infty E)^{\scr T}$.
Since on the left hand side we have not inverted $T^\rho$, the functor $(\ph)^{\scr T}$ is accessible essentially by direct computation.
In the case of suspension spectra (which suffice for our purposes), one finds a splitting of $(\ph)^{\scr T}$ into a sum over conjugacy classes of subgroups of certain homotopy orbit spectra---precisely the expected form of the tom Dieck splitting.
On the right hand side one finds the same expression, by the main result of \cite{gepner-heller}.
This concludes the proof.

Actually making precise the identification of the two tom Dieck splittings turns out to be rather cumbersome, so in the current writeup the proof is reformulated in terms of Adams isomorphisms and abstract forms of filtration by isotropy.
But the above sketch is how I discovered it.

\begin{remark}
The same argument can be used to give a proof of the (original) theorem of Guillou--May, where the only non-trivial input is the Adams isomorphism.
\end{remark}

\subsection{Open questions} \label{sec:intro-questions}
Examining the ingredients of the proof, a number of natural questions arise.
I have no intention of working on them.

\subsubsection{}
As pointed out above, the main non-formal ingredient of the motivic Guillou--May theorem is the Adams isomorphism, as established for finite groups in \cite{gepner-heller}.
Given a $G$-scheme $X$ and a normal subgroup $N$ acting trivially, the Adams transformation studied by Gepner--Heller is in a natural way associated with the morphism of stacks $X \sslash G \to X \sslash (G/N)$.
More generally, given any proper étale morphism $p: \scr X \to \scr Y$ of reasonable stacks, there is a natural generalization of the Adams transformation (see \S\ref{sec:ambidex} and \S\ref{subsec:isotropy-spec}).
\begin{question} \label{qst:adams}
If $p: \scr X \to \scr Y$ is a proper étale morphism of linearly scalloped stacks, is the Adams transformation an equivalence?
\end{question}
The special case where $\scr X$ and $\scr Y$ are DM-stacks is called the \emph{Adams hypothesis} in this article (see Definition \ref{def:adams-hyp}).

\subsubsection{}
Let $\scr X$ be a stack.
We have a canonical functor \[ \SH(\scr X) \to \Spc^\fet(\scr X) \otimes_{\Spc(\scr X)} \SH(\scr X). \]
If $\scr X$ is a tame DM-stack and either $\scr X$ which is locally constant, or the Adams hypothesis holds, then this functor is an equivalence.
This is immediate from the motivic Guillou--May theorem in the form of Theorem \ref{thm:main}, but I believe that it should hold much more generally.
\begin{question}
Is the functor $\SH(\scr X) \to \Spc^\fet(\scr X) \otimes_{\Spc(\scr X)} \SH(\scr X)$ an equivalence for all (reasonable) stacks $\scr X$?
\end{question}

\subsubsection{}
Assuming that this is true, for any morphism of stacks $\scr X \to \scr Y$, we obtain a functor \begin{equation}\label{eq:qst} \Spc^\fet(\scr X) \otimes_{\Spc(\scr Y)} \SH(\scr Y) \to \Spc^\fet(\scr X) \otimes_{\Spc(\scr X)} \SH(\scr X) \wequi \SH(\scr X). \end{equation}
The motivic Guillou--May theorem implies that this is an equivalence for morphisms of the form $X \sslash G \to X \sslash (G/N)$, and the Adams hypothesis implies this for any proper étale morphism of DM-stacks.
One may show that (an affirmative answer to) Question \ref{qst:adams} implies that \eqref{eq:qst} is an equivalence for any proper étale morphism $p$.
It is also easy to see that this is true for any representable morphism $p$.
\begin{question}
For which morphisms $\scr X \to \scr Y$ is  \eqref{eq:qst} an equivalence?
\end{question}
For example, what about $\Spec(\F_p) \sslash \mu_p \to \Spec(\F_p)$?

\subsection{Organization}
We now give a quick rundown of the contents of the sections in order.

\subsubsection{Preliminaries (\S\ref{sec:preliminaries})}
In this section we establish some preliminary results.
There are two main topics.
Firstly, we collect some results about morphisms of stacks with flat inertia, for which we could not find references.
Secondly, we establish our conventions for unstable motivic homotopy theory.
In particular we give a definition of $\Spc(\scr X)$ for any stack $\scr X$, which coincides with the construction of \cite{hoyois-equivariant} when the latter is defined, and which coincides with the construction of \cite{khan2021generalized} when the latter is defined and $\scr X$ has affine diagonal.
This allows us in the sequel to treat all of these cases at the same time.

\subsubsection{Norms (\S\ref{sec:norms})}
In this section we extend some of the constructions and results of \cite{bachmann-norms} to stacks.

In \S\ref{subsec:normed-oo-cat} we give a formula for ``free normed objects'' in a very general situation.

Next we construct norm functors for stacks.
For a tame proper étale morphism $p: \scr X \to \scr Y$ of stacks, the functor $p^*: \SmQA_{\scr Y} \to \SmQA_{\scr X}$ admits a right adjoint $p_\otimes$.
(By decomposing $p$ fppf locally on $Y$ into a projection $B \sslash G \to B$ followed by a finite étale morphism, we see that $p_\otimes$ arises as a kind of combination of fixed points and Weil restriction.)
\begin{comment}
Let us describe this right adjoint in more detail.
We may factor $p$ as $\scr X \xrightarrow{q} \scr X' \xrightarrow{r} \scr Y$, where $r$ is a \emph{finite} (so in particular affine) étale morphism and $q$ is a gerbe with tame finite étale inertia.
In other words there exists an fppf covering $B \to \scr X'$ (with $B$ an affine scheme) and a tame finite étale group $G$ over $B$ such that $B \times_{\scr X'} \scr X \wequi B \sslash G$.
The functor $q_\otimes$ is compatible with base change, and the functor $q'_\otimes: \SmQA_{B\sslash G} \to \SmQA_B$ just sends a smooth quasi-affine $B$-scheme $X$ with a $G$-action to the fixed points $X^G$.
Similarly the functor $r_\otimes$ is compatible with base change, and since $r$ is finite étale, fppf locally on $\scr Y$ we just recover the functor $r'_\otimes: \SmQA_{X'} \to \SmQA_{Y}$ of Weil restriction for schemes.
\end{comment} 
We show that this construction extends to $\Spc(\ph)$ (see \S\ref{subsec:norms-for-spaces}), and also to $\Spc(\ph)_*$, $\Spc^\fet(\ph)$ (see \S\ref{subsec:spcfet}), and even to $\SH(\ph)$ provided that the relevant stacks are scalloped (see \S\ref{subsec:stab-norms}).

In the final subsection \S\ref{subsec:free-normed} we analyze the suspension spectra of free normed monoids.
Thus let $X \in \Spc(\scr S)_*$ and write $U_*F_*X$ for the effect of the composite $\Spc(\scr S)_* \to \Spc^\fet(\scr S) \to \Spc(\scr S)_*$.
Generalizing well-known results in classical homotopy theory (see e.g. \cite{cohen1980stable}) we show that (see Corollary \ref{cor:SH-sph-normed}) \[ \Sigma^\infty U_*F_* X \wequi \bigvee_{n \ge 1} \Sigma^\infty D_n(X), \] where $D_n: \Spc(\scr X)_* \to \Spc(\scr X)_*$ is a version of the motivic extended power functor.

\subsubsection{Complements on motivic homotopy theory of stacks (\S\ref{sec:complements})}
This section is largely concerned with generalizing some of the complementary constructions and results of \cite{gepner-heller}, many of which are related to filtration by isotropy.

One construction of this section worth mentioning are \emph{orbit stacks}.
If $p: \scr Y \to \scr S$ is proper étale, for example $\scr Y = B \sslash G$ and $\scr S=B$, we construct the following diagram of stacks, where on the right hand side we display the special case of the example
\begin{equation*} \label{eq:pp'-diamond}
\begin{tikzcd}
& \widetilde \Orb_p \ar[dl, "p_1"] \ar[dr, "p_2"] & & & & \coprod_{(H)} B \sslash NH \ar[dl, "p_1"] \ar[dr, "p_2"] \\
\scr X \ar[dr,"p"] & & \Orb_p \ar[dl, "p'"] &  & B \sslash G \ar[dr,"p"] & & \coprod_{(H) \subset G} B \sslash WH \ar[dl, "p'"] \\
 & \scr S & & & & B.
\end{tikzcd}
\end{equation*}
Here the disjoint unions are over conjugacy classes of subgroups, $NH$ denotes the normalizer of $H$, and $WH = NH/H$ the Weil group.
We arrived at this construction originally in order to establish some technical results.
However it now seems that orbit stacks can be used to generalize a number of classical results, like the tom Dieck splitting theorem and the geometric fixed point functors.
We use it also as a convenient tool for defining filtration by isotropy in a general context.

One of the main results of this section is the generation by gerbes; see Corollary \ref{cor:gerbes-gen}.

\subsubsection{Ambidexterity (\S\ref{sec:ambidex})}
Ambidexterity refers to the following phenomenon: given an $\infty$-category $\scr X$ and a functor $\scr C: \scr X^\op \to \Cat_\infty$, for certain morphisms $p: X \to Y \in \scr X$ it makes sense to ask that $p^*: \scr C(Y) \to \scr C(X)$ admits a left adjoint $p_!$ and a right adjoint $p_*$, and that there is a canonical equivalence (``Wirthmüller isomorphism'') $p_! \wequi p_*$.
Essentially there is an inductive construction of a transformation $p_! \to p_*$ which uses the iterated diagonals of $p$.
There are various reference for these ideas, see e.g. \cite{hopkins2013ambidexterity,carmeli2018ambidexterity}.

We extend this as follows.
Without assuming that $p_!$ exists, let $A \in \scr C(X)$.
One can ask if the functor $\scr C(Y) \to \Spc, B \mapsto \Map(A, p^*B)$ is representable; if so we denote the representing object by $p_!(A)$ and we shall say that $p_!(A)$ exists.
A similar definition applies to $p_*$.

Our main observation in this section is that instead of asking ``is the morphism $p$ ambidextrous for $\scr C$?'', it makes sense to ask ``for which objects $A \in \scr C(X)$ do $p_!(A)$ and $p_*(A)$ both exist and are naturally isomorphic?''.
We call such objects \emph{$p$-ambidextrous}.
This is relevant in our situation, since the \emph{Adams isomorphism} holds only for free spectra and not all of them; this is precisely an instance of a partially defined left adjoint agreeing with a right adjoint.

The main technical result (Lemma \ref{lemm:norm-nat}) is that under appropriate assumptions, the formation of Wirthmüller isomorphisms is natural in $\scr C$.

\subsubsection{More about spaces with finite étale transfers (\S\ref{sec:more-fet})}
In this section we study the functor $\Spc^\fet(\ph)$ in more detail.
We first prove that it satisfies quasi-affine Nisnevich descent and continuity (\S\ref{subsec:descent-ctty}).
Then we show that it interacts well with isotropy specification (\S\ref{subsec:spcfet-isotropy-spec}), homotopy orbit functors (\S\ref{subsec:spcfet-quotients}) and geometric fixed point functors (\S\ref{subsec:spcfet-geom-fixed}).

We also prove that it satisfies ambidexterity for all tame proper étale morphisms (see \S\ref{subsec:spcfet-ambidex}): if $p: \scr X \to \scr Y$ is tame proper étale, then all free objects of $\Spc^\fet(\scr X)$ are $p$-ambidextrous.
This is a relatively straightforward ab initio computation, in particular it does not use \cite{gepner-heller}.

\subsubsection{DM-stacks (\S\ref{sec:DM})}
In this section we prove our main results.
After recalling some basic facts about DM-stacks, in \S\ref{subsec:gen-triv} we use filtration by isotropy to show that $\SH(\scr S)$ is generated by $\Gm$-desuspensions of suspension spectra, where $\Gm$ carries the trivial action.
Then in \S\ref{subsec:adm-gerbes} and \S\ref{subsec:orbit-stack-ind} we set up our main tool of proof: an abstract form of filtration by isotropy which we call \emph{orbit stack induction}.
Then in \S\ref{subsec:htpy-t} we establish basic properties of the homotopy $t$-structure on DM-stacks.
Finally in \S\ref{subsec:main} we establish our main theorem.

\subsection{Acknowledgements}
The results in this article are cobbled together from a variety of sources of inspiration.
\S\ref{sec:norms} on norms in the motivic homotopy theory of stacks is a relatively straightforward extensions of my paper with Marc Hoyois \cite{bachmann-norms}.
\S\ref{sec:complements} consists of more or less straightforward generalizations of standard methods in equivariant homotopy theory to the motivic homotopy theory of stacks, following closely \cite{gepner-heller} (which treats certain special cases).

The debt of this article to both \cite{hoyois-equivariant} (where Marc Hoyois establishes the foundations of stable motivic homotopy theory of global quotient stacks) and \cite{gepner-heller} (where David Gepner and Jeremiah Heller establish the Adams isomorphism for finite discrete tame groups) is hard to overstate.
We also immensely profited from the work of Khan--Ravi on extending motivic stable homotopy theory from quotient stacks to more general stacks \cite{khan2021generalized}.

During the very long gestation period I have had fruitful discussions about this work with many people, including Elden Elmanto, Jeremiah Heller, Marc Hoyois, Adeel A. Khan.
My special thanks goes to Markus Hausmann and David Gepner who (on independent occasions, several years apart) taught me that the tom Dieck splitting theorem is essentially equivalent to the spectral Mackey functor description of genuine $G$-spectra.

\section{Preliminaries} \label{sec:preliminaries}
\subsection{Conventions regarding stacks}
We follow the stacks project conventions \cite{stacks-project} regarding stacks and algebraic spaces, except where noted otherwise.
That is:
\begin{enumerate}
\item We choose a big fppf site $\Sch$.
  A morphism of sheaves $F \to G$ is called \emph{schematic} if for every $X \in \Sch$ and every morphism $X \to G$, the pullback $X \times_G F$ is a scheme.
  We call $F \to G$ \emph{schematically affine}, \emph{schematically étale}, etc. if it is schematic and $X \times_G F \to X$ is affine/étale/etc.
\item By an \emph{algebraic space} we mean a $0$-truncated sheaf $F$ admitting a schematically étale epimorphism $Y \to F$, where $Y$ is a scheme.
  We call a morphism $F \to G$ of algebraic spaces étale, smooth or flat if for every scheme $X \to G$ and every schematically étale morphism $Y \to X \times_G F$ with $Y$ a scheme, the composite $X \to Y$ is étale/smooth/flat.
  We call a morphism of sheaves $F \to G$ \emph{representable} if for every $X \in \Sch$ and every morphism $X \to G$, the pullback $X \times_G F$ is an algebraic space.
  Note that this is in contrast with the stacks project convention.
\item By an \emph{algebraic stack} we mean a $1$-truncated sheaf $F$ admitting a representably smooth morphism $Y \to F$, where $Y$ is an algebraic space (or equivalently scheme).
  A morphism of stacks $F \to G$ is called étale, smooth or flat if for every scheme $X \to G$ and every representably smooth morphism $Y \to X \times_G F$ with $Y$ a scheme, the composite $X \to Y$ is étale/smooth/flat.
\end{enumerate}
One may show that the diagonal of an algebraic space is schematic\NB{ref?}, and the diagonal of an algebraic stack is representable\NB{ref?}.
This implies that the category of algebraic spaces (respectively algebraic stacks) is closed under pullbacks.

By an affine morphism of stacks we mean a schematically affine one.
Recall that a scheme is called quasi-affine if it is isomorphic to a quasi-compact open subscheme of an affine scheme; a morphism of stacks is called quasi-affine if the pullback to any affine scheme is a quasi-affine scheme (that is, the morphism is ``schematically quasi-affine'').
Note that this property is fpqc local (on the target) \cite[Tag 02L7]{stacks-project}.
\begin{comment}
Note the following trivial observation.
\begin{lemma} \label{lemm:quasi-affine-fact}
Let $p: \scr X \to \scr Y$ be a quasi-affine morphism of stacks.
Then there exists a factorization $\scr X \to \scr X' \to \scr Y$ with $\scr X' \to \scr Y$ affine and $\scr X \to \scr X'$ a quasi-compact open immersion.
\end{lemma}
\begin{proof}
Take $\scr X'$ to be the relative spectrum of $p_* \scr O_X$ over $\scr Y$ \cite[Tag 01LQ]{stacks-project}.
Everything can be checked locally, so the claims follow from \cite[Tag 01SM]{stacks-project}.
\end{proof}
\begin{example} \label{ex:quaff-quotient} \NB{See e.g. \url{https://mathoverflow.net/a/267148}}
If a group $G$ acts on a quasi-affine scheme $X$, then there exists a quasi-affine scheme $X'$ with an action by $G$ and a $G$-invariant open immersion $X \to X'$.
(Apply Lemma \ref{lemm:quasi-affine-fact} to the affine morphism $X \sslash G \to * \sslash G$.)
In particular $X/G$ exists and is a quasi-affine scheme (namely an open subscheme of $X'/G$, which is exists and is affine \cite{TODO}).
\end{example}
\end{comment}

All our stacks are assumed qcqs and algebraic.

We call a stack with finite inertia \emph{tame} if it satisfies the equivalent conditions of \cite[Theorem 3.2]{abramovich2008tame} (i.e. all geometric stabilizers are linearly reductive).
We call a morphism of stacks $\scr X \to \scr Y$ \emph{tame} if for every affine scheme $A \to \scr Y$ the stack $\scr X \times_{\scr Y} A$ is tame.

We use the terminology on gerbes from \cite[Tag 06QB]{stacks-project}.
Thus a morphism of stacks is a gerbe if fppf locally on the target it takes the form $B \sslash G \to B$, and a stack $\scr X$ is called a gerbe if there is a gerbe $\scr X \to X$ with $X$ some algebraic space.

\begin{remark} \label{rmk:pet-gerbe-local-structure}
We often use the following fact: If $\scr X \to \scr S$ is a proper étale gerbe, then fppf locally on $\scr S$ it takes the form $B \sslash G \to B$, for finite groups $G$.
Indeed by \cite[Tag 06QH]{stacks-project} this is the case for $G$ a group scheme; but in order for $B \sslash G \to B$ to be proper étale $G \to B$ must be finite étale, and hence étale locally constant on $B$.
\end{remark}

\subsection{Flat inertia}
\subsubsection{First properties}
Recall that if $\scr X \to \scr Y$ is a morphism of stacks, then the \emph{relative inertia stack} is \[ I_{\scr X/\scr Y} = \scr X \times_{\scr X \times_{\scr Y} \scr X} \scr X, \] i.e. the pullback of the relative diagonal along itself.
Recall that $\scr X \to \scr Y$ is representable if and only if $I_{\scr X/\scr Y} \wequi \scr X$ \cite[Tag 04YY]{stacks-project}.

\begin{lemma} \label{lemm:inertia-composition}
Given $\scr X \to \scr Y \to \scr Z$ with $I_{\scr Y/\scr Z} \wequi \scr Y$ we have $I_{\scr X/\scr Y} \wequi I_{\scr X/\scr Z}$.
\end{lemma}
\begin{proof}
Follows from the cartesian square
\begin{equation*}
\begin{CD}
\scr X \times_{\scr X \times_{\scr Y} \scr X} \scr X @>>> \scr X \times_{\scr X \times_{\scr Z} \scr X} \scr X \\
@VVV @VVV \\
\scr Y @>>> \scr Y \times_{\scr Y \times_{\scr Z} \scr Y} \scr Y,
\end{CD}
\end{equation*}
in which the bottom horizontal map is an equivalence by assumption.
\end{proof}

\begin{corollary} \label{cor:bc-inertia}
Suppose given a cartesian square
\begin{equation*}
\begin{CD}
\scr X' @>>> \scr X \\
@VVV          @VVV  \\
\scr S' @>>> \scr S
\end{CD}
\end{equation*}
with $\scr X \to \scr S$ representable.
Then \[ I_{\scr X'/\scr S} \wequi I_{\scr S'/\scr S} \times_{\scr S'} \scr X'. \]
\end{corollary}
\begin{proof}
By Lemma \ref{lemm:inertia-composition} we have $I_{\scr X'/\scr S} \wequi I_{\scr X'/\scr X}$; now apply \cite[Tag 06PQ]{stacks-project}.
\end{proof}

\subsubsection{Quotients} \label{subsec:app-quotients}
\begin{proposition} \label{prop:relative-quotient-gerbe}
Let $\scr X \to \scr Y$ be a morphism of algebraic stacks such that $I_{\scr X/\scr Y} \to \scr X$ is flat and locally of finite presentation.
Then there exists an initial factorization \[ \scr X \to \repr_{\scr Y}(\scr X) \to \scr Y \] with $\repr_{\scr Y}(\scr X) \to \scr Y$ representable.
Moreover the following hold.
\begin{enumerate}
\item Formation of $\repr_{\scr Y}(\scr X)$ is compatible with arbitrary base change in $\scr Y$.
\item $\scr X \to \repr_{\scr Y}(\scr X)$ is a gerbe, so in particular smooth and surjective.
\item If $\scr X \to \scr Y$ is proper (respectively étale, (quasi-)separated, (locally) of finite type, (locally) of finite presentation, (locally) quasi-finite, flat, or quasi-compact) then so are $\scr X \to \repr_{\scr Y}(\scr X)$ and $\repr_{\scr Y}(\scr X) \to \scr Y$.
\item Let $\scr X' \to \scr X$ be a morphism such that $I_{\scr X'/\scr Y}$ is also flat and locally of finite presentation.
  Suppose that $I_{\scr X/\scr Y}$ is finite locally free.
  If $\scr X' \to \scr X$ is (quasi-)affine, then so is $\repr_{\scr Y}(\scr X') \to \repr_{\scr Y}(\scr X)$.
\end{enumerate}
\end{proposition}
\begin{proof} \tdn{revise this}
Let $\repr_{\scr Y}(\scr X)$ be the $0$-truncation of $\scr X$ in the $\infty$-topos of sheaves over $\scr Y$.
It is clear that this sheaf has the desired universal property.
Truncation commutes with pullbacks by \cite[Proposition 5.5.6.28]{lurie-htt} and universality of colimits in $\infty$-topoi.
Hence (1).
To prove that $\repr_{\scr Y}(\scr X)$ is a stack, as well as (2,3,4), we may assume by fppf descent that $\scr Y = A$ is an affine scheme.
Let us put $\ul{\scr X} = \repr_{\scr Y}(\scr X)$.
Then by Lemma \ref{lemm:inertia-composition} we have $I_{\scr X/A} \wequi I_{\scr X}$.
It follows now from \cite[Tags 06QD and 06QJ]{stacks-project} that $\scr X \to \ul{\scr X}$ is a gerbe (and in particular $\ul{\scr X}$ is an algebraic space, so stack).
Thus $\scr X \to \ul{\scr X}$ is smooth and surjective \cite[Tage 0DN8]{stacks-project}.

(3) Suppose that $\scr X \to A$ is (quasi-)separated.
Then so is $\ul{\scr X} \to \scr X$ \cite[Remark 2.11(i,iii)]{Rydh-quotients}.
Since $\ul{\scr X}$ is an algebraic space, $\Delta_{\ul{\scr X}/A}$ is separated \cite[Tag 03HK]{stacks-project}, and hence $\scr X \to \ul{\scr X}$ is (quasi-)separated \cite[Tag 050M]{stacks-project}.

Suppose that $\scr X \to A$ is locally of finite type (respectively presentation).
The morphism $\scr X \to \ul{\scr X}$ is always locally of finite presentation, being a gerbe \cite[Tag 0CPS]{stacks-project}.
Choose an fppf morphism $B \to \scr X$, with $B$ an affine scheme.
Then $B \to \scr X \to \ul{\scr X}$ is fppf and representable and $B \to A$ is locally of finite type (respectively presentation), whence so is $\ul{\scr X} \to A$ by \cite[Tags 06U8 and 06Q9]{stacks-project}.

Suppose that $\scr X \to A$ is quasi-compact.
Since $\scr X \to \ul{\scr X}$ is a universal homeomorphism\NB{ref?} it is quasi-compact, and since $\scr X$ is quasi-compact so is $\ul{\scr X}$.
It follows ($A$ being arbitrary) that $\ul{\scr X} \to A$ is quasi-compact.

Suppose that $\scr X \to A$ is locally quasi-finite.
Since $\ul{\scr X} \to A$ is representable it is quasi-DM, and hence locally quasi-finite by \cite[Tag 06UG]{stacks-project}.
Since $\scr X \to A$ is locally quasi-finite it is quasi-DM (by definition), hence so is $\scr X \to \ul{\scr X}$ \cite[Tag 050M]{stacks-project}.
Since $\scr X \to A$ is locally quasi-finite it is locally of finite type (by definition), hence so is $\scr X \to \ul{\scr X}$ (as we have already seen).
It follows that $\scr X \to \ul{\scr X}$ is locally quasi-finite, being a universal homeomeomorphism, quasi-DM and locally of finite type.

Since finite type means locally of finite type and quasi-compact \cite[06FS]{stacks-project}, this property is also dealt with; similarly for quasi-finite, finite presentation.

Suppose that $\scr X \to A$ is proper.
Then we have already seen that $\ul{\scr X} \to A$ is separated (since $\scr X \to A$ is), and hence $\scr X \to \ul{\scr X}$ is proper by \cite[Tag 0CPT]{stacks-project}.\NB{Or use that $\scr X \to \ul{\scr X}$ is universal homeomorphism, separated and finite type.}
We have also already seen that $\ul{\scr X} \to A$ is finite type, and hence it is proper by \cite[Tag 0CQK]{stacks-project}.

Suppose that $\scr X \to A$ is étale.
Then $\scr X \to A$ is unramified, and since $\ul{\scr X} \to A$ is DM (being representable \cite[Tag 050E]{stacks-project}), $\scr X \to \ul{\scr X}$ is unramified \cite[Tag 0CIZ]{stacks-project} and hence étale, being fppf \cite[Tag 0CJ1]{stacks-project}.
This implies that $\ul{\scr X} \to A$ is étale, since this property is étale local on the source (essentially by definition).

Suppose that $\scr X \to A$ is flat (respectively smooth).
Since $\scr X \to \ul{\scr X}$ is smooth surjective, we deduce that $\ul{\scr X} \to A$ is flat (respectively smooth).

(4) We may assume that $\scr Y = A$ is an affine scheme; then $\repr_{A}(\scr X) = \ul{\scr X}$ is the truncation, and similarly for $\scr X'$.
It follows that we may replace $A$ by $\ul{\scr X}$.
Working further locally on $\ul{\scr X}$ we may assume that $\ul{\scr X} = A$ is an affine scheme and $\scr X = A \sslash G$ \cite[Tag 06QH]{stacks-project}.
Here $G$ is an $A$-group scheme, namely the pullback of $I_{\scr X}$; in particular $G \to A$ is finite locally free.
Since $\scr X' \to \scr X$ is (quasi-)affine, $\scr X' = B \sslash G$, where $B$ is (quasi-)affine and carries a $G$-action\NB{Moreover stabilizer of $G$-action on $B$ is flat.} and $\ul{\scr X'} = B/G$.
It remains to show that $B/G$ is a (quasi-)affine scheme.
If $B$ is affine, this is proved in \cite[Theorem 4.1]{Rydh-quotients}.
If $B$ is quasi-affine, then $B \to \Spec(\scr O_B) =: B'$ is a $G$-equivariant open immersion into an affine $G$-scheme $B$ (over $A$).
Then $B/G \to B'/G$ is an open immersion, since GC quotients are uniformly categorical \cite[Remak 3.18]{Rydh-quotients}.
The result follows.\tdn{better argument/details}
\end{proof}

\begin{example} \label{ex:fact-etale}
Let $p: \scr X \to \scr Y$ be étale.
Then $\Delta_p$ is étale \cite[Tag 0CJ1]{stacks-project} whence flat and locally of finite presentation \cite[Tag 0DNP]{stacks-project}, and hence so is $I_p$.
In particular we may apply Proposition \ref{prop:relative-quotient-gerbe} to obtain a factorization $\scr X \to \ul{\scr X} \to \scr Y$ into étale morphisms with $\ul{\scr X} \to \scr Y$ representable.
If $p$ is in addition proper then $\ul{\scr X} \to \scr Y$ is finite étale (being separated and locally quasi-finite, so schematic \cite[Theorem A.2]{laumon2018champs}, whence finite since proper \cite[Tag 02LS]{stacks-project}).
\end{example}

\begin{example} \label{ex:fact-qf-proper}
Let $p: \scr X \to \scr Y$ be proper, finitely presented, quasi-finite and with flat inertia.
Then $\scr X \times_{\scr Y} \scr X \to \scr Y$ is proper and quasi-finite\NB{stability under base change and composition; hence in particular finite type}, whence $\Delta_{\scr X/\scr Y}$ is proper, finitely presented and quasi-finite \cite[Tags 0CPT, 06Q6 and 06UG]{stacks-project}.
Thus $\Delta_{\scr X/\scr Y}$ is proper and quasi-finite, so separated, so quasi-affine \cite[Theorem A.2]{laumon2018champs}, whence finite \cite[Tag 02LS]{stacks-project}.
It follows that $I_{\scr X/\scr Y}$ is finite locally free \cite[Tag 02KB]{stacks-project}.
We obtain a factorization $\scr X \to \ul{\scr X} \to \scr Y$ where $\ul{\scr X} \to \scr Y$ is proper, quasi-finite, flat, finitely presented and representable, and hence also finite locally free.
\end{example}

\begin{corollary} \label{cor:p*-ff}
Let $p: \scr X \to \scr Y$ be a gerbe.
Then the base change functor $p^*: \Stk_{\scr Y}^\repr \to \Stk_{\scr X}^\repr$ is fully faithful.
\end{corollary}
\begin{proof}
Denote the partially defined left adjoint of $p^*$ by $p_!$.
It suffices to show that for $\scr T \in \Stk_{\scr X}^\repr$, $p_!p^* \scr T$ is defined and isomorphic to $\scr T$.
By Corollary \ref{cor:bc-inertia}, $I_{p^* \scr T/\scr X}$ is flat and locally of finite presentation; hence by Proposition \ref{prop:relative-quotient-gerbe} $p_!p^* \scr T$ exists and is stable under base change.
We may thus assume that $\scr Y = A$ is an affine scheme and $\scr X = A \sslash G$ is a classifying stack.
Then $\Stk_{\scr X}^\repr$ is equivalent to the category of algebraic spaces over $A$ with a $G$-action, and $p^*$ is the ``trivial $G$-action'' functor.
This is clearly fully faithful.\NB{Or just directly observe that for any space $X$ and set $T$ over $\pi_0 X$ we have $\pi_0(X \times_{\pi_0 X} T) \wequi T$...}
\end{proof}

\subsubsection{Fixed points} \label{appsub:fixed-points}
Given a morphism of stacks $p: \scr X \to \scr Y$, the functor $p^*: \PSh(\Stk_{\scr Y}) \to \PSh(\Stk_{\scr X})$ (obtained by left Kan extension from the base change functor $\Stk_{\scr Y} \to \Stk_{\scr X}$) admits a right adjoint $p_*$.
\begin{proposition} \label{prop:fixed-points-functor}
Let $p: \scr X \to \scr Y$ be a gerbe with finite locally free inertia.
Let $\scr T \to \scr X$ be schematic\tdn{remove this assumption?} and separated.
\begin{enumerate}
\item $p_*(\scr T) \in \PSh(\Stk_{\scr Y})$ is representable by a stack $p_*(\scr T) \in \Stk_{\scr Y}$ which is representable and separated.
\item The adjunction map $p^*p_*(\scr T) \to \scr T$ is a closed immersion, finitely generated if $\scr T$ is locally of finite type.
\item If $\scr T$ is locally of finite type, affine or quasi-affine (over $\scr X$) then so is $p_*(\scr T)$ (over $\scr Y$).
\item Suppose that $p$ is tame.
  If $\scr T' \to \scr X$ is also schematic and separated, $\scr T, \scr T'$ are locally of finite presentation over $\scr X$, and $\alpha: \scr T' \to \scr T$ is an $\scr X$-morphism which is smooth, étale, an open immersion, a closed immersion or proper, then so is $p_*(\alpha)$ (tameness and locally finite presentation are not necessary for the last two).
\end{enumerate}
\end{proposition}
\begin{proof}
(1,2) The problem being local on $\scr Y$, we may assume that $\scr Y = A$ is an affine scheme and $\scr X = A \sslash G$.
Then $\scr T = T \sslash G$, $p_*(\scr T) = T^G$ and $p^*p_*(\scr T) = T^G \sslash G$.
The result is proved in \cite[Proposition A.8.10(1)]{conrad2015pseudo}; their proof establishes that $T^G \hookrightarrow T$ is finitely presented if $T$ is locally of finite type.

(3) Since $\scr X \to \scr Y$ is an fppf covering, it suffices to show that $p^*p_*(\scr T) \to p^*(\scr Y) = \scr X$ has the desired properties.
This factors as $p^*p_*(\scr T) \to \scr T \to \scr X$ where the first map is a closed immersion, finitely presented if $\scr Y$ is locally of finite type, and hence $p^*p_*(\scr T) \to \scr X$ is locally of finite type, affine or quasi-affine if $\scr T \to \scr X$ is.

(4) Again it suffices to show that $p^*p_*(\alpha)$ has the desired properties.
We may reduce to the case $\scr X = A \sslash G$ again.
The smoothness result is proved in  \cite[Proposition A.8.10(2)]{conrad2015pseudo}, assuming that $A$ is noetherian.
We can write $A = \colim_i A_i$, where $A_i$ is finitely generated over $\Z$ and hence noetherian.
Since $G$, $T$ and $T'$ are of finite presentation over $A$, they can be defined over $A_i$ for $i$ sufficiently large, and will satisfy all the same properties as before \cite[Propositions B.2 and B.3]{rydh2015noetherian}.
The smoothness follows in general.
If $\alpha$ is étale then $p^*p_*(\alpha)$ is smooth (by what we just said) and quasi-finite (as a consequence of (2)), and hence étale.
If $\alpha$ is an open immersion then $p^*p_*(\alpha)$ is an étale immersion (using (2) again), so an open immersion.
If $\alpha$ is a closed immersion or proper then so is $p^*p_*(\alpha)$, by (2).
\end{proof}

\subsubsection{Weil restriction} \label{appsub:weil-res}
\begin{corollary} \label{cor:weil-res}
Consider a cartesian square of stacks
\begin{equation*}
\begin{CD}
\scr X' @>>> \scr X \\
@Vp'VV        @VpVV   \\
\scr Y' @>>> \scr Y.
\end{CD}
\end{equation*}
Assume that $p$ is flat, proper, finitely presented, quasi-finite and with flat inertia.
Let $\scr T, \scr T' \to \scr X$ be quasi-affine and $\alpha: \scr T \to \scr T'$ an $\scr X$-morphism.
\begin{enumerate}
\item $p_*(\scr T) \in \PSh(\Stk_{\scr Y})$ is represented by a quasi-affine $\scr Y$-stack $p_*(\scr T)$.
\item Suppose $p$ is tame.
  If $\alpha$ is smooth, étale, an open immersion, or a closed immersion, then so is $p_*(\scr \alpha)$.
  If $p$ is étale, the same is true for $\alpha$ proper.\NB{Taking tangent bundles is a special case, and even if $X$ is proper $TX$ usually isn't, so the étale assumption cannot be removed entirely.}
\item The canonical map induces an equivalence \[ p_*(\scr T) \times_{\scr Y} \scr Y' \wequi p'_*(\scr T \times_{\scr X} \scr X'). \]
\end{enumerate}
\end{corollary}
\begin{proof}
Property (3) holds by definition of $p_*$, so we need only prove (1) and (2).
By Example \ref{ex:fact-qf-proper}, we may factor $p$ as $\scr X \xrightarrow{q} \ul{\scr X} \xrightarrow{r} \scr Y$, where $q$ is a gerbe with finite locally free inertia and $r$ is a finite locally free morphism.
We may thus assume that $p$ has one of these two properties.
If $p$ is a gerbe, the claims follow from Proposition \ref{prop:fixed-points-functor}.
If $p$ is finite locally free we may assume by descent that it is a morphism of schemes, and the claims follow from \cite[Proposition 7.6.2, proof of Theorem 7.6.4, Proposition 7.6.5]{bosch2012neron}.
They do not explicitly state that $p_*$ preserves smooth/étale morphisms, but this is clear from characterization in terms of lifting properties (see e.g. \cite[Tags 02HM and 02H6]{stacks-project}).
Similarly it is clear that $p_*$ for $p$ finite étale preserves proper morphisms, since we can reduce to the case of fold maps.
\end{proof}

\subsection{Unstable motivic homotopy theory of stacks}
Recall from \cite[Definition 2.5]{hoyois2017vanishing} that a Nisnevich covering of a stack $\scr X$ is a family of étale maps $\scr Y_i \to \scr X$ such that every residual gerbe lifts.
For definiteness, we call this an \emph{non-representable} Nisnevich covering.
If all the maps $\scr Y_i \to \scr X$ are representable (respectively quasi-affine), we call this a \emph{representable} (respectively \emph{quasi-affine}) Nisnevich covering.

Given a stack $\scr X$, we denote by $\SmQA_{\scr X}$ the full subcategory of $\Stk_{/\scr X}$ consisting of quasi-affine smooth (finitely presented) morphisms.
We say that a presheaf is \emph{Nisnevich local} if it satisfies the sheaf condition for all quasi-affine Nisnevich covers, and we say that it is \emph{homotopy invariant} if it inverts all torsors under vector bundles (see \cite[Definition 3.3]{hoyois-equivariant}; note that vector bundle torsors are affine since vector bundles are, so $\SmQA_{\scr X}$ is closed under vector bundle torsors and this definition makes sense).
We denote the corresponding localization functors by $L_\Nis$ and $L_\htp$.
\begin{definition}
We write $\Spc(\scr X) \subset \PSh_\Sigma(\SmQA_{\scr X})$  for the subcategory of presheaves that is both Nisnevich local and homotopy invariant, and denote by $L_\mot$ the localization functor.
\end{definition}

\subsubsection{Comparison results}
We now compare our definition of $\Spc(\scr S)$ with the ones given in \cite{khan2021generalized}.
In the sequel will frequently use terminology from \emph{loc. cit.} regarding (nicely or linearly) scalloped, fundamental and quasi-fundamental stacks.
\begin{proposition} \label{prop:nice-scallop}
Suppose that $\scr S$ is nicely scalloped \cite[Definition 2.9]{khan2021generalized} with affine diagonal\tdn{get away with less?}.
Then $\Spc(\scr S)$ is naturally equivalent to the category defined in \cite[\S3.1]{khan2021generalized}.
\end{proposition}
\begin{proof}
Then category of \cite{khan2021generalized} is obtained from the category $\PSh_\Sigma(\Sm_{\scr S}^\repr)$ of presheaves on smooth \emph{representable} $\scr S$-stacks by inverting \emph{representable} Nisnevich covers and contracting $\A^1$ (not a priori all vector bundle torsors).
This construction contracts all vector bundle torsors \cite[A.3.2(b)]{khan2021generalized}.
It will thus suffice to prove that the left Kan extension functor \[ \Shv_{\qaff\Nis}(\SmQA_\scr{S}) \to \Shv_{\repr\Nis}(\Sm_\scr{S}^\repr) \] is an equivalence.
Note that all objects of $\Sm^\repr_{\scr S}$ have quasi-affine diagonal (they are finitely presented over $\scr S$, whence quasi-separated over $\scr S$, so with quasi-affine diagonal over $\scr S$ \cite[note after Tag 02X5]{stacks-project}), and hence any morphism $\scr X \to \scr X'$ with $\scr X$ quasi-affine over $\scr S$ and $\scr X' \in \Sm_{\scr S}^\repr$ is automatically quasi-affine.
With this observation in mind, applying \cite[Lemma C.3]{hoyois2015quadratic} it suffices to prove that any object of $\Sm^\repr_{\scr S}$ admits a Nisnevich covering by objects in $\SmQA_{\scr S}$.
Using \cite[Theorem 2.12(c)]{khan2021generalized} we may assume that $\scr S$ is fundamental, and then we conclude by \cite[Theorem 2.14(i)]{khan2021generalized}.
\end{proof}

\begin{proposition} \label{ref:linear-scallop}
Suppose that $\scr S$ is linearly scalloped \cite[Definition A.10]{khan2021generalized}.
Then $\Spc(\scr S)$ is naturally equivalent to the category defined in \cite[\S A.3]{khan2021generalized}.
\end{proposition}
\begin{proof}
Note that a quasi-projective étale morphism is quasi-affine; one deduces that the left Kan extension functor \[ \Shv_{\qaff\Nis}(\SmQA_\scr{S}) \to \Shv_{qp\Nis}(\SmQP_{\scr{S}}) \] is fully faithful.
Since the homotopy localizations are the same on both sides \cite[Proposition 3.4(1)]{hoyois-equivariant}, we deduce the same statement on the level of motivic spaces.
Now we conclude by \cite[Proposition A.9]{khan2021generalized}.
\end{proof}

\begin{remark} \label{rmk:hoyois-comparison}
Proposition \ref{ref:linear-scallop} shows in particular that $\Spc(\scr S)$ coincides with the category of \cite{hoyois-equivariant} whenever the latter is defined, that is, whenever $\scr S = S \sslash G$ where a linearly reductive group scheme $G$ defined over an affine base $B$ acting on a quasi-projective $B$-scheme $S$.
\end{remark}

\subsubsection{Functoriality}
Given any morphism $f: \scr S' \to \scr S$ we have the base change functor $f^*: \SmQA_{\scr S} \to \SmQA_{\scr S'}$ inducing an adjunction \[ f^*: \PSh_\Sigma(\SmQA_{\scr S}) \adj \PSh_\Sigma(\SmQA_{\scr S'}): f_* \] in which $f^*$ preserves motivic equivalences.
It follows as usual that there is an induced adjunction \[ f^*: \Spc(\scr S) \adj \Spc(\scr S'): f_*. \]

If $f$ is quasi-affine, then $f^*: \SmQA_{\scr S} \to \SmQA_{\scr S'}$ admits a left adjoint $f_\sharp$, inducing an adjunction \[ f_\sharp: \PSh_\Sigma(\SmQA_{\scr S'}) \adj \PSh_\Sigma(\SmQA_{\scr S}): f^*. \]
Here both functors preserve colimits and motivic equivalences, and there is an induced adjunction \[ f_\sharp: \Spc(\scr S') \adj \Spc(\scr S): f^*. \]

\subsubsection{Nisnevich separation}
\begin{lemma} \label{lemm:qaff-descent}
The presheaves on $\Stk$ given by \[ \scr X \mapsto L_\Nis \PSh_\Sigma(\SmQA_{\scr X}) \quad\text{and}\quad \scr X \mapsto \Spc(\scr X) \] satisfy quasi-affine Nisnevich descent.
\end{lemma}
\begin{proof}
This is the same proof as \cite[Proposition 4.8]{hoyois-equivariant}.
Note that any morphism of quasi-affine schemes is quasi-affine \cite[Tag 054G]{stacks-project}.
Write $\Shv(\SmQA_{\scr X}) := L_\Nis \PSh_\Sigma(\SmQA_{\scr X})$ for the Nisnevich topos.
If $\scr X \to \scr Y$ is a quasi-affine étale morphism, then $\Shv(\SmQA_{\scr X}) \wequi \Shv(\SmQA_{\scr Y})_{/\scr X}$.
The first claim thus follows from internal descent for $\infty$-topoi \cite[Theorem 6.1.3.9]{lurie-htt}.
For the second claim, given the first one, it suffices to show that if $\scr F \in \Shv(\SmQA_{\scr X})$, then (1) given $f: \scr Y \to \scr X$ quasi-affine étale, if $\scr F$ is homotopy invariant then so is $f^* \scr F$, and (2) given a quasi-affine Nisnevich covering $f: \scr Y \to \scr X$, if $f^* \scr F$ is homotopy invariant then so is $\scr F$.
This is clear.
\end{proof}
\begin{lemma} \label{lemm:Nis-sep}
Let $\{f_i: \scr S_i \to \scr S\}$ be a quasi-affine Nisnevich covering.
Then the functors $f_i^*: \Spc(\scr S) \to \Spc(\scr S_i)$ are jointly conservative.
\end{lemma}
\begin{proof}
This is true for any functor satisfying quasi-affine Nisnevich descent (like $\Spc(\ph)$, by Lemma \ref{lemm:qaff-descent}).
Indeed given a morphism $\alpha \in \Spc(\scr S)$ with $f_i^*(\alpha)$ an equivalence for all $i$, it suffices to show that if $p: C \to \scr S$ is any of the terms of the Čech nerve, then $p^*(\alpha)$ is an equivalence.
Since $p$ factors through one of the $f_i$, this is clear.
\end{proof}

\subsubsection{Continuity}
\begin{lemma} \label{lemm:cty-spaces}
Let $\scr X \wequi \lim_\alpha \scr X_\alpha$ where the transition maps are affine and the diagram is cofiltered.
Then $\Spc(\scr X) \wequi \lim_\alpha \Spc(\scr X_\alpha)$.
\end{lemma}
\begin{proof}
This is a standard argument; see e.g. \cite[Proposition C.7]{hoyois2015quadratic} for a similar proof.
\end{proof}

\subsubsection{Homotopy purity}
By a \emph{quasi-affine smooth closed pair} we mean $\scr X \in \SmQA_{\scr S}$ and $\scr Z \subset \scr X$ closed and smooth over $\scr S$ (so that $\scr Z \in \SmQA_{\scr S}$).
In this situation there in particular exists a \emph{normal bundle} $N_{\scr Z/\scr X}$ over $\scr Z$.

\newcommand{\reasonableness}[1]{Suppose that $#1$ admits a quasi-affine Nisnevich cover by (linearly or nicely) quasi-fundamental stacks.}
\begin{lemma} \label{lemm:homotopy-purity}
\reasonableness{\scr S}
Let $(\scr X, \scr Z)$ be a quasi-affine smooth closed pair over $\scr S$.
There exists a canonical equivalence \[ \scr X/\scr X \setminus \scr Z \wequi N_{\scr Z/\scr X}/N_{\scr Z/\scr X} \setminus \scr Z =: Th(N_{\scr Z/\scr X}) \in \Spc(\scr S)_*. \]
\end{lemma}
\begin{proof}
The deformation space construction provides a canonical zigzag \[ \scr X/\scr X \setminus Z \to D/D \setminus \scr Z \times \A^1 \leftarrow Th(N_{\scr Z/\scr X}); \] it suffices to show that both maps are motivic equivalences.
By Nisnevich separation (Lemma \ref{lemm:Nis-sep}) and our local structure assumption, we may assume that $\scr S$ is linearly or nicely quasi-fundamental.
In the linearly quasi-fundamental case, this reduces (using Remark \ref{rmk:hoyois-comparison}) to \cite[Theorem 3.23]{hoyois-equivariant}.
In the nicely quasi-fundamental case, the same proof works (see also \cite[Proof of Theorem 6.6]{khan2021generalized}).
\end{proof}

\section{Norms} \label{sec:norms}

\subsection{Normed $\infty$-categories} \label{subsec:normed-oo-cat}
Throughout this section we shall work in the following context.
\begin{itemize}
\item We fix an extensive \cite[Definition 2.3]{bachmann-norms} $\infty$-category $\scr C$ and classes $\scr E \subset \scr L$ of morphisms in $\scr C$.
\item We assume that pullbacks along morphisms in $\scr E$  exist in $\scr C$, $\scr E$ contains all equivalences and fold maps, and is closed under composition, pullback and disjoint union, and the same holds for $\scr L$.
\end{itemize}

\begin{definition}
By a \emph{normed $\infty$-category} we mean a functor \[ \scr D^\otimes: \Span(\scr C, \all, \scr E) \to \Cat_\infty \] preserving finite products.
Given a span $X \xleftarrow{f} Y \xrightarrow{p} Z$ (with $Z \in \scr E$) we denote the induced functors by \[ f^*: \scr D(X) \to \scr D(Y) \quad\text{and}\quad p_\otimes: \scr D(Y) \to \scr D(Z). \]

We call $\scr D^\otimes$ \emph{presentably normed} if the following hold.
\begin{enumerate}
\item Each $\scr D(X)$ is presentable and each functor $f^*: \scr D(X) \to \scr D(Y)$ preserves colimits.
\item For each $p \in \scr E$ the functor $p_\otimes$ preserves sifted colimits.
\item For each $f \in \scr L$ the functor $f^*$ has a left adjoint $f_\sharp$ compatible with base change.
\item The distributivity law of \cite[Proposition 5.10(1)]{bachmann-norms} holds for $h \in \scr L$ and $p \in \scr E$.
\end{enumerate}
\end{definition}

\subsubsection{Free normed spaces}
Suppose given a category $\scr C$ with classes $\scr E, \scr L$ as above.
Throughout this subsection, we further assume the following.
\begin{enumerate}
\item For $X \in \scr C$ we denote by $\scr L_X \subset \scr C_{/X}$ the full subcategory on morphisms in $\scr L$.
  We assume that for $p: X \to Y \in \scr E$ the functor $p^*: \scr L_{Y} \to \scr L_{X}$ admits a right adjoint $p_\otimes$.
\item We assume furthermore $\scr L$ is closed under $p_\otimes$.
\item Denote by $\scr E_/ \in \PSh(\scr C)$ the presheaf $X \mapsto \scr E_{/X}^\wequi$.
  We assume given a decomposition \begin{equation}\label{eq:E-decomp} \coprod_{\lambda \in \Lambda} M_\lambda \wequi \scr E_/ \in \PSh_\Sigma(\scr C). \end{equation}
\end{enumerate}
\begin{remark}
Note that $\scr E_/ \in \PSh_\Sigma(\scr C)$, since $\scr E$ is closed under finite coproducts.
\end{remark}

In this situation, given $X \in \scr C$, we can form the category \cite[\S C]{bachmann-norms} \[ \Cor^{\scr E}(X) = \Span(\scr L_X, \scr E, \all). \]
The functor $\scr L_X \to \Cor^{\scr E}(X)$ preserves finite coproducts \cite[Lemma C.3]{bachmann-norms} and hence induces a cocontinuous functor \[ F: \PSh_\Sigma(\scr L_X) \to \PSh_\Sigma(\Cor^{\scr E}(X)). \]
We denote by $U$ its right adjoint.
\begin{remark}
The functors $X \mapsto \scr L_X$ and $X \mapsto \Span(\scr L_X, \scr E, \all)$ can be upgraded to normed $\infty$-categories $\scr L^\otimes, \Span(\scr L^\otimes, \scr E, \all)$ (see \S\ref{subsec:norms-for-spaces} for details about the case of spans), and their sifted cocompletions to presentably normed $\infty$-categories $\PSh_\Sigma(\scr L^\otimes), \PSh_\Sigma(\Span(\scr L^\otimes, \scr E, \all))$ (see \cite[proof of Lemma 5.6]{bachmann-norms} for the distributivity law).
\end{remark}
Let $Y \in \scr L_X$, $\lambda \in \Lambda$ and $\alpha: Y \to M_\lambda \in \PSh_\Sigma(\scr C)$.
Pulling back along $M_\lambda \to \scr E_/$ we obtain a map $p: Y' \to Y \in \scr E$.
We denote this situation by $(p: Y' \to Y) \in M_\lambda|_X$.
We can define a functor \[ \D_\lambda: \scr L_X \to \PSh_\Sigma(\scr L_X), X' \mapsto \colim_{p: Y' \to Y \in M_\lambda|_X} p_\otimes(X' \times_X Y'). \]
See \cite[\S2]{bachmann-powerops} for details.
By the universal property of $\PSh_\Sigma$, this extends to a sifted cocontinuous functor \begin{equation} \label{eq:D-lambda} \D_\lambda: \PSh_\Sigma(\scr L_X) \to \PSh_\Sigma(\scr L_X). \end{equation}

Our aim is to determine a formula for the functor $UF: \PSh_\Sigma(\scr L_X) \to \PSh_\Sigma(\scr L_X)$.
Note that \[ UF X \wequi \scr E_/|_X := \scr E_/|_{\scr L_X}, \] essentially by definition.
Given $X' \in \scr L_X$, the canonical map $X' \to X$ induces $UFX' \to UFX \wequi \scr E_/|_X$.

\begin{theorem} \label{thm:free-space-with-transfers}
For $X \in \PSh_\Sigma(\scr L_S)$, there is a natural equivalence \[ UFX \wequi \coprod_{\lambda \in \Lambda} \D_\lambda(X) \in \PSh_\Sigma(\scr L_S). \]
\end{theorem}
\begin{proof}
The functor $UF$ preserves sifted colimits, and so does the right hand side.
Hence it suffices to exhibit the natural equivalence for $X \in \scr L_S$.
Considering the map $UFX \to UFS \wequi \scr E_/|_S$, by universality of colimits in $\PSh_\Sigma(\scr C)$\NB{ref?} we obtain \[ UFX \wequi \coprod_\lambda \colim_{p: Y' \to Y \in M_\lambda|_S} Y \times_{\scr E_/|_S} UF X =: \coprod_\lambda D'_\lambda(X). \]
It remains to identify $D'_\lambda$ with $D_\lambda$.
Thus we need to show that \[ Y \times_{\scr E_/|_S} UF X \wequi p_\otimes(X \times_S Y'), \] naturally in $X$ and $Y$.

Let $Z \in (\scr L_S)_{/Y}$.
It suffices to show that \[ \Map'(Z, UFX) \wequi \Map_Y(Z, p_\otimes(X \times Y')), \] naturally in $X,Y,Z$.
Here $\Map'$ denotes those maps lying over $Z \to Y \xrightarrow{p} \scr E_/|_S$.
By definition, the left hand side consists of spans $Z \leftarrow Z \times_Y Y' \to X$, whereas the right hand side is $\Map_{Y'}(Z \times_Y Y', X \times Y')$.
These are clearly the same.
\end{proof}

\subsubsection{Free normed objects: general case}
We now extend some of the results of \cite{bachmann-norms} to our more general context.
We assume familiarity in particular with \cite[\S7,\S16]{bachmann-norms}.
Given a normed $\infty$-category $\scr D$ and $X \in \scr C$, following \emph{loc. cit.} we can make sense of the category $\NAlg(\scr D(X))$ as a category of partially cartesian sections of the cocartesian fibration corresponding to $\scr D|_{\Span(\scr L_X, \all, \scr E)}$.

Next we extend the thom spectrum construction.
\begin{theorem}
Let $\scr D^\otimes$ be a presentably normed $\infty$-category.
There is a natural transformation \[ M: \scr L^\otimes_{\sslash \scr D} \to \scr D^\otimes \in \Fun(\Span(\scr C, \all, \scr E), \Cat) \] with the following properties:
\begin{enumerate}
\item For $X \in \scr C$ we have $\scr L^\otimes_{\sslash \scr D}(X) \wequi \scr (\scr L_X)_{\sslash \scr D_X}$, the source of the cartesian fibration classified by the functor $\scr L_X^\op \to \Cat$ obtained by restricting $\scr D^\otimes$.
\item For $Y \in \scr L_X$ and $e = (E \in \scr D(Y), (p:Y \to X) \in \scr L_X) \in \scr (\scr L_X)_{\sslash \scr D_X}$ we have $M(e) \wequi p_\sharp(E) \in \scr D(X)$.
\end{enumerate}
\end{theorem}
\begin{proof}
The construction of \cite[\S16.3]{bachmann-norms} goes through unchanged.\tdn{really?}
\end{proof}

\begin{remark} \label{rmk:M-right-adjoint}
Since $\scr D$ is valued in $\Cat_\infty^\sift$, we can extend $M$ to the objectwise sifted cocompletion \[ M: \PSh_\Sigma(\scr L_{\sslash \scr D})^\otimes \to \scr D^\otimes. \]
This functor has a sectionwise right adjoint, which we can informally describe as \[ \scr D(X) \ni d \mapsto ((X \xrightarrow{\id} X) \in \scr L_X, d \in \scr D(X)) \in \scr L_{X \sslash \scr D}. \]
%Also by construction there is a forgetful transformation \[ \scr L_{\sslash \scr D}^\otimes \to \scr L^\otimes, \] and again we can pass to sifted cocompletions.
\end{remark}

\begin{lemma} \label{lemm:NAlg-adj}
Suppose given $F: \scr D_1 \to \scr D_2 \in \Fun(\Span(\scr C, \all, \scr E), \Cat_\infty)$.
Suppose that $F$ sectionwise admits a right adjoint compatible with pullbacks along morphisms in $\scr L$.
Then there is an adjunction \[ F: \NAlg(\scr D_1) \adj \NAlg(\scr D_2): R \] computed sectionwise.
\end{lemma}
\begin{proof}
There is an adjunction on section categories by \cite[Lemmas D.3(1) and D.6]{bachmann-norms}, which restricts to the subcategories of normed objects by assumption.
\end{proof}

\begin{lemma} \label{lemm:Spcfet-NMon-abstract}
There is a canonical equivalence \[ \NAlg(\PSh_\Sigma(\scr L_X)) \wequi \PSh_\Sigma(\Cor^{\scr E}(X)). \]
\end{lemma}
\begin{proof}
The argument of \cite[Remark 16.18]{bachmann-norms} goes through unchanged.\tdn{really?}
\end{proof}

\begin{proposition} \label{prop:free-normed}
Let $S \in \scr S$.
\begin{enumerate}
\item The following diagram commutes
\begin{equation*}
\begin{CD}
\PSh_\Sigma(\scr L_{S\sslash \scr D}) @>F>> \NAlg(\PSh_\Sigma(\scr L_{S\sslash \scr D})) \\
@VVV                                           @VVV                                            \\
\PSh_\Sigma(\scr L_S) @>F>> \NAlg(\PSh_\Sigma(\scr L_S)), \\
\end{CD}
\end{equation*}
  where the vertical functors forget the object of $\scr D$, and the horizontal functor is the free normed object.

\item The following diagram commutes
\begin{equation*}
\begin{CD}
\PSh_\Sigma(\scr C_{S\sslash \scr D}) @>F>> \NAlg(\PSh_\Sigma(\scr L_{S\sslash \scr D})^\otimes) \\
@VMVV                                           @VMVV                                          \\
\scr D(S)                            @>F>> \NAlg(\scr D(S)). \\
\end{CD}
\end{equation*}

\item The following diagram commutes
\begin{equation*}
\begin{CD}
\NAlg(\PSh_\Sigma(\scr L_{S\sslash \scr D})) @>M>> \NAlg(\scr D(S)) \\
                 @VVV                                             @VVV          \\
\PSh_\Sigma(\scr L_{S\sslash \scr D})   @>M>> \scr D(S), \\
\end{CD}
\end{equation*}
  where the vertical functors are the forgetful ones.
\end{enumerate}
\end{proposition}
\begin{proof}
(1, 2) It suffices to establish commutativity of the right adjoints, which is immediate from Lemma \ref{lemm:NAlg-adj}.

(3) Immediate from Lemma \ref{lemm:NAlg-adj} and Remark \ref{rmk:M-right-adjoint}.
\end{proof}

\begin{theorem}\label{thm:free-normed}
Let $\scr D$ be presentably normed and $S \in \scr C$.
Then for $E \in \scr D(S)$ we have \[ FE \wequi \coprod_\lambda \colim_{X \xrightarrow{p} Y \xrightarrow{q} * \in M_\lambda|_S} q_\sharp p_\otimes(E_X). \]
\end{theorem}
\begin{proof}
Consider $e = (E \in \scr D(*), * \in \scr C) \in \scr C_{\sslash \scr D}$.
Then $Me \wequi E$ by construction; we thus need to determine (the object underlying) $FMe$.
By Proposition \ref{prop:free-normed}(2) this is the same as $MFe$, and by Proposition \ref{prop:free-normed}(3) the underlying object is the same as $MUFe$.
By Proposition \ref{prop:free-normed}(1), $UFe \in \PSh(\scr C_{\sslash \scr D})$ is given by a diagram $F(*) \to \scr D$.
We can see that it takes the claimed form by unwinding the definitions and using Theorem \ref{thm:free-space-with-transfers}, via Lemma \ref{lemm:Spcfet-NMon-abstract}.
\end{proof}

\subsubsection{The case of stacks}
Fix a base stack $\scr S$.
From now on we put $\scr C = \Stk_{/\scr S}$, $\scr L$ the class of smooth quasi-affine morphisms, $\scr E$ the class of finite étale morphisms.
For the decomposition \eqref{eq:E-decomp} we set $\lambda = \Z_{\ge 0}$ and $M_\lambda$ the stack of finite étale schemes of degree $n$.

\subsection{Norms for spaces} \label{subsec:norms-for-spaces}
Let $p: \scr X \to \scr Y$ be a morphism (of stacks) which is tame, proper, flat, finitely presented, quasi-finite, and with flat inertia.
By Corollary \ref{cor:weil-res}, the functor $p^*: \SmQA_{\scr Y} \to \SmQA_{\scr X}$ has a right adjoint $p_\otimes$, stable under base change.

\begin{remark} \label{rmk:pffqf-fact}
By Example \ref{ex:fact-qf-proper}, $p$ factors as $\scr X \xrightarrow{q} \ul{\scr X} \xrightarrow{r} \scr Y$, where $q$ is a gerbe with finite locally free inertia and $r$ is finite locally free.
It follows that $p_\otimes \wequi r_\otimes q_\otimes$, where $q_\otimes$ is fppf locally given by taking fixed points with respect to the action of a finite locally free group scheme, and $r_\otimes$ is fppf locally given by the usual Weil restriction.
\end{remark}
\begin{example} \label{ex:pet}
Example \ref{ex:fact-etale} shows the following:
if $p: \scr X \to \scr Y$ is proper étale, then it satisfies all of the above assumptions, and in the factorization $\scr X \to \ul{\scr X} \to \scr Y$ the first map has finite étale inertia and the second one is finite étale.
\end{example}

\begin{proposition} \label{prop:norms-equiv}
Let $p: \scr X \to \scr Y$ be a finite étale morphism of stacks.
Then the sifted cocontinuous extension \[ p_\otimes: \PSh_\Sigma(\SmQA_{\scr X}) \to \PSh_\Sigma(\SmQA_{\scr Y}). \] preserves Nisnevich and motivic equivalences.
\end{proposition}
\begin{remark} \label{rmk:p-otimes-*}
The functor $p_\otimes$ coincides with the right adjoint $p_*$ of $p^*: \PSh_\Sigma(\SmQA_{\scr Y}) \to \PSh_\Sigma(\SmQA_{\scr X})$.
Indeed this holds on representables by construction, and both functors preserve sifted colimits.
\end{remark}
\begin{proof}
Given a Nisnevich square \cite[(2.8)]{hoyois2017vanishing}
\begin{equation*}
\begin{CD}
\scr W @>>> \scr V \\
@VVV        @VVV \\
\scr U @>>> \scr T
\end{CD}
\end{equation*}
in $\SmQA_{\scr X}$, denote by $S_{\scr U,\scr V} \hookrightarrow T$ the sieve corresponding to $\scr U \amalg \scr V \to \scr T$.
We call maps of this form generating Nisnevich equivalences.
Indeed by Lemma \ref{lemm:squared-sieve} below and \cite[Proposition 2.9]{hoyois2017vanishing}, a $\Sigma$-presheaf is a Nisnevich sheaf if and only if it is local for these maps.
Similarly, we call vector bundle torsors $G \to \scr T$ generating homotopy equivalences.
Since our generating equivalences are stable under coproduct with identities, by \cite[Lemma 2.10]{bachmann-norms} it suffices to show that $p_\otimes$ sends generating Nisnevich (respectively homotopy) equivalences to Nisnevich (respectively homotopy) equivalences.

We first deal with homotopy equivalences.
Thus let $V \to \scr T$ be a vector bundle and $G \to \scr T$ a $V$-torsor.
Since Weil restriction along finite étale morphisms preserves vector bundles and torsors\todo{ref?}, we see that $p_\otimes V \to p_\otimes \scr T$ is a vector bundle, and $p_\otimes G \to p_\otimes \scr T$ is a $p_\otimes V$-torsor.
The claim follows.

For the Nisnevich equivalences, we argue as follows.\NB{ref/more succinct argument?}
By Remark \ref{rmk:p-otimes-*}, $p_\otimes(S_{\scr U, \scr V}) \to \scr T$ is the sieve generated by $p_\otimes(\scr U \amalg \scr V) \to p_\otimes(\scr T)$.
It is thus enough to show that $p_\otimes(\scr U \amalg \scr V) \to p_\otimes(\scr T)$ is a Nisnevich covering.
Let $A \to \scr Y$ be a flat morphism splitting $\scr X$; in other words $\scr X_A \wequi A^{\amalg n}$ for some $n$.
Let $p^A: \scr X_A \wequi A^{\amalg n} \to A$ be the induced fold map.
Given $\scr B \in \SmQA_{\scr X_A}$, denote by $\scr B_i \in \SmQA_A$ its components.
We have $p^A_\otimes(\scr B) \wequi \prod_i B_i$.
Put $\scr Z = \scr T \setminus \scr U$ and choose a section $\scr Z \to \scr V$ (which exists by definition of Nisnevich square).
For $d \ge 0$, denote by \[ \scr S_d(A) \subset p^A_\otimes(\scr T_A) \wequi \prod_i (\scr T_A)_i = \prod_i (\scr U_A \amalg \scr Z_A)_i \] the locally closed subset where precisely $d$ points lie in $(\scr Z_A)_i$ (for various $i$).
Then \[ p_\otimes^A(\scr T_A) = \coprod_{d \ge 0} \scr S_d(A) \] is a locally closed decomposition.
Furthermore $\scr S_d(A)$ decomposes into disjoint clopen subsets corresponding to which factors are contained in $\scr Z$ and which in $\scr U$.
This implies that there is a section \[ \sigma_d(A): \scr S_d(A) \to p_\otimes(\scr U_A \amalg \scr V_A), \] given on each component by lifting each factor into $\scr U$ or $\scr V$.
The construction of $\scr S_d(A)$ and $\sigma_d(A)$ is functorial in $A$ (in particular does not depend on the choice of decomposition $\scr X_A \wequi A^{\amalg n}$).
This implies that the $\scr S_d(A)$ descend to a locally closed closed decomposition of $p_\otimes(\scr T)$, and the $\sigma_d(A)$ descend to sections of $p_\otimes(\scr U \amalg \scr V) \to p_\otimes(\scr T)$ over the strata.
It follows that this is indeed a Nisnevich covering \cite[Proposition 2.6]{hoyois2017vanishing}.
\end{proof}

\begin{lemma} \label{lemm:squared-sieve} \todo{Is this a general cd structure thing?}
Let
\begin{equation*}
\begin{CD}
\scr W @>>> \scr V \\
@VVV        @VVV \\
\scr U @>>> \scr T
\end{CD}
\end{equation*}
be a Nisnevich square in $\SmQA_{\scr X}$.
Then $L_\Sigma(\scr U \amalg_{\scr W} \scr V) \in \PSh_\Sigma(\SmQA_{\scr X})$ is equivalent to the sieve generated by the covering $\scr U \amalg \scr V \to \scr T$.
\end{lemma}
\begin{proof}
The morphism $\scr W \to \scr V$ is (sectionwise) a monomorphism (of sets).
This implies that $\scr F := \scr U \amalg_{\scr W} \scr V \in \PSh(\SmQA_{\scr X})$ is $0$-truncated.\NB{I.e. pushouts along monomorphisms are homotopy pushouts.}
Its image in $\scr T$ is the union of the sieves generated by $\scr V$ and $\scr U$, so contained in the sieve $\scr S$ generated by $\scr U \amalg \scr V$.
It remains to show that the resulting map $\scr F \to \scr S$ is $\Sigma$-locally an isomorphism.

\emph{Surjectivity.} Given $\beta: \scr Y \to \scr T$ and a factorization through $\alpha: \scr Y \to \scr U \amalg \scr V$, we get $\scr Y \wequi \alpha^{-1}(\scr U) \amalg \alpha^{-1}(\scr V)$.
By construction, $\beta|_{\alpha^{-1}(\scr U)}, \beta|_{\alpha^{-1}(\scr V)}$ are in the image of $\scr F \to \scr S$.

\emph{Injectivity.} Suppose given two sections $[s_1], [s_2] \in \scr F(\scr Y)$, corresponding to $s_1, s_2 \in \Hom(\scr Y, \scr U) \amalg \Hom(\scr Y, \scr V)$, and suppose that $s_1, s_2$ have the same image in $\Hom(\scr Y, \scr X)$.
If $s_1 \in \Hom(\scr Y, \scr U)$ then either $s_2 \in \Hom(\scr Y, \scr U)$ or $s_2 \in \Hom(\scr Y, \scr W)$; in either case $[s_1] = [s_2]$.
We thus need only deal with the case $s_1, s_2 \in \Hom(\scr Y, \scr V)$.
Since $\scr V \to \scr T$ is quasi-affine it is separated, and since it is also étale, its diagonal is clopen \cite[Tags 01KK and 02GE]{stacks-project}.
It follows that there is a decomposition $\scr Y \wequi \scr Y_1 \amalg \scr Y_2$ such that $s_1|_{\scr Y_1} = s_2|_{\scr Y_1}$ and $s_1|_{\scr Y_2}, s_2|_{\scr Y_2}$ have disjoint images.
Passing to the $\Sigma$-cover, we may thus assume that $s_1$ and $s_2$ have disjoint images.
Since $\scr V \to \scr T$ is an isomorphism over the complement of $\scr U$ (and thus $s_1, s_2$ have the same image over any point of that complement), we deduce that $s_1, s_2$ factor through $\scr W$.
Thus again $[s_1] = [s_2]$.

This concludes the proof.
\end{proof}

%Let us spell out the following.
%\begin{corollary}
%Let $p: \scr S' \to \scr S$ be a finite étale morphism of stacks and $\scr X' \to \scr X \in \SmQA_{\scr S'}$ a Nisnevich cover.
%Then $p_\otimes(\scr X') \to p_\otimes(\scr X) \in \SmQA_{\scr S}$ is a Nisnevich cover.
%\end{corollary}
%\begin{proof}
%$p_\otimes(\scr X') \to p_\otimes(\scr X)$ is an étale morphism whose associated sieve is a Nisnevich equivalence, by Proposition \ref{prop:norms-equiv} and Remark \ref{rmk:p-otimes-*}.
%This implies the claim.\todo{really?/ref?}
%\end{proof}

\begin{proposition} \label{prop:fixed-equiv}
Let $p: \scr X \to \scr Y$ be a tame gerbe with finite locally free inertia.
Then the cocontinuous extension $p_\otimes: \PSh(\scr X) \to \PSh(\scr Y)$ preserves Nisnevich and motivic equivalences.
\end{proposition}
\begin{remark} \label{rmk:p-otimes-*-2}
The restriction of $p_\otimes$ to $\PSh_\Sigma(\scr X) \subset \PSh(\scr X)$ preserves sifted colimits and has essential image contained in $\PSh_\Sigma(\scr Y)$; hence this is the sifted cocontinuous extension.
It follows that $p_\otimes: \PSh_\Sigma(\scr X) \to \PSh_\Sigma(\scr Y)$ also preserves Nisnevich and motivic equivalences.
Also, as in Remark \ref{rmk:p-otimes-*}, we see that both functors denoted $p_\otimes$ are right adjoint to the appropriate functor $p^*$.
\end{remark}
\begin{proof}
It suffices to show that $p_\otimes$ preserves generating homotopy and Nisnevich equivalences.
For homotopy equivalences, this follows from the fact that $p_\otimes$ preserves vector bundles and vector bundle torsors; indeed this claim can be checked fppf locally whence we may assume that $p_\otimes$ is equivalent to taking $G$-fixed points, in which case this is clear.
For Nisnevich equivalences, let $A \to \scr T$ be the Nisnevich sieve generated by the covering $A_1, A_2, \dots, A_n \to \scr T$.
Then by Remark \ref{rmk:p-otimes-*-2} we find that $p_\otimes(A) \to p_\otimes(\scr T)$ is the sieve generated by $\{p_\otimes(A_i)\}$; in other words it suffices to show that $p_\otimes$ preserves Nisnevich coverings.
By Proposition \ref{prop:fixed-points-functor}, it preserves étale morphisms and open and closed immersions.
Being a right adjoint it preserves fiber products, and hence by the characterization via Nisnevich squares \cite[Proposition 2.9]{hoyois2017vanishing} it suffices to show that $p_\otimes$ preserves locally closed decompositions.
This can be checked fppf locally, so for $G$-fixed points, where it is obvious.
\end{proof}

\begin{remark} \label{rmk:p*-pres-coverings}
The proofs of Propositions \ref{prop:fixed-equiv} and \ref{prop:norms-equiv} show that when they apply, the functors $p_\otimes$ preserve Nisnevich coverings consisting of one map.
\end{remark}

Let us write $\tpffqf$ for the class of morphisms that are tame, proper, flat, finitely presented, quasi-finite, and with flat inertia (i.e. those for which we have constructed $p_\otimes$).
For $p: \scr X \to \scr Y \in \tpffqf$, the functor $p_\otimes: \SmQA_{\scr X} \to \SmQA_{\scr Y}$ is by definition right adjoint to $p^*$, and its formation is compatible with base change (see Corollary \ref{cor:weil-res}).
Using Barwick's unfurling construction \cite[Proposition 11.6]{BarwickMackey} we thus obtain a functor \[ \SmQA^\otimes: \Span(\Stk, \all, \tpffqf) \to \Cat. \]

We write $\tpet$ for the class of tame proper étale morphisms.
By Example \ref{ex:pet}, we have $\tpet \subset \tpffqf$, and any morphism in $\tpet$ factors into a gerbe with finite étale inertia followed by a finite étale morphism.
Arguing as in \cite[\S6.1]{bachmann-norms}, it follows (from Proposition \ref{prop:norms-equiv}, Proposition \ref{prop:fixed-equiv} and Remarks \ref{rmk:p-otimes-*}, \ref{rmk:p-otimes-*-2}) that we can apply $\PSh_\Sigma$ and perform motivic localization sectionwise to obtain  \begin{equation*} L_\mot: \PSh_\Sigma(\SmQA^\otimes) \to \Spc^\otimes: \Span(\Stk, \all, \tpet) \to \Cat. \end{equation*}
Restricting to $\fet \subset \tpet$, we have in particular constructed a normed category $\Spc^\otimes$.

\begin{remark}
It seems likely that $p_\otimes$ preserves motivic equivalences for all $p \in \tpffqf$, not just $p \in \tpet$, but this seems more difficult to prove.
In any case, Proposition \ref{prop:fixed-equiv} shows that if $p: \scr X \to \scr Y \in \tpffqf$ is a gerbe then we do have a well-behaved functor $p_\otimes = p_*: \Spc(\scr X) \to \Spc(\scr Y)$.
\end{remark}

\begin{lemma} \label{lemm:p-otimes-colim}
Let $p: \scr X \to \scr Y$ be proper étale.
Then $p_\otimes: \Spc(\scr X) \to \Spc(\scr Y)$ preserves sifted colimits.
If $p$ is a gerbe, then $p_\otimes$ preserves all colimits (even for $p \in \tpffqf$).
\end{lemma}
\begin{proof}
The preservation of sifted colimits is by construction.
The preservation of all colimits follows from the fact that $p_*: \PSh(\SmQA_{\scr X}) \to \PSh(\SmQA_{\scr Y})$ preserves motivic equivalences (by Proposition \ref{prop:fixed-equiv}) and colimits.
\end{proof}

We record also the following.
\begin{corollary} \label{cor:Dn-mot-equiv}
The functors $\D_n: \PSh_\Sigma(\SmQA_{\scr S}) \to  \PSh_\Sigma(\SmQA_{\scr S})$ preserve motivic equivalences.
\end{corollary}
\begin{proof}
Since motivic equivalences are closed under colimits, and the functors $f_\sharp$ (as always) and $p_\otimes$ (by Proposition \ref{prop:norms-equiv}) preserve motivic equivalences, this follows from the defining formula \eqref{eq:D-lambda} of $\D_n$.
\end{proof}
In particular, the functors $\D_n$ descend to endofunctors of $\Spc(\scr S)$.

\subsection{Spaces with finite étale transfers} \label{subsec:spcfet}
\subsubsection{}
For $\scr S \in \Stk$ we denote by \[ \Cor^\fet(\scr S) = \Span(\SmQA_{\scr S}, \fet, \all) \] the category of finite étale correspondences \cite[Appendix C]{bachmann-norms}.
The functor $\SmQA_{\scr S} \to \Cor^\fet(\scr S)$ preserves finite coproducts and hence induces a colimit preserving functor \cite[Proposition 5.5.8.10]{lurie-htt} \[ F: \PSh_\Sigma(\SmQA_{\scr S}) \to \PSh_\Sigma(\Cor^\fet(\scr S)). \]
Its right adjoint $U$ preserves sifted colimits \cite[Proposition 5.5.8.10(4)]{lurie-htt}.

\begin{definition}
The category of \emph{motivic spaces with finite étale transfers over $\scr S$} is given by \[ \Spc^\fet(\scr S) = L_\mot \PSh_\Sigma(\Cor^\fet(\scr S)). \]
\end{definition}

The localization $L_\mot$ refers to the localization at maps of the form $F(\alpha)$, where $\alpha \in \PSh_\Sigma(\SmQA_{\scr S})$ is a (generating) motivic equivalence.
In particular the functor $F: \PSh_\Sigma(\SmQA_{\scr S}) \to \PSh_\Sigma(\Cor^\fet(\scr S))$ preserves motivic equivalences and induces $F: \Spc(\scr S) \to \Spc^\fet(\scr S)$.

\begin{remark} \label{rmk:compute-F}
We shall often make use of the following fact.
If $X \in \PSh_\Sigma(\SmQA_{\scr S})$ and $Y \in \SmQA_{\scr S}$, then \[ (FX)(Y) \wequi \colim_{U \in \FEt_Y} X(U). \]
Indeed this holds for $X \in \SmQA_{\scr S}$, essentially by definition of $\Cor^\fet$, and extends to general $X$ since $F$ preserves (sifted) colimits.\NB{better justification?}
\end{remark}

\subsubsection{}
Here is a more invariant way of constructing the category of spaces with finite étale transfers.
\begin{lemma} \label{lemm:Spcfet-NMon}
There is a canonical equivalence \[ \NMon(\Spc(\scr S)) \wequi \Spc^\fet(\scr S). \]
\end{lemma}
\begin{proof}
This follows from Lemma \ref{lemm:Spcfet-NMon-abstract}.
\end{proof}

In the same spirit, we can determine a formula for the free space with transfers.
\begin{lemma} \label{lemm:identify-free-normed-presheaf}
For $X \in \PSh_\Sigma(\SmQA_{\scr S})$ we have \[ UF X \wequi \coprod_{n \ge 0} \D_n X. \]
\end{lemma}
\begin{proof}
This is a special case of Theorem \ref{thm:free-space-with-transfers}.
\end{proof}

We deduce the following.
\begin{lemma} \label{lemm:detect-motivic-equiv}
The functor $U: \PSh_\Sigma(\Cor^\fet(\scr S)) \to \PSh_\Sigma(\SmQA_{\scr S})$ preserves and detects motivic equivalences.
\end{lemma}
\begin{proof}
We first show that $U$ preserves motivic equivalences.
Using \cite[Lemma 2.10]{bachmann-norms} it suffices to prove that if $\alpha$ is a generating Nisnevich or homotopy equivalence in $\PSh_\Sigma(\SmQA_{\scr S})$ and $X \in \SmQA_S$ then $U(F(\alpha \coprod \id_X))$ is a motivic equivalence.
If $\alpha$ is one of the generating equivalences then so is $\alpha \coprod \id_X$; hence it suffices to show that $UF(\alpha)$ is a motivic equivalence.
By Lemma \ref{lemm:identify-free-normed-presheaf} we have $UF(\alpha) \wequi \coprod_n \D_n(\alpha)$, so this follows from the fact that $\D_n$ preserves motivic equivalences (by Corollary \ref{cor:Dn-mot-equiv}).
\NB{Alternative proof in Lemma \ref{lemm:Nis-tpet}.}

Detection of motivic equivalences follows by standard arguments.\NB{$U$ preserves motivic equivalences and motivically local objects, so preserves motivic localizations, and is conservative, and a map is a motivic equivalence iff its motivic localization is an equivalence}
\end{proof}

\begin{corollary} \label{lemm:U-sifted-colim}
The functor $U: \Spc^\fet(\scr S) \to \Spc(\scr S)$ is conservative and preserves sifted colimits.
\end{corollary}
\begin{proof}
Immediate from Lemma \ref{lemm:detect-motivic-equiv} and the fact that $U: \PSh_\Sigma(\Cor^\fet(\scr S)) \to \PSh_\Sigma(\Sm_S)$ preserves sifted colimits.
\end{proof}

\begin{corollary} \label{cor:identify-free-normed-space}
For $X \in \Spc(\scr S)$ we have \[ UF X \wequi \coprod_{n \ge 0} \D_n X. \]
\end{corollary}
\begin{proof}
Immediate from Lemmas \ref{lemm:identify-free-normed-presheaf} and \ref{lemm:detect-motivic-equiv}.
\end{proof}

\begin{corollary} \label{cor:p*-fet}
Let $p: \scr X \to \scr S$ be tame proper étale.
Then \[ p_*: \PSh_\Sigma(\Cor^\fet(\SmQA_{\scr X})) \to \PSh_\Sigma(\Cor^\fet(\SmQA_{\scr X})) \] preserves motivic equivalences and the functor \[ p_*: \Spc^\fet(\scr X) \to \Spc^\fet(\scr S) \] preserves colimits.
\end{corollary}
\begin{proof}
Since $p_*$ commutes with $U$, the claim about preservation of motivic equivalences follows from Lemma \ref{lemm:detect-motivic-equiv} and Propositions \ref{prop:norms-equiv}, \ref{prop:fixed-equiv}.
The preservation of colimits thus reduces to the same claim about $p_*$ on $\PSh_\Sigma(\dots)$.
Being a right adjoint, this functor preserves finite products; and by construction it preserves sifted colimits.
The result follows by semiadditivity (see Lemma \ref{lemm:Spc-fet-semiadditive}) and \cite[Lemma 2.8]{bachmann-norms}.
\end{proof}

\subsubsection{}
Since finite étale morphisms are preserved by Weil restriction along tame proper étale morphisms (see Corollary \ref{cor:weil-res}, noting that finite étale morphisms between quasi-affines are the same as proper étale morphisms), we can lift $\SmQA^\otimes$ to $\scr M\Cat$ (the category of $\infty$-categories with a class of marked arrows) marking the finite étale morphisms, and then apply $\Span(\ph, \mathrm{marked}, \all)$ (see e.g. \cite[\S C.3]{bachmann-norms}) sectionwise to obtain the functor \[ \Cor^{\fet\otimes}: \Span(\Stk, \all, \tpet) \to \Cat, \scr S \mapsto \Cor^\fet(\Sm_{\scr S}). \]
Since clopen immersions are also preserved by Weil restriction (see again Corollary \ref{cor:weil-res}), in $\Cor^{\fet\otimes}$ we may sectionwise pass to the wide subcategory where all the backwards morphisms are required to be clopen immersions.
Since $\Span(\SmQA_{\scr S}, \mathrm{clopen}, \all) \wequi \SmQA_{\scr S+}$ (where $\SmQA_{\scr S+} \subset \SmQA_{\scr S*}$ is the full subcategory on objects of the form $\scr X \amalg \scr S$) we denote the resulting functor by $\SmQA_+^\otimes$.
Arguing slightly more carefully, we in fact obtain a sequence of natural transformations \[ \SmQA^\otimes \to \SmQA_+^\otimes \to \Cor^{\fet\otimes}: \Span(\Stk, \all, \tpet) \to \Cat. \]
Note that motivic localization on $\PSh_\Sigma(\SmQA_{\scr S+})$ (respectively $\PSh_\Sigma(\Cor^\fet_{\scr S})$) refers to localization at maps of the form $\alpha_+$ (respectively $F(\alpha)$) where $\alpha$ is a motivic equivalence in $\PSh_\Sigma(\SmQA_{\scr S})$.
It follows that the norms on $\PSh_\Sigma(\SmQA_{\scr S+})$ (respectively $\PSh_\Sigma(\Cor^\fet_{\scr S})$) automatically preserve motivic equivalences (since they do for $\PSh_\Sigma(\SmQA_{\scr S})$).
Thus we can apply $\PSh_\Sigma$ and perform motivic localization sectionwise to obtain  \begin{equation}\label{eq:F*-first} \Spc^\otimes \to \Spc_*^\otimes \xrightarrow{F_*} \Spc^{\fet\otimes}: \Span(\Stk, \all, \tpet) \to \Cat. \end{equation}
By \cite[Proposition C.9]{bachmann-norms}, this construction lifts uniquely to $\CAlg(\Cat)$.

\begin{remark}
If $p: \scr X \to \scr Y \in \tpffqf$ is a gerbe then we similarly obtain \[ p_\otimes: \Spc(\scr X)_* \to \Spc(\scr Y)_* \] and \[p_\otimes: \Spc^\fet(\scr X) \to \Spc^\fet(\scr Y). \]
\end{remark}

\subsubsection{}
If $p \in \tpffqf$ is a gerbe, the description of $p_\otimes$ for pointed spaces simplifies.
\begin{lemma} \label{lemm:pointed-fixed-points}
Let $p: \scr X \to \scr Y \in \tpffqf$ be a gerbe.
Then $p_\otimes \wequi p_*: \Spc(\scr X)_* \to \Spc(\scr Y)_*$, and both functors preserve colimits.
In particular the following square commutes
\begin{equation*}
\begin{CD}
\Spc(\scr X) @<<< \Spc(\scr X)_* \\
@V{p_\otimes}VV    @V{p_\otimes}VV \\
\Spc(\scr Y) @<<< \Spc(\scr Y)_*.
\end{CD}
\end{equation*}
\end{lemma}
\begin{proof}
The functor $p_\otimes: \Spc(\scr X) \to \Spc(\scr Y)$ preserves colimits and is right adjoint to $p^*$.
It follows from Lemma \ref{lemm:adj-preserved} that there is an induced adjunction $\Spc(\scr Y)_* \adj \Spc(\scr X)_*$, with left adjoint induced by $p^*$ and right adjoint given by $p_\otimes'$, which is just the unpointed $p_\otimes$ with the canonical pointing.
The induced left adjoint is the usual $p^*$, and so it remains to verify that $p_\otimes'$ is the usual pointed $p_\otimes$.
To see this, it suffices to prove that (1) the pointed $p_\otimes$ preserves colimits and that (2) the two functors coincide on $\SmQA_{\scr X+}$

(1)
By construction it preserves sifted colimits, so it suffices to show that it preserves finite coproducts of the generators $\scr T_+$, which it sends to $p_\otimes(\scr T)_+$.
This follows from the fact that $p_\otimes: \Spc(\scr X) \to \Spc(\scr Y)$ preserves colimits (by Lemma \ref{lemm:p-otimes-colim}).

(2)
This follows from the fact that unpointed $p_\otimes$ preserves coproducts, and so $p_\otimes'(T_+) \wequi p_\otimes(T) \amalg p_\otimes(*) \wequi p_\otimes(T_+)$.
\end{proof}

\subsection{Stabilization} \label{subsec:stab-norms}
\begin{lemma} \label{lemm:norm-Trho}
Let $f: \scr S \to \scr S'$ be a proper étale morphism (or $\tpffqf$ gerbe) in $\Stk$, and $V$ a vector bundle on $\scr S$.
Then \[ f_\otimes(Th(V)) \wequi Th(f_* V). \]
\end{lemma}
\begin{proof}
If $f$ is finite étale, the proof of \cite[Lemma 4.4]{bachmann-norms} goes through unchanged.\tdn{really?}

If $f$ is a gerbe, we have \[ f_\otimes(Th(V)) = f_\otimes(V/V \setminus 0) \wequi f_\otimes(V)/f_\otimes(V \setminus 0) \] by Lemma \ref{lemm:pointed-fixed-points}, and the latter expression is equivalent to $Th(f_\otimes(V))$ since $f_\otimes(V \setminus 0) \wequi f_\otimes(V) \setminus 0$ (as may be checked fppf locally, i.e. for taking ordinary fixed points, where it is obvious).
\end{proof}

\begin{corollary} \label{cor:SH-sph-normed}
Let $\scr S$ be a stack.
Denote by $\Sph_{\scr S} = \{Th(V)\} \subset \Spc(\scr S)_*$ the set of Thom spaces on vector bundles over $\scr S$.
There is a normed functor \[ \Span(\Stk, \all, \tpet) \to \Cat, \scr X \mapsto \Spc(\scr X)_*[\Sph_{\scr X}^{-1}]. \]
\end{corollary}
\begin{proof}
The assignment $\scr X \mapsto (\Spc(\scr X)_*, \Sph_{\scr X})$ defines a functor $\Span(\Stk_{\scr S}, \all, \tpet) \to \scr O\Cat$ (by Lemma \ref{lemm:norm-Trho}), and hence we can invert the spheres sectionswise.
\end{proof}

Unfortunately unless $\scr S$ is quasi-fundamental, $\Spc(\scr S)_*[\Sph_{\scr S}^{-1}]$ need not really be reasonable.
We remedy this now in the scalloped case.

\begin{lemma} \label{lemm:Nis-tpet}
The forgetful functor \[ \PSh_\Sigma(\Span(\Stk, \tpet, \all)) \to \PSh_\Sigma(\Stk) \] preserves (quasi-affine) Nisnevich equivalences.
\end{lemma}
\begin{proof}
Write \[ F: \PSh_\Sigma(\Stk) \adj \PSh_\Sigma(\Span(\Stk, \tpet, \all)): U \] for the adjunction.
Using \cite[Lemma 2.10]{bachmann-norms} it will be enough to show that if $Y \to X \in \Stk$ is a quasi-affine Nisnevich covering with associated sieve $Y_0 \hookrightarrow X$, then $UF Y_0 \to UFX$ is a Nisnevich equivalence.
One checks that for $T \in \Stk$, $(UF Y)(T)$ is the space of spans \[ T \xleftarrow{p} V \to X \] with $p$ tame proper étale, and $(UF Y_0)(T)$ is the subspace of those spans where $U \to X$ factors through $Y_0$.\todo{really?}
It will thus suffice to show that given arbitrary $V \to X$, there exists a quasi-affine Nisnevich covering $T' \to T$ such that if $V' = T' \times_T V$, then the composite $V' \to V \to X$ lifts to $Y$.
To do this, set $Y' = V \times_X Y$; thus we must lift $V' \to V$ into $Y'$.
Set $T' = p_\otimes(Y')$.
Then $T' \to T$ is a Nisnevich covering by Remark \ref{rmk:p*-pres-coverings}.
But now $V' \wequi p^*p_\otimes Y' = p^*p_* Y'$, which admits a map to $Y'$ by adjunction.
\end{proof}

\begin{corollary} \label{cor:shv-norms}
The forgetful functor \[ \Fun^\times(\Span(\Stk, \tpet, \all)^\op, \Cat_\infty) \to \Fun^\times(\Stk^\op, \Cat_\infty) \] commutes with quasi-affine Nisnevich sheafification.
\end{corollary}
\begin{proof}
For $\Spc$ in place of $\Cat_\infty$, this is Lemma \ref{lemm:Nis-tpet}.
This implies also the case with $\PSh(\Delta)$ in place of $\Cat_\infty$.
Using the complete Segal space presentation $\Cat_\infty \subset \PSh(\Delta)$, it will suffice to show that given a presheaf of complete Segal spaces $F \in \Fun(\Delta^\op, \PSh_\Sigma(\Stk))$, the associated Nisnevich sheaf also consists of complete Segal spaces.
This is true because the sheafification can be built using limits and filtered colimits \cite[proof of Proposition 6.2.2.7]{lurie-htt}, and complete Segal spaces are closed under these operations.\tdn{Ref?}
\end{proof}

Now consider the presheaf of categories $\Spc(\ph)_*[\Sph_{\ph}^{-1}]^\otimes$ on $\Span(\Stk, \tpet, \all)$ obtained from Corollary \ref{cor:SH-sph-normed}.
By Corollary \ref{cor:shv-norms} we can take the associated Nisnevich sheaf, and this operation is compatible with forgetting the norms.
We denote the associated Nisnevich sheaves by \[ \SH^\otimes(\ph): \Span(\Stk, \all, \tpet) \to \Cat_\infty \quad\text{and}\quad \SH(\ph): \Stk^\op \to \Cat_\infty. \]
This satisfies the following properties:
\begin{enumerate}
\item $\SH(\ph) \wequi \SH^\otimes(\ph)|_{\Stk^\op}$
\item If $\scr S$ is linearly scalloped, or nicely scalloped with affine diagonal, then $\SH(\scr S)$ coincides with the category defined in \cite[\S A, \S4,]{khan2021generalized}.
\end{enumerate}
\begin{proof}
Only the second statement is non-obvious.
Either class of stacks is closed under quasi-affine extension, so the restriction of the (quasi-affine) Nisnevich topology thereto makes sense.
Quasi-fundamental stacks provide a basis of the topology, and the two functors agree thereon by construction.
The result follows since the functors defined in \cite[\S A, \S4,]{khan2021generalized} are sheaves.
\end{proof}

\begin{remark}
If $\scr S$ is not linearly scalloped (or nicely scalloped with affine stabilizer), then it is unclear how to describe $\SH(\scr S)$ as defined above, or if this definition is even reasonable.
\end{remark}

\subsection{The free normed monoid on a pointed space} \label{subsec:free-normed}
\begin{lemma} \label{lemm:transfers-to-NAlg}
Let $\scr D^\otimes$ be a normed category over $\scr S$ and $\sigma: \Spc^\otimes \to \scr D^\otimes$ a normed functor.
Assume that $\sigma$ sectionwise admits a right adjoint $\omega$ and that $\omega$ commutes with pullback along smooth quasi-affine morphisms.

Let $X \in \Spc^\fet(\scr S)$.
Then $\sigma UX \in \scr D(\scr S)$ naturally defines an object of $\NAlg(\scr D^\otimes)$.
This induces an adjunction \[ \sigma: \Spc^\fet(\scr S) \adj \NAlg(\scr D^\otimes): \omega, \] where both adjoints are computed sectionwise as indicated.
In particular, the following diagrams commute.
\begin{equation*}
\begin{CD}
\Spc^\fet(\scr S) @>{\sigma}>> \NAlg(\scr D^\otimes) \quad @. \Spc^\fet(\scr S) @>{\sigma}>> \NAlg(\scr D^\otimes) \\
@V{U}VV                                 @V{U}VV                         @A{F}AA                                 @A{\NSym}AA    \\
\Spc(\scr S)      @>{\sigma}>> \scr D(\scr S)              @. \Spc(\scr S)      @>{\sigma}>> \scr D(\scr S)
\end{CD}
\end{equation*}
\end{lemma}
\begin{proof}
This is a special case of Lemma \ref{lemm:NAlg-adj}.
\end{proof}

Since $\Spc^\fet(\scr S)$ is pointed, the functor $U: \Spc^\fet(\scr S) \to \Spc(\scr S)$ factors through a functor $U_*: \Spc^\fet(\scr S) \to \Spc(\scr S)_*$.
The functor $U_*$ still preserves limits and sifted colimits, and hence admits a left adjoint $F_*$.
In fact we have already seen this functor in \eqref{eq:F*-first}; in particular it is compatible with norms.

\begin{lemma} \label{lemm:F*-pushout}
Let $X \in \Spc(\scr S)_*$.
Then there exists a pushout diagram in $\Spc^\fet(\scr S)$ as follows
\begin{equation*}
\begin{CD}
\1 @>>> F(X) \\
@VVV    @VVV   \\
*  @>>> F_* X.
\end{CD}
\end{equation*}
\end{lemma}
\begin{proof}
Apply the colimit preserving functor $F_*$ to the pushout diagram
\begin{equation*}
\begin{CD}
*_+ @>>> X_+  \\
@VVV     @VVV \\
\emptyset_+ @>>> X
\end{CD}
\end{equation*}
in $\Spc(\scr S)_*$ and use that $F_*((\ph)_+) \wequi F(\ph)$ (and $F\emptyset = *$ since $\Spc^\fet(\scr S)$ is pointed).
\end{proof}

Suppose given $\sigma: \Spc(\scr S) \to \scr D(\scr S)$.
If $\scr D(\scr S)$ is pointed and presentable, then there is a unique functor $\sigma_*: \Spc(\scr S)_* \to \scr D(\scr S)$ satisfying $\sigma_*(X_+) \wequi \sigma(X)$ \cite[Proposition 4.8.2.11]{lurie-ha}.\NB{It is given by $\sigma_*(X) = \cof(\1 = \sigma(*) \to \sigma(UX))$.}
\begin{corollary} \label{cor:F*-stable}
Let \[ \sigma: \Spc^\otimes \adj \scr D^\otimes: \omega \] be as in Lemma \ref{lemm:transfers-to-NAlg}.
Assume furthermore that $\scr D(\scr S)$ is stable and $\scr D$ satisfies distributivity and smooth base change.

Let $X \in \Spc(\scr S)_*$.
Then there is a canonical equivalence \[ \sigma F_* X \wequi \NSym \sigma_* X \in \NAlg(\scr D^\otimes). \]
In particular \[ \sigma_* U_* F_* X \wequi \bigvee_{n > 0} \D_n \Sigma^\infty X. \]
\end{corollary}
\begin{proof}
The cofiber sequence $\1 \to \sigma X \to \sigma_* X$ is split by ($\sigma$ applied to) the projection $X \to *$; it follows (using stability) that $\sigma X \wequi \sigma_* X \amalg \1$.
Applying $\NSym$ to this coproduct decomposition, we obtain the top square of the following diagram
\begin{equation*}
\begin{CD}
\1        @>>> \NSym(\sigma_* X) \\
@VVV             @VVV                   \\
\NSym(\1) @>>> \NSym(\sigma X) \\
@VVV            @VVV \\
\1        @>>>  \sigma F_* X.
\end{CD}
\end{equation*}
The bottom square is obtained by applying $\sigma$ to the pushout of Lemma \ref{lemm:F*-pushout}.
Thus both the top and the bottom square are pushouts (in $\NAlg(\scr D^\otimes)$), and hence so is the outer square \cite[Lemma 4.4.2.1]{lurie-htt}.
Since $\1$ is initial in $\NAlg(\scr D^\otimes)$, the left hand vertical composite is an equivalence, and hence so is the right hand one.
The first claim follows.

To deduce the second claim we use that by Lemma \ref{lemm:transfers-to-NAlg} and Proposition \ref{prop:free-normed}(4) we have \[ \sigma_* U_* F_* X \wequi \cof(\1 \to \sigma U F_* X) \wequi \cof(\1 \to U\NSym \sigma X) \wequi \cof(\1 \to \bigvee_{n \ge 0} \D_n \sigma X), \] and $\D_0 \sigma X \wequi \1$.
\end{proof}

\begin{example}
We may apply the results of this section to the stabilization functor \[ \sigma = \Sigma^\infty_+: \Spc^\otimes(\ph) \to \SH^\otimes(\ph) \] from \S\ref{subsec:stab-norms}.
We have $\sigma_* = \Sigma^\infty$.
\end{example}

\section{Complements on motivic homotopy theory of stacks} \label{sec:complements}
\subsection{Orbit stacks}
Fix a proper étale morphism $p: \scr S \to \scr B$.
For $\scr Y \in \Stk_{\scr B}$ we put $p^* \scr Y = \scr Y \times_{\scr B} \scr S$.
Given $\scr T \to p^* \scr Y$ finite étale, the composite $\scr T \to p^* \scr Y \to \scr Y$ is proper étale, and hence by Proposition \ref{prop:relative-quotient-gerbe} there exists an initial stack $\repr_{\scr Y}(\scr T)$ under $\scr T$ representable over $\scr Y$.
We put $p_!(\scr T) := \repr_{\scr Y}(\scr T)$.
Define a presheaf of $1$-groupoids $\Orb_p$ on $\Stk_{\scr B}$ by \[ \Orb_p(\scr Y) = \{ \alpha: \scr T \to p^* \scr Y \mid \alpha \text{ finite étale, } p_!(\scr T) \wequi p_!(p^* \scr Y) \}. \]
Since $\scr S \to \scr B$ is proper étale, the diagonal $\Delta: \scr S \to \scr S \times_{\scr B} \scr S$ is proper\NB{proper implies separated means proper diagonal}, étale \cite[Tag 0CJ1]{stacks-project} and representable, so finite étale \cite[Tags 02LR, 02LS]{stacks-project}.
The composite $\scr S \xrightarrow{\Delta} \scr S \times_{\scr B} \scr S \wequi p^* \scr S \to \scr S$ (projection to either factor) is the identity and hence representable; thus $p_!(\Delta) = \scr S$ and so $\Delta \in \Orb_p(\scr S)$.
Given $\scr Y \to \scr S$, we can pull back $\Delta$ along $p^* \scr Y \to p^* \scr S$ to obtain $\Delta_{\scr Y} \in \Orb_p(\scr Y)$; explicitly $\Delta_{\scr Y} = \scr Y \to \scr Y \times_{\scr B} \scr S$ is the identity on the first factor and the canonical morphism on the second.
Define a further presheaf of $1$-groupoids $\widetilde{\Orb}_p$, this time on $\Stk_{\scr S}$, by \[ \widetilde{\Orb}_p(\scr Y \to \scr S) = \{\Delta_{\scr Y} \twoheadrightarrow \scr T \mid \scr T \in \Orb_p(\scr Y) \}. \]
Here by an isomorphism of surjections we mean a commutative diagram
\begin{equation*}
\begin{CD}
\Delta_{\scr Y} @>{\wequi}>> \Delta_{\scr Y} \\
@VVV                            @VVV \\
\scr T @>{\wequi}>> \scr T';
\end{CD}
\end{equation*}
note that since the vertical maps are surjective the bottom map (if it exists) is completely determined by the top map.
\begin{definition}
Let $p: \scr S \to \scr B$ be a proper étale morphism.
The above construction yields a span of presheaves over $\scr B$ \[ \scr S \xleftarrow{p_1} \widetilde\Orb_p \xrightarrow{p_2} \Orb_p, \] which we call the \emph{orbit stack diagram associated with $p$}.
\end{definition}

\begin{remark} \label{rmk:orbit-stack-bc}
These constructions are stable under base change; in particular given a cartesian square
\begin{equation*}
\begin{CD}
\scr Y' @>>> \scr S \\
@V{p'}VV      @V{p}VV \\
\scr Y @>>> \scr B
\end{CD}
\end{equation*}
we have \[ \Orb_{p'}(\scr Y) \wequi \Orb_p(\scr Y) \text{ and } \Orb_{p'}(\scr Y'' \to \scr Y') \wequi \Orb_p(\scr Y'' \to \scr S). \]
\end{remark}

\begin{example} \label{ex:orbits-BG-1}
Let $p: \scr S = B \sslash G \to B$, where $G$ is a finite constant group, and let $\scr Y = S$ be a strictly henselian local scheme.
A finite étale morphism $\scr T \to p^* S$ is a finite étale $S$-scheme together with a $G$-action.
Such a scheme non-equivariantly splits into a finite disjoint union of copies of $S$, and hence equivariantly splits into a finite disjoint union of $G$-orbits $G/H \times S$.
We will have $\scr T \in \Orb_p(\scr B)$ if and only if the number of orbits is $1$.
In other words $\Orb_p(S)$ is equivalent to the core of the orbit category of $G$.
\end{example}

\begin{example} \label{ex:orbits-BG-2}
Continuing with the previous example, suppose given instead a morphism $S \to B \sslash G$, i.e. a $G$-torsor $U$ on $S$.
Then $\Delta_{S} = (S \to S \times B \sslash G) = (U \to S) \sslash G$.
Since $S$ was strictly henselian local, actually $U=G \times S$ and so $\widetilde\Orb_p(S)$ is the core of the subcategory of $\Fin_{G/}$ on the $G$-orbits.
Any $G$-orbit under $G$ has the form $G \to G/H$, and the automorphisms of $G/H$ under $G$ (in the above sense) are given by $NH$.
\end{example}

We amplify the above two examples.
Fix a finite group $G$ and a base stack $\scr B$.
Given a subgroup $H \subset G$ and a $WH$-torsor $A$ (over $\scr B$), we obtain a finite étale $G$-scheme over $\scr B$ given by $A \times_{WH} G/H$.
Given instead a $NH$-torsor $A'$, we can form $A' \times_{NH} G \to A' \times_{NH} G/H$.
\begin{lemma} \label{lemm:orbits-BG}
The above constructions induce equivalences \[ \coprod_{(H) \subset G} \scr B \sslash WH \xrightarrow{\wequi} \Orb_{\scr B\sslash G/\scr B} \text{ and } \coprod_{(H) \subset G} \scr B \sslash NH \xrightarrow{\wequi} \widetilde\Orb_{\scr B \sslash G/\scr B}; \] here the coproducts are over conjugacy classes of subgroups.
\end{lemma}
\begin{proof}
First observe that $A \times_{WH} G/H$ and $A' \times_{NH} G/H$ are indeed orbits (being locally just given by the orbit $G/H$).
Consequently the maps are well-defined.
Since both sides are $1$-truncated (finitary) fppf sheaves, it suffices to check that the maps induce equivalences when evaluated on strictly henselian local schemes.
This was worked out in Examples \ref{ex:orbits-BG-1} and \ref{ex:orbits-BG-2}.
\end{proof}

\begin{lemma} \label{lemm:orbits-fet}
Let $p: \scr B \to \scr B'$ be finite étale.
Then $\widetilde \Orb_p \wequi \widetilde \Orb_{p' \circ p}$ and $\Orb_p \wequi \Orb_{p' \circ p}$.
\end{lemma}
\begin{proof}
Indeed it is clear by definition\NB{...} that if $p_0: \scr S \to \repr_{\scr B} \scr S$ is the relative quotient, hen $\Orb_p \wequi \Orb_{p_0}$ and similarly for $\widetilde\Orb$.
\end{proof}
In particular, for most problems regarding orbits stacks we may assume that $p$ is a gerbe, and conversely most constructions still make sense for any morphism $\scr S \to \scr B$ with finite étale inertia.

We can now prove that $\Orb_p$ and $\widetilde \Orb_p$ are actually stacks.
\begin{proposition}
The presheaves $\widetilde\Orb_p, \Orb_p$ are algebraic stacks.
The morphism $p_1$ is finite étale, $p_2$ is a gerbe (tame if $p$ is), and $\Orb_p \to \scr B$ is proper étale (tame if $p$ is).
\end{proposition}
\begin{proof}
Using the factorization $\scr S \to \repr_{\scr B} \scr S \to \scr B$ and Lemma \ref{lemm:orbits-fet}, we may assume that $p$ is a proper étale gerbe.
The claims being fppf local, we may assume (using Remark \ref{rmk:pet-gerbe-local-structure}) that $\scr S = B \sslash G$ for some finite group $G$ (reductive over $\scr B=B$ if $p$ is tame).
In this case we are reduced to Lemma \ref{lemm:orbits-BG} (using that subquotients of $G$ are also reductive over $B$ if $G$ is).
\end{proof}

Recall that a morphism of $\scr B$-stacks $\scr A \to \scr B$ is called \emph{fixed-point reflecting} if the induced map $I_{\scr A/\scr B} \to I_{\scr B/\scr B} \times_{\scr B} \scr A$ is an equivalence \cite[Tag 0DU6]{stacks-project}.
We say that $\scr A \in \Stk_{/\scr B}$ has \emph{flat inertia (over $\scr B$)} if $I_{\scr A/\scr B} \to \scr A$ is flat.
One of the main reasons for considering orbit stacks is the following.
\begin{proposition} \label{prop:identify-flat-inertia}
Let $p: \scr S \to \scr B$ be proper étale.\NB{do not seem to need tame...?}
The composite \[ \SmQA_{\Orb_p}[\scr F_\free/\scr B] \xrightarrow{p_2^*} \SmQA_{\widetilde\Orb_p} \xrightarrow{p_{1\sharp}} \SmQA_{\scr S} \] is an equivalence onto the subcategory of stacks with flat inertia (over $\scr B$) and fixed-point reflecting morphisms.
Here we write $\SmQA_{\Orb_p}[\scr F_\free/\scr B] \subset \SmQA_{\Orb_p}[\scr F_\free/\scr B]$ for the full subcategory on stacks with free action, i.e. representable over $\scr B$.
\end{proposition}
\begin{proof}
By Lemmas \ref{lemm:inertia-composition} and \ref{lemm:orbits-fet} we may assume that $p$ is a proper étale gerbe.
Since all our properties are fppf local, we are dealing with fppf sheaves of $1$-categories; we may hence assume that $\scr S = B\sslash G$ where $G$ is a finite constant group, and $B$ is an affine scheme.

We first prove essential surjectivity; thus let $X$ be a smooth quasi-affine $B$-scheme with a $G$-action and flat inertia.
Lemma \ref{lemm:flat-inertia-form} below implies that \[ X \wequi \coprod_{(H) \subset G} G \times_{NH} X(H); \] here the coproduct is over conjugacy classes of subgroups of $G$.
Here we use that \[ \coprod_{H' \text{ conj. to } H} X(H') = \coprod_{gNH \in G/NH} X(H)^g \wequi G \times_{NH} X(H). \]
Note that $H$ acts trivially on $X(H)$, and $NH/H$ acts freely (because all points have isotropy precisely $H$, so no element of $NH \setminus H$ can fix any point).
It hence suffices to show (using Lemma \ref{lemm:orbits-BG}) that if $X$ is a smooth quasi-affine $B$-scheme with action by $NH$, such that $H$ acts trivially and $NH/H$ acts freely, then $X$ is obtained by inflation from a free $WH$-scheme.
This is clear.

Now we prove fully faithfulness onto the subcategory of fixed-point reflecting morphisms.
Thus let $f: X \to Y$ be a morphism of (smooth quasi-affine) $B$-schemes with $G$-action and flat inertia.
If $f$ is fixed-point reflecting, then $f(X(H)) \subset Y(H)$, i.e. $f$ respects the decompositions of $X$ and $Y$.
Note also that any morphism of $G$-schemes both of which have all isotropy conjugate to $H$ is necessarily fixed-point reflecting.
It thus suffices to prove that if $H$ is a subgroup of $G$, then the functor $p_{1\sharp}p_2^*$ from free $WH$-schemes to $G$-schemes is fully faithful.
In other words if $X$ is a free $WH$-scheme (smooth and quasi-affine over $B$), then we seek to prove that $(G \times_{NH} X)^H \wequi X$ (i.e. $p_{2*}p_1^*p_{1\sharp}p_2^* X \wequi X$).
Considering the functor of points, we reduce to the same statement about \emph{sets}, which is easy.\NB{$(g,x) \in G \times_{NH} X$ is fixed by $H$ only if $g \in NH$, i.e. only one component left}
\end{proof}
\begin{lemma} \label{lemm:flat-inertia-form}
Let $G$ be a finite constant group acting on a locally finite type, separated scheme $X$ with flat inertia.
For a subgroup $H \subset G$, write $X(H)$ for the set of points with isotropy precisely $H$.
Then \[ X = \coprod_{H} X(H) \] is a decomposition into disjoint clopen subsets, where the disjoint union runs over all subgroups.
\end{lemma}
\begin{proof}
Consider the cartesian squares (for some $g \in G$)
\begin{equation*}
\begin{CD}
X^g @>>> I @>>> X \\
@VVV @VVV  @V{\Delta}VV \\
\{g\} \times X @>>> G \times X @>>> X \times X.
\end{CD}
\end{equation*}
The composite $\{g \times X\} \to X \times X$ is just the diagonal followed by an automorphism, whence a finitely presented \cite[Tag 0818]{stacks-project} closed immersion.
By assumption $I \to X$ is flat, and by what we just said it is finitely presented, whence open \cite[Tag 01UA]{stacks-project}.
Thus the same follows for $X^g \to X$, which is also closed (by what we just said).
We have thus shown that $X^g \subset X$ is a clopen subset, for any $g \in G$.

The result follows easily from this.
Indeed $X(H)$ is obtained from subschemes of the form $X^g$ by finitely many operations of intersecting and complementing, whence it is a clopen subscheme.
By construction $X(H) \cap X(H') = \emptyset$ whenever $H \ne H'$.
Hence \[ \coprod_H X(H) \subset X \] is a clopen subscheme.
Since every point is in one of the sets $X(H)$, the result follows.
\end{proof}

\subsection{Isotropy specification} \label{subsec:isotropy-spec}
Let $p: \scr S \to \scr B$ be a proper étale morphism.
\begin{definition}
By an \emph{isotropy specification for $p$} we mean a clopen substack of $\Orb_p$.
\end{definition}
\begin{remark}
Since $\widetilde \Orb_p \to \Orb_p$ is a gerbe, their sets of clopen substacks are in bijection.\NB{check fppf locally, i.e. for $B \sslash G$}\tdn{ref?}
\end{remark}

Let $\{B_\alpha \to \scr B\}$ be an fppf cover such that for each $\alpha$, $B_\alpha \times_{\scr B} \scr S \wequi B_\alpha \sslash G_\alpha$ for some finite constant group $G_\alpha$.
(This always exists, by Remark \ref{rmk:pet-gerbe-local-structure}.)
By a \emph{compatible collection of isotropy groups} $\scr F$ we mean for each $\alpha$ a collection $\scr F_\alpha$ of subgroups of $G_\alpha$ which is stable under conjugation and such that whenever $B_\alpha \times_{\scr B} B_\beta \ne \emptyset$, we have $\scr F_\alpha \wequi \scr F_\beta$ under the induced isomorphism $G_\alpha \wequi G_\beta$.
\begin{remark}
There is no canonical isomorphism $B_\alpha \times_{\scr B} \scr S \wequi B\sslash G_\alpha$; equivalently $B_\alpha \sslash G_\alpha$ has automorphisms.
These are given (clopen-locally on $B_\alpha$) by conjugation by elements of $G$.
The assumption that $\scr F_\alpha$ by stable under conjugation thus ensures that the family of subgroups \emph{is} canonical.
\end{remark}
\begin{lemma}
There is a natural injective map \[ \{\text{compatible collections of isotropy groups} \} \hookrightarrow \{\text{isotropy specifications}\}. \]
If each $B_\alpha$ is connected, the map is also surjective.
In general, every isotropy specification arises in this way up to a clopen refinement of the cover.
\end{lemma}
\begin{proof}
Given the collection $\scr F_\alpha$, we obtain the clopen substack \[ \coprod_{(H \in \scr F_\alpha)} B_\alpha \sslash WH \subset \coprod_{(H)} B_\alpha \sslash WH \wequi \Orb_{p \times_{\scr B} B_\alpha}. \]
The compatibility assumption implies that this is a descent datum; hence we have produced an isotropy specification.
Conversely, an isotropy specification yields a clopen substack of \[ \coprod_{(H)} B_\alpha \sslash WH. \]
This must be a disjoint union of clopen substacks of the $B_\alpha \sslash WH$, which correspond to clopen substacks of $B_\alpha$.
Thus up to clopen refinement of $B_\alpha$, we have produced a collection of subgroups of $G_\alpha$ stable under conjugation.
The two operations are inverses of one another, to the extent that they are defined.
\end{proof}

For this reason, we often denote isotropy specifications by $\scr F$, and think of them as collections of subgroups.
\begin{notation}
When thinking of an isotropy specification $\scr F$ as a collection of subgroups, we denote for clarity the corresponding clopen substack of $\Orb_p$ by $\Orb_\scr F$, and its preimage in $\widetilde\Orb_p$ by $\widetilde\Orb_\scr F$.
\end{notation}

\begin{definition} \label{def:family}
We call an isotropy specification $\scr F$ a \emph{family} if each $\scr F_\alpha$ is a family, that is, closed under passage to subgroups.
We call $\scr F$ \emph{thin} if whenever $H_1 \subset H_2$ with $H_1, H_2 \in \scr F_\alpha$, we have $H_1 = H_2$.
We call families $\scr F_1 \subset \scr F_2$ \emph{weakly adjacent} if $\scr F_2 \setminus \scr F_1$ is thin.
\end{definition}

\begin{example} \label{ex:families}
Taking $\scr F_\alpha$ to be the collection of subgroups of size $\le i$, we obtain an isotropy specification which we denote by $\scr F_i/\scr B$.
One may check this is independent of the presentation of $\scr S$.
These are families, and $\scr F_n/\scr B \subset \scr F_{n+1}/\scr B$ are weakly adjacent.
We also write $\scr F_\free/\scr B := \scr F_0/\scr B$.
\end{example}
\begin{example}
Taking $\scr F_\alpha$ to consist of proper subgroups, we obtain an isotropy specification (in fact family) $\scr F_\prop/\scr B$.
\end{example}
\begin{example}
Taking $\scr F_\alpha$ to consist of \emph{all} subgroups, we obtain an isotropy specification (in fact family) $\scr F_\all/\scr B$.
This is mainly notationally useful.
\end{example}

\begin{definition}
Given a family $\scr F$ and $\scr S' \to \scr S$ representable, we say that \emph{$\scr S$ has isotropy in $\scr F$} if whenever $H \not\in \scr F_\alpha$ and $\scr S' \times_{\scr S} B_\alpha \sslash G_\alpha \wequi X' \sslash G_\alpha$, we have $X'^H = \emptyset$.

Given a thin isotropy specification $\scr F$, we say that $\scr S$ has isotropy in $\scr F$ if $X' = \bigcup_{H \in \scr F_\alpha} X'^H$ and whenever $H \in \scr F_\alpha$, $H \subsetneq K$ we have $X'^K = \emptyset$.

In either case we denote by $\SmQA_\scr{S}[\scr F] \subset \SmQA_\scr{S}$ the full subcategory on stacks with isotropy in $\scr F$.
\end{definition}

\begin{remark} \label{rmk:isotropy-sieve}
It is clear that if $\scr F$ is a family, then $\SmQA_\scr{S}[\scr F]$ is a sieve: if $X \to Y \in \SmQA_{\scr S}$ and $Y \in \SmQA_\scr{S}[\scr F]$, then $X \in \SmQA_\scr{S}[\scr F]$.
\end{remark}

\begin{lemma} \label{lemm:define-TF}
Let $\scr F$ be a family.
For $\scr T/\scr S$ representable, there exists a largest open substack $\scr T(\scr F)$ with isotropy contained in $\scr F$.
\end{lemma}
\begin{proof}
Working fppf locally, we may assume that $\scr S = B \sslash G \to B$ and $\scr T = T \sslash G$, where $G$ is a finite constant group.
We can then put \[ T(\scr F) = T \setminus \bigcup_{H \not\in \scr F'} T^H. \]
\end{proof}

\begin{lemma} \label{lemm:adjacent-inertia-flat}
Let $\scr F_1 \subset \scr F_2$ be weakly adjacent families.
Suppose that $\scr S \to \scr B$ is tame.
For $\scr T/\scr S$ schematic, smooth and separated, the complement $\scr T(\scr F_2 \setminus \scr F_1) := \scr T(\scr F_2) \setminus \scr T(\scr F_1)$ admits a canonical scheme structure which is smooth over $\scr S$ and has flat inertia.
\end{lemma}
\begin{proof}
We may assume that $\scr T = \scr T(\scr F_2)$, and we may work locally so that $\scr S = B \sslash G\to B$ and $\scr T = T \sslash G$.
Let $\scr E = \scr F_2' \setminus \scr F_1'$; this is a set of subgroups of $G$.
The complement $T \setminus T(\scr F_1)$ consists of those points with isotropy in $\scr E$; since $T$ has isotropy in $\scr F_2$ this is precisely \[ T \setminus T(\scr F_1) = \bigcup_{H \in \scr E} T^H. \]
Note that if $H_1 \ne H_2 \in \scr E$ then $T^{H_1} \cap T^{H_2} = \emptyset$.
Indeed if $x \in T^{H_1} \cap T^{H_2}$ then $x$ must have isotropy group $H$ containing both $H_1$ and $H_2$; then $H_i \subsetneq H$ and all of these groups are in $\scr E$, contradicting our adjacency assumption.
We deduce that \[ T \setminus T(\scr F_1) = \coprod_{H \in \scr E} T^H; \] this is smooth (e.g. by Proposition \ref{prop:fixed-points-functor}) and the isotropy above $T^H$ is constantly $H$, so that we have flat inertia.
\end{proof}

\begin{lemma} \label{lemm:adjacent-splitting}
Let $p: \scr S \to \scr B$ be a tame proper étale gerbe.
Let $\scr F_1 \subset \scr F_2$ be weakly adjacent families.
Let $\scr T/\scr S$ be schematic and smooth, with $\scr T = \scr T(\scr F_2 \setminus \scr F_1)$, and $V \to \scr T$ a vector bundle.
Then there exists a splitting $V \wequi V_1 \oplus V_2$ such that $V_1(\scr F_2 \setminus \scr F_1) = V_1$ and $V_2 \setminus 0$ has isotropy in $\scr F_1$.
\end{lemma}
\begin{proof}
Provided we construct the splitting canonically we may work fppf locally; thus we may assume that $\scr B = B$ and $\scr S = B \sslash G$, for a finite constant tame group $G$ over $B$.
By Lemma \ref{lemm:adjacent-inertia-flat} and Proposition \ref{prop:identify-flat-inertia}, we may assume that there exists a subgroup $H \in \scr F_2 \setminus \scr F_1$ such that $T = G \times_{NH} T'$ for some $NH$-scheme $T'$ on which $H$ acts trivially and $WH$ acts freely.
The vector bundle $V$ is then also induced; hence we may replace $G$ by $NH$ and so assume that $T=T'$ and $H$ is normal in $G$.
We put $V_1 = V^H$.
This still has a $G$-action since $H$ is normal, and there is a canonical complement $V_2$ by the usual averaging procedure (using crucially that $G$, and hence $H$, is tame).
By construction, all isotropy groups of $V_1$ contain $H$, whence $V_1(\scr F_2 \setminus \scr F_1) = V_1$.
Similarly by construction $V_2^H = \{0\}$, whence $(V_2 \setminus 0)^H = \emptyset$, so that all isotropy of $V_2 \setminus 0$ is properly contained in $H$.
Since $H \in \scr F_2$ and $\scr F_1, \scr F_2$ are weakly adjacent, we deduce that all isotropy of $V_2 \setminus 0$ is in $\scr F_1$, as needed.
\end{proof}

\subsection{Filtration by isotropy}
Let again $p: \scr S \to \scr B$ be a proper étale morphism.

Given a family $\scr F$, we obtain an adjunction \[ i_!: \PSh_\Sigma(\SmQA_{\scr S}[\scr F]) \adj \PSh_\Sigma(\SmQA_{\scr S}): i^*. \]
\begin{definition}
We put \[ \E\scr F := i_!(*) \in \PSh_\Sigma(\SmQA_{\scr S}). \]
Given a second family $\scr F' \subset \scr F$, we put \[ \E(\scr F', \scr F) := \cof(\E\scr F' \to \E\scr F) \in \PSh_\Sigma(\SmQA_{\scr S})_*. \]
We also put \[ \widetilde{\E\scr F} = \E(\scr F, \scr F_\all/\scr B). \]
\end{definition}

Since $\SmQA_{\scr S}[\scr F]$ is a sieve (Remark \ref{rmk:isotropy-sieve}), the above constructions have some standard properties which we summarize now.
\begin{remark} \label{rmk:EE-interpretation}
Explicitly, $\E(\scr F)$ is the universal space with isotropy in $\scr F$; in other words $(\E \scr F)(\scr T)$ takes value $*$ if $\scr T$ has isotropy in $\scr F$, and value $\emptyset$ else.
In particular if $\scr F_1 \subset \scr F_2$ then \[ \E(\scr F_1) \times \E(\scr F_2) \xrightarrow{\wequi} \E(\scr F_1). \]
Also \[ \E(\scr F) \wequi \colim_{\scr X \in \SmQA_{\scr S}[\scr F]} \scr X. \]
\end{remark}
\begin{remark} \label{rmk:EFF-interpretation}
Smashing the cofiber sequence $\E(\scr F_1)_+ \to S^0 \to \widetilde{\E\scr F_1}$ with $\E(\scr F_2)_+$ and applying Remark \ref{rmk:EE-interpretation}, we deduce that \[ \E(\scr F_1, \scr F_2) \wequi \E(\scr F_2)_+ \wedge \widetilde{\E\scr F_1}. \]
In particular $\E(\scr F)_+ \wedge \widetilde{\E\scr F} = *$.
\end{remark}

\begin{remark} \label{rmk:Etilde-monoid}
Smashing the cofiber sequence $\E(\scr F_1)_+ \to S^0 \to \widetilde{\E\scr F_1}$ with $\widetilde{\E\scr F_2}$ and using that $\E(\scr F_1)_+ \wedge \widetilde{\E\scr F_2} \wequi \E(\scr F_1)_+ \wedge \E(\scr F_2)_+ \wedge \widetilde{\E\scr F_2} \wequi *$ (using Remarks \ref{rmk:EE-interpretation} and \ref{rmk:EFF-interpretation}) we see that \[ \widetilde{\E\scr F_2} \xrightarrow{\wequi} \widetilde{\E\scr F_1} \wedge \widetilde{\E\scr F_2}. \]
\end{remark}

The category $\PSh_\Sigma(\SmQA_\scr{S}[\scr F])$ has natural notions of Nisnevich, $\A^1$-and motivic equivalence, generated by respectively Nisnevich covering sieves of objects in $\SmQA_\scr{S}[\scr F]$, vector bundle torsors over objects in $\SmQA_\scr{S}[\scr F]$, or both.
Note that it is crucial here that $\scr F$ is a family, so that e.g. $\A^1 \times \scr X \in \SmQA_\scr{S}[\scr F]$ for $\scr X \in \SmQA_\scr{S}[\scr F]$.\NB{More generally we need closure under vector bundle torsors, which are not isovariant in general.}
\begin{lemma} \label{lemm:family-inclusion-pres-mot-eq}
Given a family $\scr F$, both of the functors $i_!$ and $i^*$ preserve Nisnevich, homotopy and motivic equivalences.
\end{lemma}
\begin{proof}
For $i_!$, since the functor preserves colimits and generating equivalences, this is clear.
For $i^*$, we can use the fact that if $L: \PSh_\Sigma(\SmQA_{\scr T}) \to \PSh_\Sigma(\SmQA_{\scr T})$ denotes one of the localizations, then $(LF)(\scr U)$ can be written as a colimit of $F$ evaluated on certain spaces $L$-equivalent to $\scr U$; this is clear for Nisnevich sheafification, and for homotopy localization this is \cite[Proposition 3.4(1)]{hoyois-equivariant}.
The same formula works for $\PSh(\SmQA_{\scr T}[\scr F])$, whence $i^*$ commutes with $L$, and the result follows.
\end{proof}

We put \[ \Spc(\scr S)[\scr F] := L_\mot \PSh_\Sigma(\SmQA_\scr{S}[\scr F]). \]
\begin{corollary} \label{cor:i!-ff}
There is an induced adjunction \[ i_!: \Spc(\scr S)[\scr F] \adj \Spc(\scr S): i^*. \]
The functor $i_!$ is fully faithful and identifies the source with the full subcategory of the target generated by $\SmQA_\scr{S}[\scr F]$.
The functor $i^*$ preserves colimits.
\end{corollary}
\begin{proof}
Immediate.
\end{proof}

\begin{lemma} \label{lemm:i!i*-ident}
The endofunctor $i_!i^*: \PSh_\Sigma(\SmQA_{\scr S}) \to \PSh_\Sigma(\SmQA_{\scr S})$ is equivalent to $\times \E\scr F$.
\end{lemma}
\begin{proof}
We have $\E \scr F \wequi \colim_{X \in \SmQA_{\scr S}[\scr F]} X$, whence for $Y \in \PSh(\SmQA_{\scr S})$ we have \[ \E \scr F \times Y \wequi \colim_{X \in \SmQA_{\scr S}[\scr F]} (X\times Y) \in i_! \PSh(\SmQA_{\scr S}[\scr F]). \]
Here we are using that $\scr F$ is a family, via Remark \ref{rmk:isotropy-sieve}.
It thus suffices to show that $i^*Y \wequi i^*(\E \scr F \times Y)$, which is clear since products of presheaves are computed sectionwise.
\end{proof}

\begin{comment}
\begin{lemma} \label{lemm:EF-trick}
Let $\scr F_1 \subset \scr F_2$ be weakly adjacent families, with $\scr F = \scr F_2 \setminus \scr F_1$ and associated restricted orbit stack diagram \[ \scr S \xleftarrow{a} \widetilde\Orb_\scr{F} \xrightarrow{b} \Orb_{\scr F}. \]
Then \[ a^* \E(\scr F_1, \scr F_2) \wequi \widetilde{\E(\scr F_\prop/\Orb_{\scr F})} \wedge a^* \E(\scr F_1, \scr F_2) \in \PSh_\Sigma(\widetilde\Orb_{\scr F})_*. \]
\end{lemma}
\begin{proof}
XXX this seems false
\end{proof}
\end{comment}

\begin{remark} \label{rmk:orbit-stack-usefulness}
Let $\scr F$ be a thin isotropy specification, and suppose that $p$ is tame.
It follows from Proposition \ref{prop:identify-flat-inertia} that the composite \[ \SmQA_{\Orb_\scr{F}}[\scr F_\free] \xrightarrow{b^*} \SmQA_{\widetilde\Orb_\scr{F}} \xrightarrow{a_{\sharp}} \SmQA_{\scr X} \] is fully faithful onto $\SmQA_{\scr X}[\scr F]$ (indeed any map between stacks with isotropy in $\scr F$ must be fixed-point reflecting, by thinness).
We deduce that the adjunction \[ a_\sharp b^*i_!: \Spc(\Orb_{\scr F})[\scr F_\free] \adj \Spc(\scr X): i^*b_*a^* \] is obtained as a localization of \[ j_!: \PSh_\Sigma(\SmQA_{\scr X}[\scr F]) \adj \PSh_\Sigma(\SmQA_{\scr X}): j^*. \]
Both $j_!$ and $j^*$ preserve the local equivalences, since this is true for $a_\sharp, a^*, b^*$ (easy), $b_*$ (Proposition \ref{prop:fixed-equiv} and Remark \ref{rmk:p-otimes-*-2}) and $i_!, i^*$ (Lemma \ref{lemm:family-inclusion-pres-mot-eq}).
%We will use this identification frequently in the sequel.
\end{remark}

\subsection{Generation by gerbes}
Since stabilization is compatible with localization and (appropriate) adjunctions (see e.g. Lemmas \ref{lemm:locn-preserved} and \ref{lemm:adj-preserved}), in the results from the previous subsection we can everywhere replace $\Spc(\ph)$ by $\SHS(\ph) := \Sp((\Spc(\ph)))$.
We do this freely from now on.

\begin{proposition} \label{prop:gerbes-gen}
Let $p: \scr X \to \scr S$ be a tame proper étale gerbe and $\scr F_1 \subset \scr F_2$ weakly adjacent families.
\reasonableness{\scr X}
Then $\SHS(\scr X)[\scr F_2]$ is generated (under colimits and desuspensions) by $\SHS(\scr X)[\scr F_1]$ and $\SmQA_{\scr X}[\scr F_2 \setminus \scr F_1]$.
\end{proposition}
\begin{proof}
It is clear that $\SHS(\scr X)[\scr F_2]$ is generated by $\SmQA_{\scr X}[\scr F_2]$.
Write $\scr C$ for the subcategory generated by $\SHS(\scr X)[\scr F_1]$ and $\SmQA_{\scr X}[\scr F_2 \setminus \scr F_1]$.
Now let $\scr T \in \SmQA_{\scr X}[\scr F_2]$; we shall show that $\Sigma^\infty_{S^1} \scr T_+ \in \scr C$.
Let $\scr T' = \scr T(\scr F_2 \setminus \scr F_1)$; this is smooth by Lemma \ref{lemm:adjacent-inertia-flat}.
By homotopy purity (Lemma \ref{lemm:homotopy-purity}; here we use the assumption on the local structure of $\scr X$) we have a cofiber sequence $\scr T(\scr F_1)_+ \to \scr T_+ \to Th(N_{\scr T'/\scr T})$.
Since $\scr T(\scr F_1) \in \SHS(\scr X)[\scr F_1]$, it suffices to show that $\Sigma^\infty_{S^1} Th(N_{\scr T'/\scr T}) \in \scr C$.
Lemma \ref{lemm:adjacent-splitting} provides us with a splitting $N_{\scr T'/\scr T} \wequi V_1 \times_{\scr T'} V_2$, with $V_1 \in \SmQA_{\scr X}[\scr F_2 \setminus \scr F_1]$ and $V_2 \setminus 0 \in \SmQA_{\scr X}[\scr F_1]$.
We have a pushout in $\Spc(\scr X)$ (coming from a Zariski covering)
\begin{equation*}
\begin{CD}
(V_1 \setminus 0) \times_{\scr T'} (V_2 \setminus 0) @>>> (V_1 \setminus 0) \times_{\scr T'} V_2 \wequi V_1 \setminus 0 \\
@VVV                                                               @VVV \\
V_1 \times_{\scr T'} (V_2 \setminus 0) \wequi V_2 \setminus 0  @>>> (V_1 \times_{\scr T'} V_2) \setminus 0
\end{CD}
\end{equation*}
as well as a cofiber sequence \[ (V_1 \times_{\scr T'} V_2) \setminus 0 \to V_1 \times_{\scr T'} V_2 \wequi \scr T' \to Th(N_{\scr T'/\scr T}). \]
It hence suffices to show that $\scr T', V_2 \setminus 0, V_1 \setminus 0, (V_1 \setminus 0) \times_{\scr T'} (V_2 \setminus 0) \in \scr C$.
This is true since $\scr T', V_1 \setminus 0 \in \SmQA_{\scr X}[\scr F_2 \setminus \scr F_1]$ and $V_2 \setminus 0, (V_1 \setminus 0) \times_{\scr T'} (V_2 \setminus 0) \in \SmQA_{\scr X}[\scr F_1]$.
(For the last containment, apply Remark \ref{rmk:isotropy-sieve} to the projection $(V_1 \setminus 0) \times_{\scr T'} (V_2 \setminus 0) \to V_2 \setminus 0$.)
\end{proof}

\begin{corollary} \label{cor:gerbes-gen}
Let $p: \scr X \to \scr S$ be a tame proper étale gerbe.
\reasonableness{\scr X}
Then $\SHS(\scr X)$ is generated by stacks with flat inertia over $\scr S$.
\end{corollary}
\begin{proof}
Applying Proposition \ref{prop:gerbes-gen} to the sequence of weakly adjacent families $\scr F_0/\scr S \subset \scr F_1/\scr S \subset \dots$ from Example \ref{ex:families} we find generation by $\SmQA_{\scr X}[\scr F_i/\scr S \setminus \scr F_{i-1}/\scr S]$ (for varying $i$).
These consist of stacks with flat inertia by Proposition \ref{prop:identify-flat-inertia}.
\end{proof}

\begin{proposition} \label{prop:families-reduction}
Let $\scr F_1 \subset \scr F_2$ be weakly adjacent families, with $\scr F = \scr F_2 \setminus \scr F_1$ and associated restricted orbit stack diagram \[ \scr X \xleftarrow{a} \widetilde\Orb_\scr{F} \xrightarrow{b} \Orb_{\scr F}. \]
\reasonableness{\scr X}
For $E \in \SHS(\scr X)$ we have \[ \E(\scr F_1, \scr F_2) \wedge E \wequi \E(\scr F_1, \scr F_2) \wedge a_\sharp b^{*}(\E(\scr F_\free/\scr S)_+ \wedge b_{*} a^{*} E). \]

Moreover the restricted functor \[ \SHS(\scr X) \to \SHS(\Orb_{\scr F})[\scr F_\free/\scr S] \] is a localization, with associated localization endofunctor given by $(\ph) \wedge \widetilde{\E\scr F_1}$.
\end{proposition}
\begin{proof}
First note that $(*)$ $j^*(\E(\scr F_1)_+) \wequi S^0 \wequi j^*(\widetilde{\E\scr F_2})$.

Now we prove the first statement.
Since $\E(\scr F_1, \scr F_2) \wequi \E(\scr F_1, \scr F_2) \wedge \E(\scr F_2)_+$ by Remarks \ref{rmk:EE-interpretation} and \ref{rmk:EFF-interpretation}, we may assume (using Lemma \ref{lemm:i!i*-ident} and $(*)$) that $E \in \SHS(\scr X)[\scr F_2]$.
Hence by Proposition \ref{prop:gerbes-gen}, we reduce to the case where either $E \in \SHS(\scr X)[\scr F_1]$ or $E \in \SmQA_{\scr X}[\scr F]$.
In the notation of Remark \ref{rmk:orbit-stack-usefulness}, we seek to prove that $j_!j^* E \to E$ is inverted when smashing with $\E(\scr F_1, \scr F_2)$.\NB{We use in particular that $b_*$ commutes with $\Sigma^\infty_{S^1}$, since it preserves colimits.}
In the first case, we have $j^*E = 0$ (indeed if $X \in \SmQA_\scr{X}[\scr F]$ and $Y$ in $\SmQA_{\scr X}[\scr F_1]$, then there are no maps $Y \to X$) and also \[ \E(\scr F_1, \scr F_2) \wedge E \stackrel{R.\ref{rmk:EFF-interpretation}}{\wequi} \E(\scr F_2)_+ \wedge \widetilde{\E \scr F_1} \wedge E \stackrel{L.\ref{lemm:i!i*-ident}}{\wequi} \E(\scr F_2)_+ \wedge \widetilde{\E \scr F_1} \wedge \E(\scr F_1)_+ \wedge E \stackrel{R.\ref{rmk:EFF-interpretation}}{=} 0. \]
In the second case, we have $j_!j^*E \wequi E$, and there is nothing to prove.

To prove the second statement, denote now the restricted functor by $j^*$.
Observe that $j^*$ has a left adjoint $j_!$, which is fully faithful (see Remark \ref{rmk:orbit-stack-usefulness}).
Since $j^*$ preserves colimits, it must have a right adjoint $j_*$, which hence is also fully faithful; thus $j^*$ is a localization.
By $(*)$ it inverts the localization $(\ph) \to (\ph) \wedge \widetilde{\E\scr F_1}$, whence we obtain a natural transformation $(\ph) \wedge \widetilde{\E\scr F_1} \to j_*j^*$.
As before we check that this induces an equivalence on $\map(T, \ph)$, where either $T \in \SHS(\scr X)[\scr F_1]$ or $T \in \SmQA_{\scr X}[\scr F]$.
In the former case we get $0$ on both sides as before.
In the latter we have \[ \map(j_! T, j_*j^*E) \wequi \map(j^*j_!T, j^*E) \wequi \map(T, j^*E) \] and \[ \map(j_! T, \widetilde{\E\scr F_1} \wedge E) \wequi \map(T, j^* \widetilde{\E\scr F_1} \wedge E)) \wequi \map(T, j^*E) \] (using $(*)$ once more).
This concludes the proof.
\end{proof}

\subsection{Quotients}
\begin{lemma} \label{lemm:presheaf-quotients}
Let $p: \scr X \to \scr S$ be a proper étale morphism.
\begin{enumerate}
\item The functor $p^*(\ph) \times \E(\scr F_\free/\scr S): \PSh(\SmQA_{\scr S}) \to \PSh(\SmQA_{\scr S})[\scr F_\free/\scr S]$ admits a left adjoint $p_!$.
\item For $\scr Y \in \SmQA_{\scr X}[\scr F_\free/\scr S]$, we have $p_!(\scr Y) = \scr Y \in \SmQA_{\scr S}$.
\item The functor $p_!$ preserves motivic equivalences.
\end{enumerate}
\end{lemma}
\begin{proof}
We have adjunctions \[ \PSh(\SmQA_{\scr X}[\scr F_\free/\scr S]) \stackrel[i^*]{i_!}{\rightleftarrows} \PSh(\Stk_\scr{X}) \stackrel[p^*]{\tilde{p}_!}{\rightleftarrows} \PSh(\Stk_{\scr S}), \] where $\tilde p_!$ is just the left Kan extension of the evident functor $\Stk_\scr{X} \to \Stk_\scr{S}$.
If $\scr Y \in \SmQA_{\scr X}[\scr F_\free/\scr S]$, then $\scr Y \to \scr S$ is representable and so $\scr Y \wequi \repr_{\scr S}(\scr Y)$.
By Proposition \ref{prop:relative-quotient-gerbe}, $\scr Y \to \scr S$ is smooth and quasi-affine.
It follows that the composite $\tilde{p}_! i_!$ lands in the full subcategory $\PSh(\SmQA_\scr{S})$.
Denote by $p_!$ the resulting functor \[ \PSh(\SmQA_{\scr X}[\scr F_\free/\scr S]) \wequi \PSh(\SmQA_{\scr X})[\scr F_\free/\scr S] \to \PSh(\SmQA_{\scr S}). \]
By construction it has a right adjoint given by the restriction of $i^*p^*$, which coincides with $p^*(\ph) \times \E\scr F_\free/\scr S$ by Lemma \ref{lemm:i!i*-ident}.
This proves (1,2).
Since $p_!$ is essentially the identity, it clearly preserves generating motivic equivalences, and hence (3) holds since $p_!$ preserves colimits.
\end{proof}

\begin{corollary} \label{cor:stable-quotients}
Let $p: \scr X \to \scr S$ be a proper étale morphism.
The functor \[ i^*p^*: \Spc(\scr S) \to \Spc(\scr X)[\scr F_\free/\scr S] \] admits a left adjoint $p_!$ given on representables by $p_!(\scr Y \to \scr X) = (\scr Y \to \scr X \to \scr S)$.
\end{corollary}
\begin{proof}
Immediate from Lemma \ref{lemm:presheaf-quotients}.
\end{proof}

\subsection{Stabilization}
\begin{lemma} \label{lemm:stabilization-lemma}
Let $p: \scr S \to \scr B$ be a tame proper étale gerbe.
\begin{enumerate}
\item Let $V$ be a vector bundle on $\scr S$.
  The canonical map \[ Th(p^*p_* V) \wedge \widetilde{\E(\scr F_\prop/\scr B)} \to Th(V) \wedge \widetilde{\E(\scr F_\prop/\scr B)} \in \Spc(\scr S)_* \] is an equivalence.
\item Suppose that $\scr B$ is linearly or nicely quasi-fundamental.
  Then \[ \Spc(\scr S)[\scr F_\free/\scr B]_*[\Sph_\scr{B}^{-1}] \wequi \Spc(\scr S)[\scr F_\free/\scr B]_*[\Sph_\scr{S}^{-1}]. \]
  (Recall the notation $\Sph_{\ph}$ from Corollary \ref{cor:SH-sph-normed}.)
\end{enumerate}
\end{lemma}
\begin{proof}
(1)
Let $V_0 = p^*p_* V \subset V$ be the fixed point bundle.
By tameness, this inclusion is split: $V \wequi V_0 \oplus V'$.
Thus $Th(V) \wequi Th(V_0) \wedge Th(V')$ and so it suffices to treat the case $V=V'$, $V_0 = 0$.
Smashing the cofiber sequences \[ \E(\scr F_\prop/\scr B)_+ \to S^0 \to \widetilde{\E(\scr F_\prop/\scr B)} \] and \[ (V \setminus 0)_+ \to V_+ \wequi S^0 \to Th(V) \] we obtain a diagram of cofiber sequences
\begin{equation*}
\begin{CD}
\E(\scr F_\prop/\scr B)_+ \wedge Th(V) @>>> Th(V) @>>> \widetilde{\E(\scr F_\prop/\scr B)} \wedge Th(V) \\
@AAA                                       @AAA                                @AAA \\
\E(\scr F_\prop/\scr B)_+ @>>> S^0 @>>>  \widetilde{\E(\scr F_\prop/\scr B)} \\
@AAA                             @AAA    @AAA \\
\E(\scr F_\prop/\scr B)_+ \wedge (V \setminus 0)_+ @>>> (V \setminus 0)_+ @>>> \widetilde{\E(\scr F_\prop/\scr B)} \wedge (V \setminus 0)_+.
\end{CD}
\end{equation*}
It suffices to show that the lower right hand corner is contractible, for which it is enough that the bottom left hand map is an equivalence, i.e. $V \setminus 0 \in \SmQA_{\scr S}[\scr F_\prop/\scr B]$.
In other words we need $p_*(V \setminus 0) = \emptyset$, which is true since $p_*(V) = 0$.

(2)
Let $\scr X \in \SmQA_{\scr S}[\scr F_\free/\scr B]$ and consider the square of adjunctions
\begin{equation*}
\begin{tikzcd}
\Spc(\scr X)_*[\Sph_\scr{B}^{-1}] \ar[r, "\sigma^\infty", bend left=5] \ar[d, "f_\sharp" swap, bend right] & \Spc(\scr X)_*[\Sph_\scr{S}^{-1}] \ar[l, "\omega^\infty", bend left=5] \ar[d, "f_\sharp" swap, bend right] \\
\Spc(\scr S)[\scr F_\free/\scr B]_*[\Sph_\scr{B}^{-1}] \ar[r, "\sigma^\infty", bend left=5] \ar[u, "f^*" swap, bend right] & \Spc(\scr S)[\scr F_\free/\scr B]_*[\Sph_\scr{S}^{-1}]. \ar[l, "\omega^\infty", bend left=5] \ar[u, "f^*" swap, bend right]
\end{tikzcd}
\end{equation*}
Here $\sigma^\infty$ commutes with $f^*$ and $f_\sharp$ for formal reasons ($f_\sharp \dashv f^*$ being an adjunction of modules).
It follows that also $\omega^\infty$ commutes with $f^*$.
We seek to prove that the natural transformations $\id \to \omega^\infty\sigma^\infty$ and $\sigma^\infty\omega^\infty \to \id$ are equivalences.
Functors of the form $f^*$ yield a conservative collection, so we may as well prove that $f^* \to f^*\omega^\infty\sigma^\infty \wequi \omega^\infty\sigma^\infty f^*$ and $\sigma^\infty\omega^\infty f^* \wequi f^*\sigma^\infty\omega^\infty \to f^*$ are equivalences (here we use the commutation explained above).
It thus will be enough to prove that if $V$ is a vector bundle on $\scr X$ (in fact $\scr S$), then $Th(V)$ is invertible in $\Spc(\scr X)[\Sph_\scr{B}^{-1}]$.
This follows from the fact that $\scr X \wequi \repr_{\scr B}(\scr X) \to \scr B$ is quasi-affine and of finite presentation (Proposition \ref{prop:relative-quotient-gerbe}), whence linearly/nicely quasi-fundamental, and so $\Spc(\scr X)[\Sph_\scr{B}^{-1}] \wequi \SH(\scr X)$ has the desired property \cite[Theorem 4.5(i, iv, v)]{khan2021generalized}.
\end{proof}

\begin{lemma} \label{lemm:tensor-E} \NB{move somewhere else?}
Let $p: \scr S \to \scr B$ be a proper étale morphism, $\scr F$ a family and $\scr M$ a (presentable) $\Spc(\scr S)_*$-module.
There are canonical equivalences \[ \Spc(\scr S)[\scr F] \otimes_{\Spc(S)} \scr M \wequi \E(\scr F)_+ \wedge \scr M \] and \[ (\Spc(\scr S)_* \wedge \widetilde{\E\scr F}) \otimes_{\Spc(S)_*} \scr M \wequi \widetilde{\E\scr F} \wedge \scr M. \]
\end{lemma}
\begin{proof}
Since $\Spc(\scr S)[\scr F] \to \Spc(\scr S)$ admits a right adjoint preserving colimits, the functor remains fully faithful after tensoring with $\scr M$ (see Lemma \ref{lemm:scalar-extension-2funct}).
The image will be generated by objects of the form $X \otimes M$, with $X \in \Spc(\scr S)[\scr F]$ and $M \in \scr M$.
Since $\E\scr F$ is an idempotent co-monoid (Remark \ref{rmk:EE-interpretation}), the category generated by these objects coincides with $\E(\scr F)_+ \wedge \scr M$, as needed.
It also follows that $(\ph) \wedge  \widetilde{\E\scr F}$ is the localization annihilating $\Spc(\scr S)_*[\scr F_\free]$, from which one deduces that we have a cofiber sequence of presentable categories \[ \Spc(\scr S)_*[\scr F_\free] \to \Spc(\scr S)_* \to \Spc(\scr S)_* \wedge \widetilde{\E\scr F}. \]
Since this is preserved by scalar extension (Lemma \ref{lemm:scalar-extension-cocont}), the second statement follows (by reversing the logic).
\end{proof}

\begin{proposition} \label{prop:stab-free-prop}
Let $\scr S \to \scr B$ be a tame proper étale morphism, with $\scr S$ and $\scr B$ both linearly scalloped, or nicely scalloped with affine diagonal.
Then \[ \SH(\scr S) \wedge \E(\scr F_\free/\scr B)_+ \wequi \Spc(\scr S)[\scr F_\free/\scr B] \otimes_{\Spc(\scr B)} \SH(\scr B) \] and  \[ \SH(\scr S) \wedge \widetilde{\E(\scr F_\prop/\scr B)} \wequi (\Spc(\scr S)_* \wedge \widetilde{\E(\scr F_\prop/\scr B)}) \otimes_{\Spc(\scr B)_*} \SH(\scr B). \]
\end{proposition}
\begin{proof}
Factor $p$ through $\scr S \to \repr_{\scr B} \scr S =: \scr B' \to \scr B$.
Then $\scr B' \to \scr B$ is finite étale, and so $\scr B'$ is also scalloped.
The result for $\scr B' \to \scr B$ holds by \cite[Theorem 4.5(iv)]{khan2021generalized}, and filtration by isotropy disregards this part.
Hence we may replace $\scr B$ by $\scr B'$ and so assume that $p$ is a gerbe.

Suppose first that $\scr S$ and $\scr B$ are both quasi-fundamental.
In this case $\SH(\scr S)$ is obtained from $\Spc(\scr S)$ by inverting Thom spaces of all vector bundles, and so $\SH(\scr S) \wedge \E(\scr F_\free/\scr B)_+$ is obtained from $\Spc(\scr S)[\scr F_\free/\scr B]$ by the same procedure (use Lemma \ref{lemm:tensor-E}).
Similarly $\SH(\scr B)$ is obtained from $\Spc(\scr B)$ by inverting all Thom spaces.
The first claim thus reduces to Lemma \ref{lemm:stabilization-lemma}(2).
Similarly $\SH(\scr S) \wedge \widetilde{\E(\scr F_\prop/\scr B)}$ is obtained from $\Spc(\scr S)_*\wedge \widetilde{\E(\scr F_\prop/\scr B)}$ by inverting all Thom spaces, and by the Lemma \ref{lemm:stabilization-lemma}(2) for this it suffices to invert Thom spaces of vector bundles on $\scr B$; hence the second claim.

We now reduce to the general statement to this case.
The argument is the same for both claims, so we focus on the first one.

Given $\scr U \to \scr B$, put $\scr U' = \scr U \times_\scr{B} \scr S$.
Consider the natural transformation of presheaves of categories $F \to G$ on $\SmQA_{\scr B}$ where \[ F(\scr U) = \Spc(\scr S)[\scr F_\free/\scr B] \otimes_{\Spc(\scr S)} \Spc(\scr U') \otimes_{\Spc(\scr B)} \SH(\scr B) \] and \[ G(\scr U) = \Spc(\scr S)[\scr F_\free/\scr B] \otimes_{\Spc(\scr S)} \Spc(\scr U') \otimes_{\Spc(\scr S)} \SH(\scr S). \]
Note that $\E(\scr F_\free/\scr B)|_{\SmQA_{\scr U'}} \wequi \E(\scr F_\free/\scr U)$ and hence by Lemma \ref{lemm:tensor-E} we get \[ \Spc(\scr S)[\scr F_\free/\scr B] \otimes_{\Spc(\scr S)} \Spc(\scr U') \wequi \Spc(\scr U')[\scr F_\free/\scr U']. \]
Thus \[ F(\scr U) \wequi \Spc(\scr U')[\scr F_\free/\scr U'] \otimes_{\Spc(\scr U)} \Spc(\scr U) \otimes_{\Spc(\scr B)} \SH(\scr B) \wequi \Spc(\scr U')[\scr F_\free/\scr U'] \otimes_{\Spc(\scr U)} \SH(\scr U). \]
Similarly \[ G(\scr U) \wequi \Spc(\scr U')[\scr F_\free/\scr U'] \otimes_{\Spc(\scr U')} \SH(\scr U') \wequi \SH(\scr U') \wedge \E(\scr F_\free/\scr U')_+, \] using once more Lemma \ref{lemm:tensor-E}.
All this shows that proving that $F(\scr U) \to G(\scr U)$ is an equivalence amounts to proving the theorem with $\scr S \to \scr B$ replaced by $\scr U' \to \scr U$.
Moreover, applying Example \ref{ex:sheaf-stabilization} and Lemma \ref{lemm:qaff-descent} to the very first definitions of $F, G$, we see that both are sheaves in the quasi-affine Nisnevich topology.
Consequently it suffices to prove the theorem for all the terms arising from the nerve of some quasi-affine Nisnevich cover of $\scr B$.
Lemma \ref{lemm:gerbes-form} below supplies such a cover where all terms are quasi-fundamental (using that quasi-fundamental stacks are stable under quasi-affine étale extension).
This concludes the reduction and hence the proof.
\end{proof}

\begin{lemma} \label{lemm:gerbes-form}
Let $p: \scr S \to \scr B$ be a gerbe with finite inertia, such that both $\scr S$ and $\scr B$ are linearly scalloped, or nicely scalloped with affine diagonal.
Then there exists a quasi-affine Nisnevich cover $\scr U \to \scr B$ such that both $\scr U$ and $\scr U \times_{\scr B} \scr S$ quasi-fundamental.
\end{lemma}
\begin{proof}
Replacing it by an appropriate cover (existence of which is ensured by the scallopedness assumption), we may assume that $\scr B$ is quasi-fundamental.
There exists a scallop decomposition $\scr S'_\bullet$ of $\scr S$, in which all terms are quasi-fundamental.
We claim that $p_* \scr S'_\bullet \to \scr B$ is a scallop decomposition, $p^*p_*\scr S'_\bullet \to \scr S'_\bullet \to \scr S$ is a scallop decomposition, and $p^*p_*\scr S'_\bullet \to \scr S'_\bullet$ is an open immersion.
This will prove what we want.
Note that a sequence of squares forming a scallop decomposition is fppf local on the target.
Consequently everything we claim is fppf local on $\scr B$, and so we may assume that $\scr B = B$ is an affine scheme and $\scr S = B \sslash G$ for some finite flat group scheme $G$.
The scallop decomposition $\scr S'_\bullet \to \scr S$ consists of various quasi-affine $G$-schemes, étale over $B$.
Then $p_*\scr S'_\bullet$ consists of the fixed loci, which in our case are just the isovariant loci over $B$.
Thus $p_*\scr S'_\bullet$ consists of open subschemes and remains a scallop decomposition for formal reasons; see \cite[Proposition 2.13]{gepner-heller} for more details.
This concludes the proof.
\end{proof}

\begin{corollary} \label{cor:geom-fixed-points}
Let $p: \scr S \to \scr B$ be a as in Proposition \ref{prop:stab-free-prop}.
Then the two functors \[ p_\otimes \text{ and } p_*(\Sigma^\infty \widetilde{\E(\scr F_\prop/\scr S)} \wedge \ph): \SH(\scr S) \to \SH(\scr B) \] are naturally equivalent.
In particular the square
\begin{equation*}
\begin{CD}
\Spc(\scr S)_* \wedge \widetilde{\E(\scr F_\prop/\scr S)} @>{\Sigma^\infty}>> \SH(\scr S) \\
@V{p_*}VV                                                             @V{p_*}VV       \\
\Spc(\scr B)_* @>{\Sigma^\infty}>> \SH(\scr B)
\end{CD}
\end{equation*}
commutes.
\end{corollary}
\begin{proof}
The functor $p_*(\Sigma^\infty \widetilde{\E(\scr F_\prop/\scr S)} \wedge \ph)$ is right adjoint to the composite \[ F: \SH(\scr B) \to \SH(\scr S) \to \SH(\scr S) \wedge \widetilde{\E(\scr F_\prop)} \] (use that $\widetilde{\E(\scr F_\prop)}$ is an idempotent monoid by Remark \ref{rmk:Etilde-monoid}).
By Proposition \ref{prop:stab-free-prop}, $F$ is obtained by scalar extension along $\Spc(\scr B) \to \SH(\scr B)$ from its analog on the level of pointed spaces.
Since $p_*$ on the level of spaces is a module functor over $\Spc(\scr B)$ (being essentially given by fixed points), Lemma \ref{lemm:scalar-extension-2funct} implies that the right adjoint of $F$ is thus obtained by scalar extending the right adjoint at the level of spaces.
\tdn{details from here on}
This implies the second claim; it also shows that the functor is symmetric monoidal.
For the first, using the definition of $p_\otimes$ by passing to associated sheaves (see \S\ref{subsec:stab-norms}) one reduces as in Proposition \ref{prop:stab-free-prop} to the case where $\scr S$ is quasi-fundamental.
Now $p_\otimes$ has a universal property shared by the other functor, by what we already observed; hence they agree.
\end{proof}

\begin{corollary} \label{cor:quotients-stable}
Let $p: \scr S \to \scr B$ be a tame proper étale morphism, with $\scr S, \scr B$ as in Proposition \ref{prop:stab-free-prop}.
The functor \[ \SH(\scr B) \to \SH(\scr S) \to \SH(\scr S)[\scr F_\free/\scr B] \] admits a left adjoint $p_!$, compatible with $\Sigma^\infty$ and arbitrary base change.
\end{corollary}
\begin{proof}
Using Proposition \ref{prop:stab-free-prop}, we may construct the left adjoint by scalar extending the functor on the level of spaces from \S\ref{subsec:app-quotients}; all properties follows.
\end{proof}

\begin{corollary} \label{cor:stable-families-reduction}
Let $p: \scr X \to \scr S$ be proper tame étale.
Let $\scr F_1 \subset \scr F_2$ be weakly adjacent families, with $\scr F = \scr F_2 \setminus \scr F_1$ and associated restricted orbit stack diagram \[ \scr X \xleftarrow{a} \widetilde\Orb_\scr{F} \xrightarrow{b} \Orb_{\scr F}. \]
Assume that $\Orb_{\scr F}$ and $\scr B$ are both (linearly or nicely with affine diagonal) scalloped.

For $E \in \SH(\scr X)$ we have a natural equivalence \[ \E(\scr F_1, \scr F_2) \wedge E \wequi \E(\scr F_1, \scr F_2) \wedge a_\sharp b^{*}(\E(\scr F_\free/\scr S)_+ \wedge b_\otimes a^{*} E). \]
\end{corollary}
\begin{proof}
The functor $j_1'^*: \SH(\scr X) \to \SH(\Orb_{\scr F})[\scr F_\free/\scr S]$ inverts $\wedge \E(\scr F_2)_+$ and $\wedge \widetilde{\E\scr F_1}$, as does the functor $\wedge \E(\scr F_1, \scr F_2)$.
We may thus restrict to $E \in \SH(\scr X)[\scr F_2]$.
Denote by $j_0: \SHS(\scr X)[\scr F_2] \to \SHS(\Orb_{\scr F})[\scr F_\free/\scr S]$ the canonical functor.
This is a module functor over $\Spc(\scr X)$ and also a localization (the latter by Proposition \ref{prop:families-reduction}).
Scalar extending along $\Spc(\scr X) \to \SH(\scr X)$ we obtain a functor \[ j_1^*: \SH(\scr X)[\scr F_2] \to \SH(\Orb_{\scr F})[\scr F_\free/\scr S] \] (use Proposition \ref{prop:stab-free-prop} to identify the target), which is still a localization.
By construction, $j_1^*$ is the restriction of $j_1'^*$.
Since $j_1^*$ is precisely the localization at maps inverted by $\wedge \widetilde{\E\scr F_1}$, we find that $j_{1*}j^*_1 \wequi (\ph) \wedge \widetilde{\E\scr F_1}$.
It will thus suffice construct an equivalence $j_1^* \wequi j_1^* a_\sharp b^* j_1^*$.
We in fact produce an equivalence \[ \id \wequi j^*_1 a_\sharp b^* \in \mathrm{End}(\SH(\Orb_{\scr F})[\scr F_\free/\scr S]). \]
Indeed this functor factors as \begin{gather*} \SH(\Orb_{\scr F})[\scr F_\free/\scr S] \wequi \SHS(\Orb_{\scr F})[\scr F_\free/\scr S] \otimes_{\Spc(\scr B)} \SH(\scr B) \xrightarrow{a_\sharp b^*} \\ \SHS(\scr X)[\scr F_2]  \otimes_{\Spc(\scr B)} \SH(\scr B) \xrightarrow{\Sigma^\infty} \SH(\scr X)[\scr F_2] \xrightarrow{j_1^*} \SH(\Orb_{\scr F})[\scr F_\free/\scr S] \wequi \SH(\Orb_{\scr F})[\scr F_\free/\scr S],\end{gather*} the composite in the second row is $j_0^*$, and the functor in the first row is the left adjoint of $j_0^*$.
We conclude since $j_0^*$ is a localization and any left adjoint of a localization (if it exists) is fully faithful.
\end{proof}

\section{Ambidexterity} \label{sec:ambidex}
\subsection{Partially defined adjoints}
Let $p^*: \scr C \to \scr D$ be a functor of $\infty$-categories.
\begin{definition}
By a right adjoint to $p^*$ at $X \in \scr D$ we mean an object $U \in \scr C$ together with a map $c: p^*U \to X$ such that, for every $T \in \scr C$ the composite \[ \Map_{\scr C}(T, U) \xrightarrow{p^*} \Map_{\scr D}(p^*T, p^*U) \xrightarrow{c \circ } \Map_{\scr D}(p^*T, X) \] is an equivalence.
\end{definition}

\begin{example}
If $p$ admits a right adjoint $p_*$, then the co-unit $p^*p_* X \to X$ exhibits $p_*X$ as partial right adjoint of $p^*$ at $X$.
\end{example}

\begin{remark}
Passing to a bigger universe if necessary, we may assume $\scr C, \scr D$ are small.
Then the left Kan extension $p^*: \PSh(\scr C) \to \PSh(\scr D)$ has right adjoint $p_*$, given by precomposition with $p$.
For $X \in \scr D$, write $R_X \in \PSh(\scr D)$ for the associated representable presheaf.
Then a map $c: p^*U \to X$ is the same thing as a map $c^\dagger: R_U \to p_*(R_X)$, and $c$ exhibits $U$ as a right adjoint to $p^*$ at $X$ if and only if $c^\dagger$ is an equivalence.
\end{remark}

It follows in particular that partial right adjoints are unique in a strong sense.
Abusing notation somewhat, when no confusion can arise, we will denote the partial right adjoint to $p^*$ at $X$ by $p_* X$.

\begin{remark}
Let $\scr D' \subset \scr D$ denote the full subcategory on those objects which admit a partial right adjoint of $p^*$.
Then we obtain a functor $p_*: \scr D' \to \scr C$ ($\subset \PSh(\scr C)$).
\end{remark}

\begin{remark} \label{rmk:partial-co-unit}
Let $V \in \scr C$, and $U$ a partial right adjoint of $p^*$ at $p^* V$.
Then the identity map $p^* V \to p^* V$ corresponds to a map $u: V \to U$ which we call the \emph{(partial) unit}.
By definition of $u$, the composite \[ p^*V \xrightarrow{p^*u} p^*U \xrightarrow{c} p^*V \] is the identity.
\end{remark}

Now suppose given a commutative square of $\infty$-categories
\begin{equation*}
\begin{CD}
\scr C @>F>> \scr C' \\
@A{p^*}AA     @A{p^*}AA \\
\scr D @>F>> \scr D'.
\end{CD}
\end{equation*}
\begin{definition} \label{def:preserve-partial-adjoint}
Let $c: p^*U \to X \in \scr C$ exhibit $U \in \scr D$ as partial right adjoint of $p^*$ at $X$.
We say that \emph{$F$ preserves the partial right adjoint at $X$} if \[ p^*(FU) \wequi F(p^*U) \xrightarrow{Fc} FX \] exhibits $FU$ as partial right adjoint of $p^*$ at $FX$.
Alternatively, suppose that $U' \in \scr D'$ is a partial right adjoint of $p^*$ at $FX$.
Then the map $p^*(FU) \to FX$ induces a map \[ \Ex: FU \to U', \] called the \emph{(partial) exchange transformation at $X$}.
\end{definition}

\begin{remark}
We also have ``exchange transformations'' $Fp^* \wequi p^*F$.
In general we denote by $\Ex$ exchange transformations built by composing the above two types.
\end{remark}

\begin{lemma} \label{lemm:u-commutes}
Suppose given a commutative square as above, and $Y \in \scr D$.
Assume that $p^*Y$ and $Fp^*Y$ admit partial right adjoints $p_*p^*Y, p_*p^*FY$.
Then the following diagram commutes
\begin{equation*}
\begin{tikzcd}
FY \ar[r,"Fu"] \ar[dr,"uF" swap] & Fp_*p^*Y \ar[d,"\Ex"] \\
   & p_*p^*FY.
\end{tikzcd}
\end{equation*}
\end{lemma}
\begin{proof}
We need to prove that two maps with target the partial right adjoint $p_*(p^* FY)$ are homotopic.
In order to do this we pass to the adjoint diagram, with target $p^* FY$.
\begin{equation*}
\begin{tikzcd}
Fp^* Y \ar[r, "Fp^* u"] & Fp^*p_*p^*Y \ar[ddd, "Fc", bend left=50] \\
p^*FY \ar[r, "p^*Fu"] \ar[u, "\wequi"] \ar[d, "p^*uF"] \ar[dd, bend right=70, "\id" swap] & p^*Fp_*p^*Y \ar[u, "\wequi"] \ar[ld, "p^*\Ex"] \\
p^*p_*p^*FY \ar[d, "c"] \\
p^*FY \ar[r, "\wequi"] & Fp^*Y \\
\end{tikzcd}
\end{equation*}
The top cell commutes for trivial reasons, the left hand cell by Remark \ref{rmk:partial-co-unit}, and the right hand cell by the definition of $\Ex$.
It remains to prove that the composite $Fp^* Y \xrightarrow{Fp^* u} Fp^*p_*p^*Y \xrightarrow{Fc} Fp^*Y$ is also the identity.
This holds by Remark \ref{rmk:partial-co-unit}.
\end{proof}

\begin{remark}
Dualizing the discussion, we can also defined partial left adjoints.
For example, to exhibit a partial left adjoint $U \in \scr C$ of $p^*$ at $X \in \scr D$, we need to proved a map $X \to p^* U$ satisfying a universal property.
We often denote the partial right adjoint by $p_! X$.
\end{remark}

\begin{remark}
Given $X \in \scr D$, by an \emph{ambidextrous partial adjoint of $p^*$ at $X$} we mean an object $U \in \scr C$ together with maps $X \to p^* U$ and $p^* U \to X$ exhibiting $U$ as both partial left and partial right adjoint.
Beware that, even if this exists, it is not uniquely determined by $X$ and $p^*$.
Indeed the extra data is an equivalence between the partial left and right adjoints (each separately uniquely determined).
\end{remark}

\subsection{Normable and ambidextrous objects}
Fix an $\infty$-category $\scr X$ with pullbacks and a functor $\scr A: \scr X^\op \to \Cat_\infty$.
For $f: X \to Y \in \scr X$, we denote the functoriality by $f^*: \scr A(Y) \to \scr A(X)$.
We write $\Delta_f: X \to X \times_Y X$ for the diagonal of $f$, \[ \Delta_f^2: X \to X \times_{X \times_Y X} X \] for the double diagonal, and more generally $\Delta_f^n$ for the iterated diagonal.
We call $f$ \emph{truncated} if $\Delta_f^n$ is an equivalence for some $n$.
\begin{definition}
Let $f: X \to Y \in \scr X$ be truncated.
We shall inductively define what it means for an object $T \in \scr A(X)$ to be $f$-normable (respectively $f$-ambidextrous).
For an $f$-normable object $T$, the partial left adjoint to $f^*$ at $X$ will exist, say denoted by $(U,\eta)$, and come with a canonical map $\nu = \nu_{f,U,\eta}: f^* U \to T$.
We call $T$ $f$-ambidextrous if $\nu$ exhibits $U$ as a partial right adjoint to $f^*$ at $X$.

To begin with, if $f$ is an equivalence, then all objects of $\scr A(X)$ are $f$-ambidextrous, $U := (f^*)^{-1}(T)$ and $T \to f^*U \to T$ are the canonical maps.
If $f$ is not an equivalence, then by induction we may assume that $\Delta_f$-ambidexterity has been defined.
Consider the pullback square
\begin{equation*}
\begin{CD}
X \times_Y X @>{\pi_1}>> X \\
@V{\pi_2}VV             @VfVV \\
X @>f>>                 Y.
\end{CD}
\end{equation*}
We call $T \in \scr A(X)$ $f$-normable if
\begin{enumerate}
\item a partial left adjoint $(U,\eta: T \to f^* U)$ of $f^*$ at $T$ exists,
\item $f^*$ preserves this partial left adjoint (i.e. $f^*U \wequi \pi_{2!}\pi_1^* T$; see Definition \ref{def:preserve-partial-adjoint}), and
\item $T$ is $\Delta_f$-ambidextrous.
\end{enumerate}
Now in order to construct $\nu: f^* U \to T$, by (2) it suffices to exhibit $\nu': \pi_1^* T \to \pi_2^* T$.
Note that $\Delta_f^* \pi_i^* T \wequi T$, which by assumption is $\Delta_f$-ambidextrous.
Let $V \in \scr A(X \times_Y X)$ be the ambidextrous partial adjoint of $\Delta_f^*$ at $T$.
Writing $T = \Delta^* \pi_1^* T$ we obtain a unit $\pi_1^* T \to V$, and writing $T = \Delta^* \pi_2^* T$ we obtain a co-unit $V \to \pi_2^* T$.
We let $\nu'$ be their composite.
\end{definition}
\NB{what about replacing left by right adjoints?}

If $X$ is $f$-normable and the partial right adjoint also exists at $X$ (e.g. $X$ if is $f$-ambidextrous), then we obtain a map \[ W = W_f: f_! X \to f_* X. \]
By definition $X$ is $f$-ambidextrous if and only if $W_f$ is an equivalence.

\subsection{Naturality of the norm}
Consider a natural transformation \[ \scr A \to \scr B \in \Fun(\scr X^\op, \Cat_\infty). \]
\begin{lemma} \label{lemm:norm-nat}
Let $f: X \to Y \in \scr X$ be truncated and $T \in \scr A(X)$.
Suppose that $T, FT$ are $f$-normable, partial right adjoints $f_* T, f_* FT$ exist, and for each iterated diagonal $\Delta = \Delta_f^n$ (with $n \ge 1$) the exchange transformation \[ \Delta_! FT \to F\Delta_! T \] is an equivalence.
Then the following diagrams commute
\begin{equation*}
\begin{tikzcd}
f_!FT \ar[d] \ar[r, "WF"] & f_* FT & f^*f_! FT \ar[d] \ar[r, "\nu F"] & FT & \pi_1^*FT \ar[d, "\wequi"] \ar[r, "\nu' F"] & \pi_2^* FT \ar[d, "\wequi"] \\
Ff_!T \ar[r, "FW"] & Ff_*T \ar[u] & Ff^*f_!T \ar[ur, "F\nu" swap]     &&     F\pi_1^*T \ar[r, "F\nu'"] & F\pi_2^* T,
\end{tikzcd}
\end{equation*}
where the unlabelled maps are exchange transformations.
\end{lemma}

Our main use for this is the following.
\begin{proposition}
Let $f: X \to Y \in \scr X$ be truncated and $T \in \scr A(X)$.
Suppose that $T, FT$ are $f$-ambidextrous and for each iterated diagonal $\Delta = \Delta_f^n$, where we allow $n=0$ (and $\Delta_f^0 := f$) the exchange transformation \[ \Delta_! FT \to F\Delta_! T \] is an equivalence.
Then the exchange transformation \[ Ff_* T \to f_* FT \] is an equivalence (and the same is true for $\Delta$ in place of $f$).
\end{proposition}
\begin{proof}
Immediate from commutativity of the first diagram, using $2$-out-of-$3$.
\end{proof}

\begin{proof}[Proof of Lemma \ref{lemm:norm-nat}.]
The first diagram asserts that two maps into $f_*FT$ are homotopic.
We hence pass to adjoints as follows.
\begin{equation*}
\begin{tikzcd}
f^*f_!FT \ar[r, "f^*WF"] \ar[d] & f^*f_*FT \ar[r, "c"] & FT \\
f^*Ff_!T \ar[r, "f^*FW"] \ar[rd, "\wequi"] & f^*Ff_*T \ar[r, "\wequi"] \ar[u] & Ff^*f_*T \ar[u, "Fc"] \\
& Ff^*f_!T \ar[ru, "Ff^*W" swap]
\end{tikzcd}
\end{equation*}
The right hand cell commutes by definition of the exchange transformation for $f_*$, and the bottom cell commutes for trivial reasons.
The ``outer triangle'' (involving the top left hand, top right hand, and bottom corner) is thus isomorphic to the second diagram (by definition of $W$ in terms of $\nu$).
The first commutativity statement thus reduces to the second.

We expand the second diagram as follows.
\begin{equation*}
\begin{tikzcd}
\pi_{2!}\pi_1^* FT \ar[r]\ar[d] & f^*f_! FT \ar[r, "\nu F"] \ar[d] &  FT \\
F\pi_{2!}\pi_1^* T \ar[r] & Ff^*f_!T \ar[ur, "F\nu"]
\end{tikzcd}
\end{equation*}
The partial left adjoints $\pi_{2!}$ exist by definition of $f$-normability.
The left hand cell commutes by naturality of exchange transformations.\NB{elaborate?}
The left hand horizontal maps are equivalences by definition of normability, so it suffices to prove commutativity of the outer diagram.
We do so by passing to adjoints as follows.
\begin{equation*}
\begin{tikzcd}
\pi_1^*FT \ar[d, "\wequi"] \ar[r, "uF"] & \pi_2^*\pi_{2!} \pi_1^*FT \ar[r, "\pi_2^*\nu F"] \ar[d] & \pi_2^* FT \\
F\pi_1^*T \ar[r, "Fu"] & \pi_2^*F\pi_{2!}\pi_2^*T \ar[ru, "\pi_2^*F\nu" swap]
\end{tikzcd}
\end{equation*}
The left hand cell commutes by the definition of the exchange transformation (for $\pi_{2!}$).
The outer triangle is thus isomorphic to the third diagram (by definition of $\nu$ in terms of $\nu'$), and so the second commutativity statement reduces to the third.

Finally we expand the third diagram using the definition of $\nu'$ as follows.
\begin{equation*}
\begin{CD}
\pi_1^* FT @>{uF}>> \Delta_*\Delta^* \pi_1^*FT @<{WF}<< \Delta_!\Delta^*\pi_1^*FT @>{\wequi}>> \Delta_! \Delta^*\pi_2^*FT @>{cF}>> \pi_2^*FT \\
@V{\wequi}VV     @A{\Ex'}AA                                  @V{\Ex}VV                             @VVV                     @V{\wequi}VV \\
F\pi_1^*T @>{Fu}>> F\Delta_*\Delta^*\pi_1^*T @<{FW}<< F\Delta_!\Delta^*\pi_1^*T @>{\wequi}>> F\Delta_!\Delta^*\pi_2^*T @>{Fc}>> F\pi_2^*T
\end{CD}
\end{equation*}
The outer cells commute by Lemma \ref{lemm:u-commutes}.
The right hand middle cell commutes for trivial reasons.
The left hand middle cell commutes by induction.
By assumption, the map $\Ex$ is an equivalence, and so are $FW, WF$; hence the same is true for $\Ex'$.
One deduces that we may reverse the arrows $WF, FW$ and $\Ex'$ while retaining commutativity.
The result follows.
\end{proof}

\section{More about spaces with finite étale transfers} \label{sec:more-fet}
\subsection{Descent and continuity} \label{subsec:descent-ctty} \NB{this actually holds for $\NAlg$ in any normed category satisfying descent/continuity}
\begin{lemma} \label{lemm:spcfet-descent}
The functor $\Stk^\op \to \Cat, \scr X \mapsto \Spc^\fet(\scr X)$ satisfies quasi-affine Nisnevich descent.
\end{lemma}
\begin{proof}
Let $\scr X_\bullet \to \scr X$ be the nerve of a quasi-affine Nisnevich covering.
We wish to show that $c: \Spc^\fet(\scr X) \to \lim_\Delta \Spc^\fet(\scr X_\bullet)$ is an equivalence.
The right hand side consists of compatible families of objects $(F_i \in \Spc^\fet(\scr X_i))_{i \in \Delta}$ \cite[Corollary 3.3.3.2]{lurie-htt}.
In these terms the functor $c$ is given by $c(F)_i = p_i^*(F)$, where $p_i: \scr X_i \to \scr X$ is the projection.
It follows (from \cite[Proposition 5.5.3.13]{lurie-htt}) that $c$ has a right adjoint $r$, given by $r(F) = \lim_i p_{i*}(F_i)$.
Since the $p_i$ are smooth quasi-affine, both $p_i^*$ and $p_{i*}$ commute with the conservative forgetful functors $U:\Spc^\fet(\ph) \to \Spc(\ph)$, and the same is of course true for the limit.
Thus the unit and counit maps $\id \to rc$ and $cr \to \id$ are both equivalences, since they are after applying $U$, since $\scr X \mapsto \Spc(\scr X)$ satisfies quasi-affine Nisnevich descent (see Lemma \ref{lemm:qaff-descent}).
\end{proof}

\begin{lemma} \label{lemm:spcfet-cont}
Let $\scr X \wequi \lim_\alpha \scr X_\alpha$ where the transition maps are affine and the diagram is cofiltered.
Then $\Spc^\fet(\scr X) \wequi \lim_\alpha \Spc^\fet(\scr X_\alpha)$.
\end{lemma}
In order to make sense of the above statement, we view $\Spc^\fet$ as a functor $\Sch \to \Cat$ (not $\Sch^\op$!), with functoriality coming from functors of the form $p_*$.
Equivalently we have \cite[Corollary 5.5.3.4 and Theorem 5.5.3.18]{lurie-htt} \[ \Spc^\fet(\scr X) \wequi \colim_\alpha \Spc^\fet(\scr X_\alpha) \in Pr^L, \] where this time we do use the $p^*$ functoriality.
\begin{proof}
We have $\colim_\alpha \SmQA_{\scr X_\alpha} \wequi \SmQA_{\scr X}$ (see e.g. the proof of Lemma \ref{lemm:cty-spaces}).
The complete Segal space description of $\Cor^\fet(\SmQA_{\ph})$ together with continuity of $\FEt_{\ph}$ \cite[Théorèmes 8.8.2(ii), 8.10.5(x) and Proposition 17.7.8(ii)]{EGAIV} then implies that \[ \colim_\alpha \Cor^\fet(\SmQA_{X_\alpha}) \wequi \Cor^\fet(\SmQA_{X}). \]
The functor $\PSh_\Sigma$ is left adjoint to the forgetful functor from presentable categories to categories with finite coproducts \cite[Proposition 5.5.8.15]{lurie-htt}, and hence we get \[ \colim_\alpha \PSh_\Sigma(\Cor^\fet(\SmQA_{X_\alpha})) \wequi \PSh_\Sigma(\Cor^\fet(\SmQA_{X})) \in Pr^L. \]
Rewriting this as a limit, it remains to prove that if $F \in \PSh_\Sigma(\Cor^\fet(\SmQA_{X}))$ then $F$ is homotopy invariant (respectively a quasi-affine Nisnevich sheaf) if and only if the same is true for $p_{\alpha*} F$ for every $\alpha$, where $p_\alpha: \scr X \to \scr X_\alpha$ is the projection.
This follows from the same statement for $\Spc(\ph)$, which holds by Lemma \ref{lemm:cty-spaces}.
\end{proof}

\subsection{Isotropy specification} \label{subsec:spcfet-isotropy-spec}
Fix a proper étale morphism $p: \scr S \to \scr B$ and a family $\scr F$ as in Definition \ref{def:family}.
We write \[\Spc^\fet(\scr S)[\scr F] := L_\mot \PSh_\Sigma(\Cor^\fet(\SmQA_{\scr S}[\scr F])). \]
Left Kan extension induces an adjunction \[ i_!: \PSh_\Sigma(\Cor^\fet(\SmQA_{\scr S}[\scr F])) \adj \PSh_\Sigma(\Cor^\fet(\SmQA_{\scr S})): i^*. \]
\begin{lemma} \label{lemm:i!-fet-pres-mot-equiv}
Given a family of isotropy groups $\scr F$, both of the functors $i_!$ and $i^*$ preserve Nisnevich, homotopy and motivic equivalences.
The forgetful functor $U: \PSh_\Sigma(\Cor^\fet(\SmQA_{\scr S}[\scr F])) \to \PSh(\SmQA_{\scr S}[\scr F])$ preserves and detects motivic equivalences.
\end{lemma}
\begin{proof}
By construction $i_!$ preserves motivic equivalences.
Since $i_!$ is fully faithful, the following diagram commutes
\begin{equation*}
\begin{CD}
\PSh_\Sigma(\Cor^\fet(\SmQA_{\scr S}[\scr F])) @>{i_!}>> \PSh_\Sigma(\Cor^\fet(\SmQA_{\scr S})) \\
@V{U}VV                                                   @V{U'}VV \\
\PSh_\Sigma(\SmQA_{\scr S}[\scr F]) @<{(i^*)'}<< \PSh_\Sigma(\SmQA_{\scr S}).
\end{CD}
\end{equation*}
It follows that $U$ preserves motivic equivalences, since so do $i_!$ (by what we just said), $U'$ (by Lemma \ref{lemm:detect-motivic-equiv}) and $(i^*)'$ (by Lemma \ref{lemm:family-inclusion-pres-mot-eq}).
Since $U$ preserves motivically local objects by construction, it also detects motivic equivalences.
Finally since $Ui^* \wequi (i^*)'U'$, we see that $Ui^*$ preserves motivic equivalences, and hence so does $i^*$.
\end{proof}

\begin{lemma} \label{lemm:Spc-fet-semiadditive}
The categories $\Cor^\fet(\SmQA_{\scr S}[\scr F]), \PSh_\Sigma(\Cor^\fet(\SmQA_{\scr S}[\scr F]))$ and $\Spc^\fet(\scr S)[\scr F]$ are semiadditive.
\end{lemma}
\begin{proof}
This is obvious for $\Cor^\fet(\SmQA_{\scr S}[\scr F])$, holds for $\PSh_\Sigma(\Cor^\fet(\SmQA_{\scr S}[\scr F]))$ by \cite[Corollary 2.4]{gepner2016universality}, and follows for $\Spc^\fet(\scr S)[\scr F]$ since the latter category is closed inside $\PSh_\Sigma(\Cor^\fet(\SmQA_{\scr S}[\scr F]))$ under finite products and hence finite coproducts.
\end{proof}

\begin{corollary}
\begin{enumerate}
\item There is an induced adjunction \[ i_!: \Spc^\fet(\scr S)[\scr F] \adj \Spc^\fet(\scr S): i^*. \]
\item The functor $i_!$ is fully faithful and the functor $i^*$ preserves colimits.
\item There is a canonical equivalence of endofunctors \[ i_!i^* \wequi (\ph \mapsto \ph \wedge F(\E \scr F)): \Spc^\fet(\scr S) \to \Spc^\fet(\scr S). \]
\end{enumerate}
\end{corollary}
\begin{proof}
(1,2) The only statement requiring proof is that $i^*$ preserves colimits.
Being a right adjoint it preserves products, and hence by semiadditivity (Lemma \ref{lemm:Spc-fet-semiadditive}) it suffices to show that it preserves sifted colimits \cite[Lemma 2.7]{bachmann-norms}.
This we can check for $\PSh_\Sigma(\dots)$ (by Lemma \ref{lemm:i!-fet-pres-mot-equiv}), where it is clear since sifted colimits are computed sectionwise \cite[Proposition 5.5.8.10]{lurie-htt}.

(3) By construction, $\E\scr F \in \Spc(\scr S)[\scr F]$ and consequently for $X \in \Spc(\scr S)$ we have $FX \wedge F(\E \scr F) \in \Spc^\fet(\scr S)[\scr F]$.
Since $\Spc^\fet(\scr S)$ is generated by objects of the form $FX$, we deduce that $F(\E \scr F) \wedge \Spc^\fet(\scr S) \subset \Spc^\fet(\scr S)[\scr F]$, this latter subcategory being closed under colimits.
It hence suffices to exhibit an equivalence \[ i^* (\ph \wedge F(\E \scr F)) \to i^* i_! i^* \wequi i^*. \]
We take the transformation induced by $\E \scr F \to *$.
Since all the relevant functors preserve motivic equivalences and sifted colimits, to show that this is an equivalence it suffices to treat $X = F(X_0) \in \PSh_\Sigma(\Cor^\fet(\scr S))$ for $X_0 \in \PSh_\Sigma(\SmQA_{\scr S})$ (or in fact $X_0 \in \SmQA_{\scr S}$); i.e. we need to show that $F(\E\scr F \times X_0) \to F(X_0)$ induces an equivalence on sections over $U \in \SmQA_{\scr S}[\scr F]$.
We use that we have (see Remark \ref{rmk:compute-F}) \[ F(X_0)(U) \wequi \colim_{U' \in \FEt_U} X_0(U'), \] and similarly for $F(\E \scr F \times X_0)$.
Since $U \in \scr F$ also $U' \in \scr F$, and consequently $(X_0 \times \E\scr F)(U') \to X_0(U')$ is an equivalence.
The result follows.
\end{proof}

\subsection{Quotients} \label{subsec:spcfet-quotients}
Let $p: \scr X \to \scr S$ be a proper étale morphism.
The functor $p_!: \SmQA_{\scr X}[\scr F_\free/\scr S] \to \SmQA_{\scr S}$ from Lemma \ref{lemm:presheaf-quotients} is basically the identity, and so preserves pullbacks and finite étale morphisms.
Thus there is an induced functor $p_!: \Cor^\fet(\SmQA_{\scr X}[\scr F_\free/\scr S]) \to \Cor^\fet(\SmQA_{\scr S})$, and hence an induced functor on $\PSh_\Sigma(\ph)$.

\begin{lemma} \label{lemm:p!-cor-exists}
The functor $p_!: \PSh_\Sigma(\Cor^\fet(\SmQA_{\scr X}[\scr F_\free/\scr S])) \to \PSh_\Sigma(\Cor^\fet(\SmQA_{\scr S}))$ is left adjoint to $p^*(\ph) \wedge F\E(\scr F_\free/\scr S)$ and preserves motivic equivalences.
\end{lemma}
\begin{proof}
The adjunction \[ p_\sharp: \Stk_{\scr X} \adj \Stk_{\scr S}: p^* \] is preserved by passing to spans \cite[Corollary C.21]{bachmann-norms}, whence $p^*: \Cor^\fet(\SmQA_{\scr S}) \to \Cor^\fet(\SmQA_{\scr X})$ has a partially defined left adjoint given on $\Cor^\fet(\SmQA_{\scr X}[\scr F_\free/\scr S])$ by $p_\sharp \wequi p_!$.
This implies the adjointness claim.
Preservation of motivic equivalences is clear by construction.
\end{proof}

\begin{corollary} \label{cor:stable-quotient-fet}
Let $p: \scr X \to \scr S$ be a proper étale morphism.
The functor \[ p^*(\ph) \wedge F\E\scr F_\free/\scr S: \Spc^\fet(\scr S) \to \Spc^\fet(\scr X)[\scr F_\free/\scr S] \] has a left adjoint $p_!$ compatible with arbitrary base change and $F: \Spc(\scr X)[\scr F_\free/\scr S] \to \Spc^\fet(\scr X)[\scr F_\free/\scr S]$.
\end{corollary}
\begin{proof}
Immediate from Lemma \ref{lemm:p!-cor-exists}.
\end{proof}

\subsection{Ambidexterity} \label{subsec:spcfet-ambidex}
\begin{proposition} \label{prop:ambidex-spcfet}
Let $p: \scr X \to \scr S$ be a tame proper étale morphism.
Then all objects of \[ \Spc^\fet(\scr X)[\scr F_\free/\scr S] \subset \Spc^\fet(\scr X) \] are $p$-ambidextrous.
\NB{If $p$ is finite étale, then $\Spc^\fet(\scr X)[\scr F_\free/\scr S] = \Spc^\fet(\scr X)$.}
\end{proposition}
\begin{proof}
By induction on the truncatedness of $p$ we have a transformation $W_p: p_! \to p_*$, which we must prove is an equivalence.
Since both functors preserve colimits and motivic equivalences (see Corollary \ref{cor:p*-fet} and Lemma \ref{lemm:p!-cor-exists}), we may check this on generators of $\PSh_\Sigma(\dots)$ (see also Lemma \ref{lemm:norm-nat}).
Thus let $\scr T \in \SmQA_\scr{X}[\scr F_\free/\scr S]$ and $U \in \SmQA_\scr{S}$.
We seek to prove that a certain map of groupoids $F(p_!\scr T)(U) \to p_*(F\scr T)(U)$ is an equivalence.
The source is the groupoid of spans of the form \[ U \xleftarrow{\fet} Z' \to p_! \scr T \] and the target is the groupoid of spans of the form \[ p^*U \xleftarrow{\fet} Z \to \scr T. \]
Unpacking the definitions\tdn{...}, the map from the source to the target is given by applying $p^*$ and then composing with the span $p^*p_! \scr T \leftarrow \scr T \rightarrow \scr T$.
Thus it is given by \[ (U \leftarrow Z' \to p_! \scr T) \mapsto (p^*U \leftarrow p^*Z' \times_{p^*p_! \scr T} \scr T \rightarrow \scr T). \]
We can write down an inverse of this map as follows.
Given $p^*U \leftarrow Z \to \scr T$, $Z$ must be free (admitting a map to $\scr T$), and so $p_!Z$ exists.
We can thus define \[ (p^*U \leftarrow Z \to \scr T) \mapsto (U \leftarrow p_!Z \to p_!\scr T), \] with the new span being induced by adjunction and functoriality.
This being maps of $1$-groupoids, it is easy to verify they are indeed inverses.
In fact the functor $p_!$ is essentially the identity, and then the key point is that \[ p^*Z' \times_{p^*p_!\scr T} \scr T = (Z \times_\scr{S} \scr X) \times_{\scr T \times_\scr{S} \scr X} \scr T \wequi Z', \] as one can see e.g. from the following cartesian squares
\begin{equation*}
\begin{CD}
Z @>>> Z \times_{\scr S} \scr X @>>> Z \\
@VVV                      @VVV         @VVV \\
\scr T @>>> \scr T \times_{\scr S} X @>>> \scr T.
\end{CD}
\end{equation*}
\end{proof}

\subsection{Geometric fixed points} \label{subsec:spcfet-geom-fixed}
\begin{proposition} \label{prop:geom-fixed-fet}
Let $p: \scr X \to \scr S$ be a tame proper étale gerbe.
Then the two functors \[ p_\otimes \text{ and } p_*(F_*\widetilde{\E(\scr F_\prop/\scr S)} \wedge \ph): \Spc^\fet(\scr X) \to \Spc^\fet(\scr S) \] are naturally equivalent.
In particular the square
\begin{equation*}
\begin{CD}
\Spc(\scr X)_* \wedge \widetilde{\E(\scr F_\prop/\scr S)} @>{F_*}>> \Spc^\fet(\scr X) \\
@V{p_*}VV                                                             @V{p_*}VV       \\
\Spc(\scr S)_* @>F_*>> \Spc^\fet(\scr S)
\end{CD}
\end{equation*}
commutes.
\end{proposition}
\begin{proof}
Given a functor $F: \SmQA_{\scr X} \to \SmQA_{\scr Y}$ preserving finite étale morphisms and pullbacks therealong, there is an induced functor $\Cor^\fet(F): \Cor^\fet(\SmQA_{\scr X}) \to \Cor^\fet(\SmQA_\scr{Y})$.
This construction is even $2$-functorial \cite[Proposition C.20]{bachmann-norms}.
The adjunction \[ p^*: \SmQA_\scr{S} \adj \SmQA_\scr{X}: p_* \] thus induces a transformation \[ \Cor^\fet(p^*) \circ \Cor^\fet(p_*) \wequi \Cor^\fet(p^* \circ p_*) \to \Cor^\fet(\id) \wequi \id. \]
Using that $\Cor^\fet(p^*) \wequi p^*$ and $\Cor^\fet(p_*) \wequi p_\otimes$, we can adjoin this over to a natural transformation \[ p_\otimes \to p_* \in \Fun(\PSh_\Sigma(\Cor^\fet(\SmQA_{\scr X})), \PSh_\Sigma(\Cor^\fet(\SmQA_\scr{S}))). \]
Since $p_*, p_\otimes$ preserve (sifted) colimits and motivic equivalence (see Corollary \ref{cor:p*-fet} and Lemma \ref{lemm:pointed-fixed-points}), it will suffice to prove that for $X \in \SmQA_\scr{X}$ the composite \[ F p_\otimes X \wequi p_\otimes FX \to p_* FX \to p_*(F_*\widetilde{\E(\scr F_\prop/\scr S)} \wedge FX) \in \PSh_\Sigma(\Cor^\fet(\SmQA_\scr{S})) \] is an equivalence.

The formula $X_+ \wedge (U_+/V_+) \wequi (X \times U)_+ / (X \times V)_+$ yields a pushout diagram
\begin{equation*}
\begin{CD}
F(X \times \E \scr F_\prop) @>>> F(X) \\
@VVV          @VVV                 \\
*    @>>> F_*(X_+ \wedge \widetilde{\E \scr F_\prop}),
\end{CD}
\end{equation*}
which is preserved by $p_*$.
By the pasting law \cite[Lemma 4.4.2.1]{lurie-htt}, it will thus suffice to show that the following is a pushout in $\PSh_\Sigma(\Cor^\fet(\SmQA_{\scr S}))$
\begin{equation*}
\begin{CD}
*              @>>>              F(p_*X) \\
@VVV                           @VVV      \\
p_*F(X \times \E \scr F_\prop) @>>> p_*F(X).
\end{CD}
\end{equation*}
Let $Y \in \SmQA_{\scr S}$ and $\scr F \in \PSh_\Sigma(\Cor^\fet(\SmQA_{\scr S}))$.
Since $\PSh_\Sigma(\Cor^\fet(\SmQA_{\scr S}))$ is semiadditive, $\scr F(Y)$ is canonically a commutative monoid, and we obtain a functor \[ \Gamma_Y: \PSh_\Sigma(\Cor^\fet(\SmQA_{\scr S})) \to \CMon(\Spc). \]
The functor $\Gamma_Y$ preserves sifted colimits and finite products, and hence all colimits.
It thus suffices to show that \[ \Gamma_Y(p_*F(X \times \E \scr F_\prop)) \oplus \Gamma_Y(F(p_* X)) \wequi \Gamma_Y(p_* FX). \]
We know that $\Gamma_Y(p_* FX)$ is the commutative monoid of spans $p^* Y \leftarrow T \to X$.
Since any finite étale $p^*Y$-scheme splits into a component with full isotropy (i.e. trivial action) and a component with isotropy contained in a proper subgroup (we may check this fppf locally, where it is clear), we find that \[ \Gamma_Y(p_* FX) \wequi \Gamma_Y(p_* FX)^\prop \oplus \Gamma_Y(p_* FX)^0, \] where the first space consists of spans where $T$ has isotropy in $\scr F_\prop$ and the second space consists of spans where $T$ has full isotropy.
The same argument applies to $\Gamma_Y(p_*F(X \times \E \scr F_\prop))$, and by construction \[ \Gamma_Y(p_*F(X \times \E \scr F_\prop))^0 \wequi * \quad\text{and}\quad \Gamma_Y(p_*F(X \times \E \scr F_\prop))^\prop \wequi \Gamma_Y(p_* FX)^\prop. \]
It thus remains to check that the groupoids of spans \[ p^*Y \leftarrow T \to X \quad\text{and}\quad Y \leftarrow T' \to p_*X \] are equivalent, where $T$ is required to have trivial action.
Since $p^*: \SmQA_{\scr S} \to \SmQA_{\scr X}$ is fully faithful onto the subcategory of stacks with trivial action, this is clear.

\end{proof}

\section{DM-stacks} \label{sec:DM}
\begin{definition}
By a \emph{tame DM-stack} we mean a stack with diagonal which is unramified and finite, and which is tame in the sense of \cite[Theorem 3.2]{abramovich2008tame}.
\end{definition}
Such a stack is nicely scalloped and admits a quasi-affine Nisnevich covering by nicely quasi-fundamental stacks; in fact by stacks affine over $B \sslash G$, where $G$ is a tame finite étale groups scheme over the affine scheme $B$ (this follows from \cite[Lemma 2.2.3]{abramovich2002compactifying} using that our stacks have affine diagonal).
If it is of finite type over an affine scheme, then in the previous sentence one may replace ``nicely'' by ``linearly''.
Note also that if $\scr B$ is a tame DM-stack and $p: \scr S \to \scr B$ is tame proper étale, then $\scr S$ is also a tame DM-stack.

\begin{definition} \label{def:adams-hyp}
We say that the \emph{Adams hypothesis} holds if given any proper étale morphism $p: \scr S \to \scr B$ of tame DM-stacks, all objects of $\SH(\scr S)[\scr F_\free/\scr B]$ are $p$-ambidextrous.
\end{definition}

\begin{definition} \label{def:loc-const-stack}
We call a tame DM-stack $\scr X$ \emph{locally constant} if it admits a quasi-affine Nisnevich cover by quasi-fundamental stacks, for various finite \emph{constant} groups.
\end{definition}

\subsection{Generation by trivial spheres} \label{subsec:gen-triv}
\begin{proposition} \label{prop:gens}
Let $\scr S$ be a tame DM-stack.
The category $\SH(\scr S)$ is generated by objects of the form $\Sigma^\infty_+ \scr X \wedge \Gmp{n}$, for $\scr X \in \SmQA_{\scr S}$ and $n \in \Z$.
If $\scr S$ is a gerbe, $\scr X$ can also be assumed to be.
\end{proposition}
\begin{proof}
We shall prove the first statement; the second then follows via Corollary \ref{cor:gerbes-gen}.
The problem being local on $\scr S$, we may assume that $\scr S$ is quasi-affine over $B \sslash G$, and we must prove that $T^{-\rho}$ is in the subcategory $\scr C$ with the claimed generators.
For this it suffices to treat the case $\scr S = B \sslash G$.
Consequently $p: \scr S \to B$ is a tame proper étale gerbe.
Using filtration by the weakly adjacent families $\scr F_i/B$, it is enough to show that $\E(\scr F_i/B, \scr F_{i+1}/B) \wedge T^{-\rho} \in \scr C$.
Applying Corollary \ref{cor:stable-families-reduction} it is enough to show that $a_\sharp b^*(\SH(\Orb_{\scr F})[\scr F_\free/B]) \in \scr C$.
This follows from Proposition \ref{prop:stab-free-prop} applied to the tame proper étale morphism $\Orb_{\scr F} \to B$, since $B$ is a scheme (whence $\SH(B) \wequi \Spc(B)[T^{-1}]$).
\end{proof}

Let $\scr S$ be a tame DM-stack and $q: \scr X \to \scr S \in \SmQA_{\scr S}$ be a gerbe with quotient $p: \scr X \to X$.
We denote by $(\ph)^{\scr X}$ the composite functor \[ \SH(\scr S) \xrightarrow{q^*} \SH(\scr X) \xrightarrow{p_*} \SH(X). \]
\begin{corollary} \label{cor:fixed-points-cons}
Let $\scr S$ be a tame DM-gerbe.
The collection of functors $(\ph)^{\scr X}$, where $\scr X$ runs over gerbes in $\SmQA_{\scr S}$, is jointly conservative.
\end{corollary}
\begin{proof}
If $E \in \SH(\scr S)$ with $E^{\scr X} = 0$, then\NB{use that $f_\sharp$ preserves $\Gmp{n}$...} \[ 0 = \map(\Gmp{n}, E^{\scr X}) \wequi \map(\Sigma^\infty_+ \scr X \wedge \Gmp{n}, E). \]
Thus the result follows from Proposition \ref{prop:gens}.
\end{proof}

\subsection{Admissible gerbes} \label{subsec:adm-gerbes}
\begin{definition} \label{def:admissible}
We call a gerbe $\scr S \to \scr B$ \emph{admissible} if there exists a cartesian square
\begin{equation*}
\begin{CD}
\scr S @>>> B \sslash G \\
@VVV           @VVV     \\
\scr B @>>> B \sslash (G/N)
\end{CD}
\end{equation*}
where $B$ is an affine scheme, $G$ is a finite group linearly reductive over $B$, $N$ is a normal subgroup of $G$ and $\scr B \to B \sslash (G/N)$ is quasi-affine.
We call a proper étale morphism admissible if the associated gerbe is.
\end{definition}

\begin{lemma} \label{lemm:orb-pres-admissible}
If $\scr S \to \scr B$ is admissible, then also $\widetilde\Orb_p \to \Orb_p$ and $\Orb_p \to \repr_{\scr B}(\Orb_p)$ are admissible.
\end{lemma}
\begin{proof}
Since admissible gerbes are stable under (quasi-affine) base change by definition, and orbit stacks are stable under base change by Remark \ref{rmk:orbit-stack-bc}, it suffices to treat the case $p: B \sslash G \to B \sslash (G/N)$.
We claim that then there are equivalences \[ \widetilde\Orb_p \wequi \coprod_{(H) \subset N} B \sslash NH \quad\text{and}\quad \Orb_p \wequi \coprod_{(H) \subset N} B \sslash WH, \] where the disjoint unions are over $G$-conjugacy classes of subgroups of $H$.
This implies what we want.
We give the proof for $\Orb_p$, the case of $\widetilde\Orb_p$ is similar.
Suppose given $T \in \Stk_{B \sslash (G/H)}$ and a morphism $\alpha: T \to B \sslash WH$, where $H$ is a subgroup of $N$.
Pulling back $B \sslash NH \to B \sslash WH$ along $\alpha$ we obtain a stack $X_\alpha$ with a morphism to $T$ and also to $B \sslash G$, whence to $p^* T$.
We claim that $X_\alpha \in \FEt_{p^*T}$, that this is an orbit, and that the induced map $\coprod_{(H) \subset N} B \sslash WH \to \Orb_p$ is an equivalence.
All of these claims are local, so we may pull back everything along $B \to B \sslash (G/N)$.
We get $B \sslash G \times_{B \sslash (G/N)} B = B \sslash N$.
Similarly $B \sslash WH \times_{B(G/N)} B$ is a stack with automorphisms given by $ker(WH \to G/N) = (NH \cap N)/H = W_N H$, but more than one object (locally, up to isomorphism).
Consider the action of $G/N$ on the set of $N$-conjugacy classes of subgroups of $G$.
The stabilizer of (the $N$-conjugacy class of) $H$ consists of those $gN \in G/N$ such that there exists $n \in N$ with $H^{gn} = H$, i.e. $gn \in NH$.
Equivalently we must have $gN \cap NH \ne \emptyset$; in other words the stabilizer is precisely the image of $NH \to G/N$.
This allows us to identify \[ B \sslash WH \times_{B \sslash (G/N)} B \wequi \coprod_{H'} B \sslash W_N H'; \] here the disjoint union is over $N$-conjugacy classes of subgroups of $N$ which are $G$-conjugate to $H$.
It follows that \[ B \times_{B \sslash (G/H)} \coprod_{(H) \subset N} B \sslash WH \wequi \coprod_{(H)_N \subset N} B \sslash W_N H, \] where the right hand side means disjoint union over $N$-conjugacy classes of subgroups of $N$.
We know that this is $\Orb_{p \times_{B \sslash (G/H)} B}$ by Lemma \ref{lemm:orbits-BG}, and the isomorphism is compatible with the construction we gave above.
This concludes the proof.
\end{proof}

\begin{lemma} \label{lemm:gerbes-admissible-auto}
Let $\scr X$ be a gerbe, admitting a quasi-affine morphism locally of finite type to $B \sslash G$ where $G$ is a finite constant group, linearly reductive over the affine scheme $B$.
Then up to a disjoint union decomposition, $\scr X \to \repr(\scr X)$ is an admissible gerbe.
\end{lemma}
\begin{proof}
It will suffice to show that $\scr X$ is a disjoint union of stacks of the form $X \sslash H$, where $H \subset G$ is a subgroup, $X$ is an $H$-scheme, and there is a normal subgroup $N \subset H$ such that $N$ acts trivially on $X$ and $H/N$ acts freely.
We have $\scr X = X_0 \sslash G$, for some (relative) quasi-affine $G$-scheme $X_0$ over $B$.
The result now follows from Lemma \ref{lemm:flat-inertia-form}.
\end{proof}

\subsection{Orbit stack induction} \label{subsec:orbit-stack-ind}
Let $p: \scr S \to \scr B$ be a proper étale morphism of tame DM-stacks.
Suppose we want to prove some result regarding $p_*(E)$, for every object $E \in \SH(\scr S)$ (or some related category with similar properties).
Filtering by isotropy, it will be enough to treat $\E(\scr F_i/\scr B, \scr F_{i+1}/\scr B) \wedge E$ for $i \in \Z$.
Consider the associated orbit stack diagram
\begin{equation*}
\begin{CD}
\widetilde\Orb_{\scr F} @>b>> \Orb_{\scr F} \\
@V{a}VV                         @V{c}VV     \\
\scr S   @>p>> \scr B,
\end{CD}
\end{equation*}
where $\scr F = \scr F_{i+1} \setminus \scr F_i$.
Applying Corollary \ref{cor:stable-families-reduction} and using the projection formula we get \[ \E(\scr F_1, \scr F_2) \wedge E \wequi  a_\sharp (a^*(\E(\scr F_1, \scr F_2)) \wedge b^{*}(\E(\scr F_\free/\scr S)_+ \wedge b_\otimes a^{*} E)). \]
By ambidexterity \cite[Theorem 7.10 and 7.1(1)]{khan2021generalized} for finite étale morphisms, $a_\sharp \wequi a_*$, and so we find \[ p_*(\E(\scr F_1, \scr F_2) \wedge E) \wequi c_* b_*(E'), \quad\text{where}\quad E' = a^*(\E(\scr F_1, \scr F_2)) \wedge b^{*}(\E(\scr F_\free/\scr S)_+ \wedge b_\otimes a^{*} E). \]
It is thus enough to prove the desired assertion for $b_*$ and $c_*$.
This can often be done by induction, as follows.
Fppf locally on $\scr B' := \repr_{\scr B} \scr S$, the morphism $p': \scr S \to \scr B'$ is of the form $B \sslash G \to B$ for finite groups $G$.
Since all our stacks are quasi-compact, the size of these groups is bounded, say by $N$.
We call this the \emph{size of the isotropy of $p$}.
Now observe that the size of the isotropy of $b$ is at most $i+1$, and the size of the isotropy of $c$ is at most $N/(i+1)$.
Since also $\scr F_0/\scr B = \emptyset$ and $\scr F_n/\scr B = \scr F_\all/\scr B$, all cases but $i \in \{0, N-1\}$ hold by induction.
Now observe that \[ \E(\scr F_0/\scr B, \scr F_1/\scr B) \wequi \E(\scr F_\free/\scr B) \] and (since $\scr F_\prop/\scr B \subset \scr F_{N-1}/\scr B$) \[ \E(\scr F_{N-1}/\scr B, \scr F_{N}/\scr B) \wequi \E(\scr F_{N-1}/\scr B, \scr F_{N}/\scr B) \wedge \widetilde{\E\scr F_\prop/\scr S}. \]
Assuming that free objects are $p'$-ambidextrous we thus get \[ p_*(\E(\scr F_0/\scr B, \scr F_1/\scr B) \wedge E) \wequi p''_*p'_!(\E(\scr F_\free/\scr B)), \] and without further assumptions we have (using Corollary \ref{cor:geom-fixed-points}) \[ p_*(\E(\scr F_{N-1}/\scr B, \scr F_{N}/\scr B) \wedge E) \wequi p''_*p'_\otimes(\E(\scr F_{N-1}/\scr B, \scr F_{N}/\scr B) \wedge E), \] where $p': \scr S \to \scr B'$ is the associated gerbe and $p'': \scr B' \to \scr B$ is finite étale.
These objects are often easier to understand than the general case.

Note also that the ambidexterity assumption follows from the Adams hypothesis, and it also holds if $p$ is admissible by \cite[\S6.1, Theorem 6.33]{gepner-heller}.
Note finally that if $p$ was admissible so are $b, c$, by Lemma \ref{lemm:orb-pres-admissible}.

\subsection{The homotopy $t$-structure} \label{subsec:htpy-t}
For a tame DM-stack $\scr S$, let $\SH(\scr S)_{\ge 0}$ be the subcategory generated under colimits and extensions by objects of the form $\Sigma^\infty_+ \scr X \wedge \Gmp{n}$, for $\scr X \in \Spc(\scr S)$ and $n \in \Z$.
This is the non-negative part of a $t$-structure \cite[Proposition 1.4.4.11]{lurie-ha}, which we call the \emph{homotopy $t$-structure}.

\begin{proposition} \label{prop:fixed-t-exact}
Consider the homotopy $t$-structure on $\SH(\scr S)$, for various tame DM-stacks $\scr S$.
\begin{enumerate}
\item The functors $f_\sharp$ (for $f$ smooth representable) are right-$t$-exact.
\item The functors $f^*$ are right-$t$-exact (for $f$ not necessarily representable), and $t$-exact for $f$ smooth representable.
\item Let $p: \scr S \to \scr B$ be a proper étale morphism.
  Suppose that either $p$ is admissible (Definition \ref{def:admissible}) or the Adams hypothesis holds (Definition \ref{def:adams-hyp}).
  Then $p_*$ is $t$-exact.
\end{enumerate}
\end{proposition}
\begin{proof}
Throughout, we use the model of $\Spc(\scr S)$ using smooth \emph{representable} stacks over $\scr S$ instead of quasi-affine ones (see Proposition \ref{prop:nice-scallop}).

We first prove the right-$t$-exactness.
All of the functors preserve colimits (and hence extensions), so it suffices to show that they map our generating family into the $\ge 0$ part.
This is obvious for $f^*$ and $f_\sharp$.

For $p_*$, we use orbit stack induction (see \S\ref{subsec:orbit-stack-ind}).
It thus suffices to show that the following functors are right-$t$-exact: $p_\otimes$ for $p$ a proper étale gerbe, $p_!$ applied to free objects, and $p_*$ when $p$ is finite étale.
Each of these is clear (note in particular that in the first case $p_\otimes(\Gm) \wequi \Gm$ since $p_\otimes$ is fixed points and the action on $\Gm$ is trivial, and that in the last case, $p_* \wequi p_\sharp$).

If $f$ is smooth representable, then $f^*$ is left-$t$-exact because its left adjoint $f_\sharp$ is right $t$-exact (see e.g. \cite[Lemma 5]{bachmann-hurewicz} for a proof of the fact that the right adjoint of a right-$t$-exact functor is left-$t$-exact).
Similarly $p_*$ is left $t$-exact because its left adjoint $p^*$ is right $t$-exact.
\end{proof}

\begin{proposition} \label{prop:homotopy-t-nondegenerate}
Let $\scr S$ be a tame DM-stack.
\begin{enumerate}
\item The homotopy $t$-structure is right non-degenerate: \[ \cap_n \SH(\scr S)_{\le n} = \{0\}. \]
\item Let $\scr S$ be finite dimensional, and suppose that either (a) $\scr S$ is locally constant (Definition \ref{def:loc-const-stack}) and locally of finite type over a scheme, or (b) the Adams hypothesis holds (Definition \ref{def:adams-hyp}).
  Then the homotopy $t$-structure on $\SH(\scr S)$ is left non-degenerate: \[ \cap_n \SH(\scr S)_{\ge n} = \{0\}. \]
\end{enumerate}
\end{proposition}
\begin{proof}
(1) is immediate from Proposition \ref{prop:gens}.

Now we prove (2).
The problem is local on $\scr S$, so in case (a) we may assume that $\scr S$ is quasi-affine of finite type over $B \sslash G$, for some finite constant (linearly reductive) group $G$.
By \cite[Theorem (B.2)]{rydh2011etale} we can find a locally closed cover $\scr S = \bigcup_i \scr S$, where each $\scr S_i$ is a gerbe.
Restriction to the $\scr S_i$ is right-$t$-exact by Proposition \ref{prop:fixed-t-exact}(2) and conservative by \cite[Proposition 5.8]{khan2021generalized}.
We may thus assume that $\scr S$ is a gerbe.
Let $p: \scr X \to \scr S \in \SmQA_{\scr S}$ where $p$ is a gerbe, and $p': \scr X \to X$ the quotient map.
In case (a), $p'$ is an admissible gerbe by Lemma \ref{lemm:gerbes-admissible-auto}.
The functors of the form $p'_*p^*$ are conservative by Corollary \ref{cor:fixed-points-cons} and right-$t$-exact by Proposition \ref{prop:fixed-t-exact}(3).
We may thus assume that $\scr S$ is an algebraic space (qcqs as always).
This admits a Nisnevich cover by schemes, so we may assume that $S$ is a finite-dimensional scheme.
By \cite[Proposition B.3]{bachmann-norms} we reduce to the case where $S$ is the spectrum of a field, in which case the claim follows for example from \cite[Corollary 2.4]{hoyois-algebraic-cobordism}.
\end{proof}

\subsection{Main result} \label{subsec:main}
\begin{lemma} \label{lemm:Spcfet-invert-rho}
Let $\scr S$ be quasi-affine over $B\sslash G$, where $B$ is affine and $G$ is tame finite étale over $B$.
Assume that either $G$ is constant or the Adams hypothesis holds.
Then the canonical functor \[ \alpha: \SH(\scr S) \to \Spc^\fet(\scr S)[T^{-\rho}] \] is an equivalence.
\end{lemma}
\begin{proof}
As a first step, we show that we may without loss of generality put some additional finiteness hypotheses on $\scr S$ and $B$.
Let $\scr S_\alpha$ be a cofiltered system of stacks with affine transition maps and limit $\scr S$.
Then $\Spc^\fet(\scr S) \wequi \lim_\alpha \Spc^\fet(\scr S_\alpha)$, by Lemma \ref{lemm:spcfet-cont}.
If each $\scr S_\alpha$ is quasi-affine over $B_0 \sslash G_0$ for some tame finite étale group $G_0$ over an affine scheme $B_0$, then we deduce that also $\Spc^\fet(\scr S)[T^{-\rho}] \wequi \lim_\alpha \Spc^\fet(\scr S_\alpha)[T^{-\rho}]$ (use Lemma \ref{lemm:scalar-extension-cocont}).
Similarly, using Lemma \ref{lemm:cty-spaces}, we also get in this situation that $\SH(\scr S) \wequi \lim_\alpha \SH(\scr S_\alpha)$.
In order to implement our desired reduction, we thus need to produce a cofiltered system of stacks $\scr S_\alpha$ as above, such that additionally each $\scr S_\alpha$ has our desired finiteness properties.

By \cite[Theorem D]{rydh2015noetherian}, we may write $\scr S = \lim_\alpha \scr S_\alpha$, where the transition maps are affine and each $\scr S_\alpha$ is quasi-affine and finitely presented over $B \sslash G$.
Hence we may assume that $\scr S \to B \sslash G$ is of finite presentation.
We can write $B = \lim_\alpha B_\alpha$, where each $B_\alpha$ is affine of finite type over $\Z$.
Since $G \to B$ is finitely presented, we may assume that $G = G_0 \times_{B_0} B$, for some tame finite étale group $G_0$ over $B_0$.
Letting $G_\alpha = G_0 \times_{B_0} B_\alpha$, we get $B \sslash G = \lim_\alpha B G_\alpha$.
By \cite[Propositions B.2 and B.3]{rydh2015noetherian} we find a stack (possibly after increasing the index $0$) $\scr S_0$ quasi-affine of finite presentation over $B \sslash G_0$ such that $\scr S \wequi \scr S_0 \times_{B \sslash G_0} B \sslash G$.
Letting $\scr S_\alpha = \scr S_0 \times_{B \sslash G_0} B \sslash G_\alpha$ we get $\scr S \wequi \lim_\alpha \scr S_\alpha$, where each $\scr S_\alpha$ is quasi-affine of finite presentation over $B \sslash G_\alpha$, and hence over $B \sslash G_0$.

\textbf{Thus from now on, we may (and will) assume that $\scr S$ is finite type over $B$, and $B$ is finite type over $\Z$.}
By construction, $\alpha$ is a colimit preserving, symmetric monoidal functor.
Denote its right adjoint by $\beta$.
Since $\alpha$ preserves a compact generating family, $\beta$ is conservative and preserves colimits.
It consequently suffices to show that the unit $\id \to \beta \alpha$ is an equivalence (indeed then so is $\beta \alpha \beta \xrightarrow{\beta \epsilon} \beta$, by the triangle identities, and whence so is the counit $\epsilon$ itself).
It is enough to check this on generators, and since $\alpha$ is symmetric monoidal we have (see e.g. \cite[Lemma 3.2]{bachmann-topmod}) \[ \beta(\ph \wedge \alpha(T^{n\rho})) \wequi \beta(\ph) \wedge T^{n\rho}. \]
We thus need only verify that $\eta: \Sigma^\infty X \to \beta \alpha \Sigma^\infty X$ is an equivalence, for all $X \in \Spc(\scr S)_*$.

The functor $U_*: \Spc^\fet(\scr S) \to \Spc(\scr S)_*$ preserves filtered colimits (Lemma \ref{lemm:U-sifted-colim}) and $T^\rho$-loops, and consequently commutes with spectrification of $T^\rho$-prespectra \cite[p. 7]{hoyois2016cdh} \cite[\S6.1]{hoyois-equivariant}.
It follows that the extension of $U_*$ to prespectra preserves stable equivalences, and hence $\beta$ may be computed levelwise\NB{direct reference? make this a separate lemma?}.
We thus find that \[ \beta \alpha \Sigma^\infty X \wequi \colim_n[T^{-n\rho} \wedge \Sigma^\infty U_*F_*(X \wedge T^{n\rho})]. \]
On the other hand \[ \Sigma^\infty X \wequi \colim_n[T^{-n\rho} \wedge \Sigma^\infty(T^{n\rho} \wedge X)]. \]
Using Corollary \ref{cor:F*-stable} (and the fact that $\D_1 \wequi \id$) we hence find that the cofiber of $\eta$ is given by \[ \cof(\eta) \wequi \colim_n \left[T^{-n\rho} \wedge \bigvee_{m \ge 2} \D_m(T^{n\rho} \wedge \Sigma^\infty X) \right]; \] we need to show that this is zero.
By non-degeneracy of the homotopy $t$-structure (Proposition \ref{prop:homotopy-t-nondegenerate}(2)), for this it suffices to show that \[ T^{-n\rho} \wedge \D_m(T^{n\rho} \wedge X) \in \SH(\scr S)_{\ge (m-1)n}. \]

\NB{Details? Make some of this separate lemmas?}
Using the definition of $\D_m$ as a colimit (see \eqref{eq:D-lambda}) and Lemma \ref{prop:fixed-t-exact}(1), for this it is enough to show that if $p: \scr S' \to \scr S$ is finite étale of degree $m$, with $\scr S$ quasi-affine over $B \sslash G$, and $X \in \Spc(\scr S')_*$, then \[ T^{-n\rho} \wedge p_\otimes \Sigma^\infty(T^{n\rho} \wedge X) \in \SH(\scr S)_{\ge (m-1)n}. \]
Since $p_\otimes$ is symmetric monoidal and commutes with $\Sigma^\infty$, we may assume that $X = S^0$.
If $\scr S = S \sslash G$, write $f: \bar{\scr S} \to \scr S$ for the finite étale morphism $(G \times S) \sslash G \to S \sslash G$, and similarly for $\scr S'$.
Note that $\bar{\scr S}$ and $\bar{\scr S'}$ are schemes.
Consider the commutative diagram
\begin{equation*}
\begin{CD}
\bar{\scr S'} @>{f'}>> \scr S' \\
@V{p'}VV            @V{p}VV  \\
\bar{\scr S} @>f>> \scr S.
\end{CD}
\end{equation*}
We have $T^{\rho}_{\scr S} = f_\otimes(T_{\bar{\scr S}})$ (this follows from Lemma \ref{lemm:norm-Trho}), and hence \[ T^{-n\rho} \wedge p_\otimes T^{n\rho} \wequi f_\otimes(p'_\otimes(T) \wedge T^{-1})^{\wedge n}. \]
It is thus enough to show that $f_\otimes(p'_\otimes(T) \wedge T^{-1}) \in \SH(\scr S)_{\ge m-1}$.
Zariski locally on $\scr S$, any vector bundle on $\bar{\scr S}$ of constant rank is trivial \cite[Lemma 1.4.4]{bruns1998cohen}\NB{for quasi-affine $G$-scheme $X$ and $Gx \subset X$, semi-localization $X_{Gx}$ carries $G$-action, and then $X_{Gx} \times G$ is still semilocal}, and hence working Zariski locally on $\scr S$, we may assume that $p'_\otimes(T) \wequi T^m$ (using Lemma \ref{lemm:norm-Trho} again).
We are thus reduced to showing that $T^\rho \in \SH(\scr S)_{\ge 1}$, which holds since $T^\rho \wequi T \wedge T^{\bar\rho}$.
\end{proof}
It follows that in this situation, we obtain a functor \[ \Sigma^\infty_\fet: \Spc^\fet(\scr S) \to \Spc^\fet(\scr S)[T^{-\rho}] \wequi \SH(\scr S). \]

\begin{lemma} \label{lemm:p*-comm}
Let $p: \scr X \to \scr B$ be a proper étale morphism of tame DM-stacks.
Assume that either $p$ is admissible or the Adams hypothesis holds.
\begin{enumerate}
\item The following diagram commutes
\begin{equation*}
\begin{CD}
\Spc^\fet(\scr X) @>{\Sigma^\infty_\fet}>> \SH(\scr X) \\
@V{p_*}VV              @V{p_*}VV   \\
\Spc^\fet(\scr B) @>{\Sigma^\infty_\fet}>> \SH(\scr B).
\end{CD}
\end{equation*}
\item After stabilization, the functors $p_*$ are module functors over their targets.
\item The following diagram commutes
\begin{equation*}
\begin{CD}
\Spc^\fet(\scr X)[T^{-1}] @>{\Sigma^\infty_\fet}>> \SH(\scr X) \\
@V{p_*}VV              @V{p_*}VV   \\
\Spc^\fet(\scr B)[T^{-1}] @>{\Sigma^\infty_\fet}>> \SH(\scr B).
\end{CD}
\end{equation*}
\end{enumerate}
\end{lemma}
\begin{proof}
(1)
We may replace $\Spc^\fet(\ph)$ by $\SH^{S^1\fet}(\ph)$.
To prove the result, we use orbit stack induction (see \S\ref{subsec:orbit-stack-ind}).
That is, we shall check that the exchange transformation is an equivalence, and we shall do so on generators.
Filtering by isotropy, it suffices to consider generators of the form $F_*(\E(\scr F_1, \scr F_2) \wedge E)$ for $E \in \SHS(\scr X)$ and $\scr F_1 \subset \scr F_2$ weakly adjacent families.
Using Proposition \ref{prop:families-reduction} and the projection formula, we have \[ F_*(\E(\scr F_1, \scr F_2) \wedge E) \wequi a_\sharp(E') \] for some $E' \in SH^{S^1\fet}(\widetilde \Orb_\scr{F})$.
By ambidexterity for finite étale morphisms (Proposition \ref{prop:ambidex-spcfet} and \cite[Theorem 7.10 and 7.1(1)]{khan2021generalized}), and naturality of the ambidexterity isomorphism (Lemma \ref{lemm:norm-nat}), we may replace $a_\sharp$ by $a_*$.
Proceeding with the orbit stack induction as usual, we see that it suffices to treat three cases: (a) $p$ finite étale, (b) $p$ a gerbe, $E \wequi E \wedge \E(\scr F_\free/\scr B)_+$, and (c) $p$ a gerbe, $E \wequi E \wedge \widetilde{\E(\scr F_\prop/\scr B)}$.
Case (a) is clear.
For cases (b) and (c), we claim the stronger result that both $F_1: \Spc(\ph)_* \to \Spc^\fet(\ph)$ and $F_2: \Spc(\ph)_* \to \SH(\ph)$ commute with $p_*$.
Case (b) reduces via naturality of ambidexterity isomorphisms (Lemma \ref{lemm:norm-nat}) to showing that $p_!$ (on free objects) commutes with $F_1, F_2$.
This follows from Corollaries \ref{cor:stable-quotient-fet} and \ref{cor:quotients-stable}.
Case (c) follows from Proposition \ref{prop:geom-fixed-fet} and Corollary \ref{cor:geom-fixed-points}.

(2)
We concentrate on $\SH^{S^1\fet}(\ph)$; the argument for $\SH(\ph)$ is the same.
Being right adjoint to a symmetric monoidal functor, $p_*$ is canonically a lax module functor\tdn{ref}.
We must prove that given $A \in \SH^{S^1\fet}(\scr B), B \in \SH^{S^1\fet}(\scr X)$ we have $p_*(p^*(A) \wedge B) \wequi A \wedge p_*(B)$.
By another orbit stack induction, we again have three cases for $p, B$: we must prove that $p_*$ is a module functor for $p$ finite étale, that $p_!$ is a module functor, and $p_\otimes$ is.
In the first case we have $p_* \wequi p_\sharp$.
Both $p_\sharp$ (for $p$ finite étale) and $p_!$ (on free objects) satisfy the projection formula, as one verifies by checking on generators,
Finally $p_\otimes$ is a module functor since it is symmetric monoidal and $p_\otimes p^* \wequi \id$ (note that $p_\otimes$ is essentially fixed points).

(3)
Immediate from (1) and (2), using $2$-functoriality of scalar extension (Lemma \ref{lemm:scalar-extension-2funct}).
\end{proof}

\begin{theorem} \label{thm:main}
Let $\scr S$ be a tame DM-stack.
Assume that either $\scr S$ is locally constant (Definition \ref{def:loc-const-stack}) or the Adams hypothesis holds (Definition \ref{def:adams-hyp}).
Then the canonical functor (induced by $\Sigma^\infty_\fet$) \[ \Spc^\fet(\scr S)[T^{-1}] \to \SH(\scr S) \] is an equivalence.
\end{theorem}
\begin{proof}
Both sides satisfy quasi-affine Nisnevich descent, by Lemma \ref{lemm:spcfet-descent} and Example \ref{ex:sheaf-stabilization} for $\Spc^\fet(\ph)[T^{-1}]$ and \cite[Theorem 4.5(iii)]{khan2021generalized} for $\SH(\ph)$.
It thus suffices to prove the equivalence for all terms in the Čech nerve of a quasi-affine Nisnevich cover of $\scr S$.
This allows us to assume that $\scr S$ is quasi-affine over $B \sslash G$, where moreover either $G$ is constant or the Adams hypothesis holds.
Since $\Spc^\fet(\scr S)[T^{-\rho}] \wequi \SH(\scr S)$ (Lemma \ref{lemm:Spcfet-invert-rho}), it is enough to show that $T^\rho$ is invertible in $\Spc^\fet(\scr S)[T^{-1}]$.
For this we may assume that $\scr S = B \sslash G$

Since $\Sigma^\infty_\fet$ has dense image by Proposition \ref{prop:gens}, it suffices to prove fully faithfulness.
In other words we must prove that given $E \in \Spc^\fet(\scr S)[T^{-1}]$ the map $E \to \Omega^\infty_\fet \Sigma^\infty_\fet E$ is an equivalence.
By Corollary \ref{cor:fixed-points-cons} it suffices to show that $(E \to \Omega^\infty_\fet \Sigma^\infty_\fet E)^{\scr X}$ is an equivalence, for all gerbes $\scr X$ smooth and quasi-affine over $B \sslash G$.
For a smooth quasi-affine morphism $f$, $\Sigma^\infty_\fet$ commutes with $f^*$ and $f_\sharp$, and so $\Omega^\infty_\fet$ commutes with $f_*$ and $f^*$.
Thus given $E \in \Spc^\fet(\scr S)[T^{-1}]$, we get $f^*\Omega^\infty_\fet\Sigma^\infty_\fet E \wequi \Omega^\infty_\fet\Sigma^\infty_\fet f^* E$.
Applying this to $f = (\scr X \to \scr S)$, we find that it suffices to prove the following: if $\scr X$ is a gerbe, smooth quasi-affine over $B \sslash G$, and with quotient $p: \scr X \to X$, then $p_*(E \to \Omega^\infty_\fet \Sigma^\infty_\fet E)$ is an equivalence.
Either the Adams hypothesis holds, or by Lemma \ref{lemm:gerbes-admissible-auto}, we may assume that $p$ is admissible.
Since $p_*$ commutes with $\Omega^\infty$, Lemma \ref{lemm:p*-comm}(3) now shows that $p_*$ commutes with \[ \Omega^\infty_\fet \Sigma^\infty_\fet: \Spc^\fet(\ph)[T^{-1}] \to \Spc^\fet(\ph)[T^{-1}]. \]
In other words, we have reduced to the case where $\scr X = X$ is an algebraic space.
In this case $T=T^\rho$, and the result reduces to Lemma \ref{lemm:Spcfet-invert-rho}.
\end{proof}

\appendix
\section{Stabilization and localization} \label{app:stab}
For a presentable $\infty$-category $\scr C$ and a small set of morphism $S$ in $\scr C$, we denote by $\scr C[S^{-1}]$ the initial presentable category under $\scr C$ in which all maps in $S$ become equivalences (here we are working in $Pr^L$).
We also denote by $\scr C_* = \scr C_{*/}$ the category of pointed objects and by $\Sp(\scr C)$ the category of spectrum objects \cite[\S1.4.2]{lurie-ha}.
We frequently use the following well-known results.

\begin{lemma} \label{lemm:locn-preserved}
Let $\scr C$ be presentable and $S$ a set of maps in $\scr C$.
\begin{enumerate}
\item $\scr C[S^{-1}]_* \wequi (\scr C_*)[f_+^{-1}, f \in S]$
\item $\Sp(\scr C[S^{-1}]) \wequi \Sp(\scr C)[(\Sigma^{\infty+n}_+ f)^{-1}, f \in S, n \in \Z]$
\end{enumerate}
\end{lemma}
\begin{proof}
One first verifies that the categories on the right hand side are pointed/stable.
By the Yoneda lemma it then suffices to show that the functors represented among pointed/stable categories are the same, which is clear.
\end{proof}

\begin{lemma} \label{lemm:adj-preserved}
Let $L: \scr C \adj \scr D: R$ be an adjunction of presentable categories, in which $R$ preserves colimits.
\begin{enumerate}
\item $L_*: \scr C_* \adj \scr D_*: R$ is an adjunction
\item $\Sp(L): \Sp(\scr C) \adj \Sp(\scr D): \Sp(R)$ is an adjunction
\item Let $S, T$ be sets of maps in $\scr C, \scr D$ such that $L(S)$ consists of $T$-equivalences, and similarly for $R$.
  Then $L: \scr C[S^{-1}] \adj \scr D[T^{-1}]: R$ is an adjunction.
\end{enumerate}
\end{lemma}
\begin{proof}
Since pointing, stabilization and localizations are examples of scalar extension \cite{gepner2016universality}, this is a special case of Lemma \ref{lemm:scalar-extension-2funct} below.
\end{proof}

If $\scr A$ is a presentably symmetric monoidal $\infty$-category and $\scr M$ is an $\scr A$-module, then there is an endofunctor of $\scr A$-modules given by $(\ph) \otimes_{\scr A} \scr M$.
If $\scr B$ is a presentably symmetric monoidal $\infty$-category under $\scr A$, then it is in particular an $\scr A$-module, and the functor $(\ph) \otimes_{\scr A} \scr B$ lifts to $\scr B$-modules.
We call either of these three functors \emph{scalar extension}.
\begin{lemma} \label{lemm:scalar-extension-cocont}
Scalar extension preserves colimits.
\end{lemma}
\begin{proof}
We just need to know that tensor products in $Pr^L$ preserve colimits in each variable separately, for which see \cite[Remark 4.8.1.8]{lurie-ha}.
\end{proof}
\begin{lemma} \label{lemm:scalar-extension-2funct}
Scalar extension is 2-functorial.
In particular, given a natural transformation \textbf{of $\scr A$-module functors} $F \to G$, there is an induced natural transformation $F \otimes_{\scr A} \scr M \to G \otimes_{\scr A} \scr M$.
\end{lemma}
\begin{proof}
One may promote an $\infty$-category to an $(\infty,2)$-category by constructing an appropriate $\Cat_\infty$-module structure (see \cite[\S7]{GH-enriched} and \cite[Theorem 7.5]{Haugseng-enriched}).
The canonical symmetric monoidal functor $\Cat_\infty \to Pr^L \to \scr A\Mod$\NB{size...?} exhibits $\scr A\Mod$ as an $(\infty,2)$-category.
Since scalar extension is an $\scr A$-module functor, it is a $\Cat_\infty$-module functor, and so a functor of $(\infty,2)$-categories.
\end{proof}

Note that given a $\scr C$-module functor $F: \scr M \to \scr N$ with a left adjoint $G$, then $G$ is automatically an oplax $\scr C$-module functor.
Asking that $G$ be a $\scr C$-module functor is thus a property, not data.
\begin{lemma}
Let $\scr C \to \scr D$ be presentably symmetric monoidal and $F$ an $I$-diagram of $\scr C$-modules in $Pr^L$.
If each of the transition functors in $F$ admits a left adjoint $\scr C$-module functor then \[ (\lim_I F)_{\otimes \scr C} \scr D \wequi \lim_I (F \otimes_{\scr C} \scr D). \]
\end{lemma}
\begin{proof}
We may form the left adjoint diagram $F^L$ on $I^\op$, by assumption it is also valued in $\scr C$-modules in $Pr^L$.
Then \[ \colim_{I^\op}(F^L \otimes_{\scr C} \scr D) \wequi (\colim_{I^\op} F^L) \otimes_{\scr C} \scr D \] by Lemma \ref{lemm:scalar-extension-cocont}.
The result follows since colimits in $Pr^L$ are computed as limits on the right adjoints.
\end{proof}

\begin{example} \label{ex:sheaf-stabilization}
If $\scr M$ is a sheaf of $\scr C$-modules (on some site) such that all the pullbacks admit left adjoints which are $\scr C$-module functors (i.e. satisfy the projection formula), then $\scr M \otimes_{\scr C} \scr D$ is also a sheaf.
\end{example}

\bibliographystyle{alpha}
\bibliography{bibliography}

\end{document}